\pdfoutput=1
\documentclass[a4paper,11pt]{article}
\usepackage{amsmath,amssymb,amsthm,graphicx,inputenc}
\usepackage{hyperref,url,enumitem,bm,esint,color}
\usepackage[margin=30mm]{geometry}
\usepackage{placeins}

\usepackage{tikz}

\definecolor{azzurro}{rgb}{0.13, 0.67, 0.8}
\definecolor{rosso}{rgb}{0.85,0,0}

\def\Accorpa #1#2 #3 {\gdef #1{\eqref{#2}--\eqref{#3}}%
  \wlog{}\wlog{\string #1 -> #2 - #3}\wlog{}}

\newcommand{\tr}{\operatorname{tr}}

\renewcommand{\hat}[1]{\widehat{#1}}
\newcommand{\eps}{\varepsilon}

\newcommand{\pd}{\partial}

\newcommand{\dn}{\partial_{\bnn}}
\def\bnn{{\boldsymbol n}}

\newcommand{\inn}[1]{\langle #1 \rangle}

\def\norma #1{\mathopen \| #1\mathclose \|}
\def\<#1>{\mathopen\langle #1\mathclose\rangle}
\def\iO {\int_\Omega}

\renewcommand{\tilde}{\widetilde}

\def\erre{{\mathbb{R}}}

\theoremstyle{plain}
\newtheorem{thm}{Theorem}[section]
\newtheorem{lem}{Lemma}[section]

\newtheorem{remark}{Remark}[section]

\newtheorem{defn}{Definition}[section]

\numberwithin{equation}{section}

\def\genspazio #1#2#3#4#5{#1^{#2}(#5,#4;#3)}
\def\spazio #1#2#3{\genspazio {#1}{#2}{#3}T0}

\def\L {\spazio L}

\def\Lx #1{L^{#1}(\Omega)}
\def\Hx #1{H^{#1}(\Omega)}

\def\multibold #1{\def\arg{#1}%
  \ifx\arg\pto \let\next\relax
  \else
  \def\next{\expandafter
    \def\csname #1#1\endcsname{{\bf #1}}%
    \multibold}%
  \fi \next}

\def\pto{.}

\def\multimathbb #1{\def\arg{#1}%
  \ifx\arg\pto \let\next\relax
  \else
  \def\next{\expandafter
    \def\csname #1#1#1\endcsname{{\mathbb #1}}%
    \multimathbb}%
  \fi \next}
  
\def\multical #1{\def\arg{#1}%
  \ifx\arg\pto \let\next\relax
  \else
  \def\next{\expandafter
    \def\csname cal#1\endcsname{{\cal #1}}%
    \multical}%
  \fi \next}

\def\multimathop #1 {\def\arg{#1}%
  \ifx\arg\pto \let\next\relax
  \else
  \def\next{\expandafter
    \def\csname #1\endcsname{\mathop{\rm #1}\nolimits}%
    \multimathop}%
  \fi \next}

\multibold
qwertyuiopasdfghjklzxcvbnmQWERTYUIOPASDFGHJKLZXCVBNM.

\multimathbb
QWERTYUIOPASDFGHJKLZXCVBNM.

\multical
QWERTYUIOPASDFGHJKLZXCVBNM.

\multimathop
diag dist div dom mean meas sign supp .

\title{Complex pattern formation governed by a Cahn--Hilliard--Swift--Hohenberg system: 
\\
Analysis and numerical simulations}
\author{Harald Garcke \footnotemark[1] \and Kei Fong Lam \footnotemark[2] \and Robert N\"urnberg \footnotemark[3] \and Andrea Signori \footnotemark[4]}
\date{ }

\begin{document}
\maketitle

\begin{abstract} 
\noindent
This paper investigates a Cahn--Hilliard--Swift--Hohenberg system, focusing on a three-species chemical mixture subject to physical constraints on volume fractions. The resulting system leads to complex patterns involving a separation into phases as typical of the Cahn--Hilliard equation and small scale stripes and dots as seen in the Swift--Hohenberg equation. We introduce singular potentials of logarithmic type to enhance the model's accuracy in adhering to essential physical constraints. The paper establishes the existence and uniqueness of weak solutions within this extended framework. The insights gained contribute to a deeper understanding of phase separation in complex systems, with potential applications in materials science and related fields. We introduce a stable finite element approximation based on an obstacle formulation. Subsequent numerical simulations demonstrate that the model allows for complex structures as seen in pigment patterns of animals and in porous polymeric materials.
\end{abstract}



\renewcommand{\thefootnote}{\fnsymbol{footnote}}
\footnotetext[1]{Fakult{\"a}t f\"ur Mathematik, Universit{\"a}t Regensburg, 93040 Regensburg, Germany
({\texttt harald.garcke@ur.de}).}
\footnotetext[2]{Department of Mathematics, Hong Kong Baptist University, Kowloon Tong, Hong Kong ({\texttt akflam@math.hkbu.edu.hk}).}
\footnotetext[3]{Department of Mathematics, University of Trento, Trento, Italy
({\texttt robert.nurnberg@unitn.it}).}
\footnotetext[4]{Department of Mathematics, Politecnico di Milano, 20133 Milano, Italy ({\texttt andrea.signori@polimi.it}), Alexander von Humboldt Research Fellow.}

\renewcommand{\thefootnote}{\arabic{footnote}}

\section{Introduction}
Pattern formation, particularly in biological systems, often arises through complex interactions between various processes occurring at multiple length and time scales. The seminal works of Turing \cite{Turing}, Meinhardt and co-workers \cite{Gierer,Meinhardt} provided a methodology to study pattern formation via reaction-diffusion systems. In a multitude of subsequent works, see, e.g.,~ \cite{Meinhardt12,Morales,Morales_JTB,Murray,Page,Rossi} and the references cited therein, inclusion of chemical and mechanical effects through coupling with nonlinear systems permits more accurate descriptions of the chemical-physical processes driving skin pigmentation.

In this work we are interested in a model proposed by Mart\'inez-Agust\'in et al.~\cite{Martinez} that couples a Cahn--Hilliard equation \cite{CahnHilliard} with a Swift--Hohenberg equation \cite{Swift}. The former is a well-known model in the theory of phase separation and the latter arises as a model in the study of patterns driven by Rayleigh--B\'enard convection in fluid thermodynamics.  While differing in origin from the reaction-diffusion systems mentioned above, the solution dynamics to both the Cahn--Hilliard and Swift--Hohenberg models generate spatial-temporal patterns under appropriate conditions. The intended application for this coupled model in \cite{Martinez} is directed towards capturing the spinodal decomposition of a (charged) polymer-polymer-solvent mixture for the purpose of designing porous polymeric materials with specialized morphologies and pore sizes. In particular, by leveraging the competition between the Cahn--Hilliard and Swift--Hohenberg dynamics, a number of complex morphologies ranging from labyrinth-like patterns, mixtures of dotted and striped phases, to well-packed laminate sheets, as well as tubular hexagonal structures can be realized.

The development of these hierarchically porous structures, with their high surface areas, high pore volume ratios, and high storage capacities, serves to enable new designs for energy storage \cite{Xiao}, catalysis \cite{Sun}, sensors \cite{Bai}, separation \cite{Ma} and adsorption processes \cite{Meng}, see also \cite{Yang} for an overview. With recent advances in additive manufacturing and 3D printing technologies, such complex and hierarchically structured designs can be rapidly prototyped and deployed in areas such as bone engineering tissues \cite{Donate}, fiber-reinforced composites \cite{Dong} and organic solar cells \cite{Gusain}. We remark that the the Cahn--Hilliard--Swift--Hohenberg model studied here can also be used for pattern formation in the biological context and our numerical simulations produced configurations similar to those in \cite{Morales_JTB} resembling complex patterns on fish skins.

Let us introduce the model to be studied. In a bounded domain $\Omega \subset \RRR^d$, $d \in \{2, 3\}$, with boundary $\pd \Omega$, we consider a mixture of two chemical species and a solvent, whose volume fractions are denoted by $\phi_1$, $\phi_2$, and $\psi$, respectively. In accordance with their role as volume fractions, we expect the following physically relevant constraints
\begin{align}\label{constraint}
0 \leq \phi_1, \; \phi_2, \; \psi \leq 1, \quad \phi_1 + \phi_2 + \psi = 1 \quad \text{ for a.e.~}(x,t) \in \Omega \times [0,T],
\end{align}
where $T> 0$ denote a fixed but arbitrary terminal time. It is more convenient to introduce the auxiliary variable $\phi := \phi_1 - \phi_2$, which in turn allows us to express $\phi_1$ and $\phi_2$ as linear combinations of $\psi$ and $\phi$ as
\[
\phi_1= \tfrac{1}{2}(1-\psi + \phi), \quad \phi_2 = \tfrac{1}{2}(1-\psi - \phi).
\]
Then, the physically relevant constraints in \eqref{constraint} can be equivalently expressed as $(\phi, \psi) \in \mathcal{K}$ for a.e.~$(x,t) \in \Omega \times [0,T]$, with the convex admissible set
\begin{align}\label{K:ad}
\mathcal{K} := \{ (r,s) \in \RRR^2 \, : \, r \in [-1,1], \, r+s \leq 1, \, s- r \leq 1 \}
\end{align}
consisting of the triangular region of $\RRR^2$ enclosed by the vertices at $(-1,0)$, $(1,0)$ and $(0,1)$ {(see Figure \ref{fig:adm:K})}. %

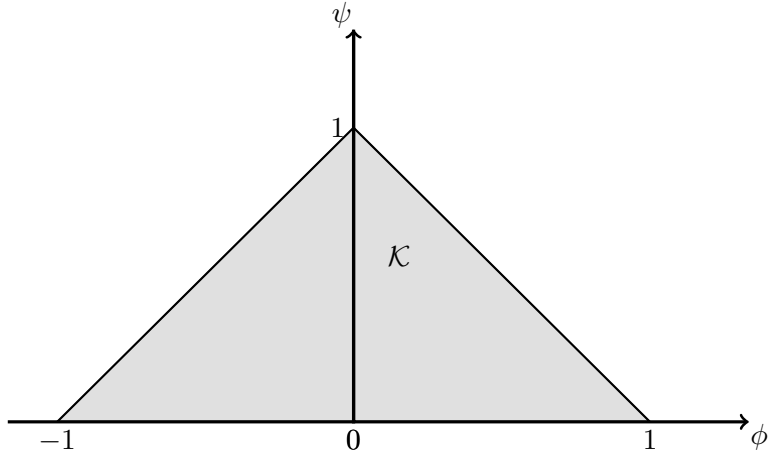
\begin{figure}
\centering
\begin{tikzpicture}[inner sep=0pt,label distance = 1em, scale=1.3]
\filldraw[color=black!100, fill=black!12, thick]
(-3,0) -- (0,3)  --(3,0) -- cycle;
\draw[very thick,<->] (4,0) node[anchor=north west] {$\phi$} -| (0,4) node[anchor=south east] {$\psi$} ;
\draw[very thick,-] (-3.5,0)  -| (0,4);
\node (K) [xshift=.6cm, yshift=2.2cm] {$\cal K$}; 
\draw (0,0) node[below=3pt] {$0$} -- (1.5,0) --
(3,0) node[below=3pt] {$1$} -- (0,3) node[left=3pt] {$1$} -- cycle;
\draw (0,0) node[below=3pt] {$0$} -- (1.5,0) --
(-3,0) node[below=3pt] {$-1$} -- (0,3) -- cycle;
\end{tikzpicture}
\caption{Schematics for the  set $\cal K$ of admissible pairs.}
\label{fig:adm:K}
\end{figure}

The total {\it free energy} of the system is given by 
\begin{align} \label{eq:Energy}
E(\phi, \psi) = \int_\Omega \frac{\eps}{2}|\nabla \phi|^2 + \frac{\lambda}{2}|(\Delta + \omega^2) (\psi-\tfrac{1}{2})|^2 + F(\phi, \psi) + \sigma  \phi \Delta \psi \, dx,
\end{align}
where $\eps > 0$, $\omega \in \RRR$, $\lambda > 0$ and $\sigma \in \RRR$ are fixed constants, $\Delta$ denotes the Neumann-Laplacian operator and $F$ denotes a potential function. It will be convenient to split $F$ into a sum of a convex part $F_0$ and a non-convex part $F_1$. One example used in previous contributions \cite{Martinez,Morales,Morales_JTB} is
\begin{equation}\label{poly}
\begin{aligned}
F_0(\phi, \psi) & = \frac{1}{4} \phi^4 + \frac{1}{4} (\psi- \tfrac{1}{2})^4,\\
F_1(\phi, \psi) & =  - \frac{\alpha}{2} \phi^2 - \frac{g}{3} (\psi-\tfrac{1}{2})^3 - \frac{\gamma}{2} (\psi-\tfrac{1}{2})^2 + \frac{\delta}{2} \phi^2 (\psi-\tfrac{1}{2}),
\end{aligned}
\end{equation}
with constants $\alpha \geq 0$, $g \in \RRR$, $\gamma \geq 0$ and $\delta \in \RRR$. In contrast to other papers, we use $|(\Delta + \omega^2) (\psi-\tfrac{1}{2})|^2$ instead of $|(\Delta + \omega^2) \psi|^2$ because we take $\psi=\frac 12$ as the center for the Swift--Hohenberg variable $\psi$, instead of $\psi=0$ in other papers. 

Setting $Q := \Omega \times (0,T)$ {and $\Sigma:=\partial \Omega \times (0,T),$} we consider the following Cahn--Hilliard--Swift--Hohenberg system that was proposed in \cite{Martinez}:
\begin{subequations}\label{CHSH}
\begin{alignat}{2}
\pd_t \phi & = \Delta \mu
\quad && \text{ in } Q, \\
{\mu} & = - \eps \Delta \phi + F_{,\phi}(\phi, \psi) + \sigma \Delta \psi 
\quad && \text{ in } Q, \label{CHSH:w} \\
\pd_t \psi & = - z \quad && \text{ in } Q, \\
z & = \lambda (\Delta + \omega^2)^2 (\psi-\tfrac{1}{2}) + F_{,\psi}(\phi, \psi) + \sigma \Delta \phi 
\quad && \text{ in } Q,\label{CHSH:z}
\\ 
\dn \phi &  = \dn \mu =  \dn \psi = \dn \Delta \psi = 0 \quad && \text{ on } \Sigma,\label{CHSH:bc}
\\  \label{CHSH:ic}
\phi(0) &  =\phi_0,
\quad \psi(0)= \psi_0 \quad && \text{ in } \Omega,
\end{alignat}
\end{subequations}
where $\dn f = \nabla f \cdot \bm{n}$ denotes the normal derivative of the function $f$ at the boundary $\pd \Omega$ with outer unit normal $\bm{n}$, and $F_{,\phi}{:=\frac {\pd F}{\pd \phi}}$ and $F_{,\psi}{:=\frac {\pd F}{\pd \psi}}$ indicate the partial derivatives of $F$. We remark that \eqref{CHSH} as presented here slightly differs from the model in \cite{Martinez} and can be recovered by performing a shift of $\psi - \frac{1}{2} \mapsto \psi$ and setting $\omega = 1$.

We now present some modeling aspects to compare our work with the existing literature. The free energy $E$ can be viewed as a sum of three energetic contributions: the first energetic contribution is a Ginzburg--Landau functional encoding short-range interactions and phase separation of the polymers and one example is
\[
E_{\mathrm{GL}}(\phi) = \int_\Omega \frac{\eps}{2} |\nabla \phi|^2 + \frac{1}{4} |\phi|^4 - \frac{\alpha}{2} |\phi|^2 \, dx,
\]
the second energetic contribution is a solvent free energy functional accounting for short-range interactions and Coulomb's electrostatic interactions between small charged particles:
\begin{align*}
E_{\mathrm{Solvent}}(\phi, \psi) & = \int_\Omega -\frac{\lambda_0}{2}
|\nabla (\psi - \tfrac{1}{2})|^ 2 +  \frac{\delta}{2} \phi^2 (\psi - \tfrac{1}{2}) \\
& \qquad + \Big ( \frac{\lambda \omega^4}{2} (\psi - \tfrac{1}{2})^2 - \frac{g}{3} (\psi - \tfrac{1}{2})^3 - \frac{\gamma}{2}(\psi - \tfrac{1}{2})^2 + \frac{1}{4} (\psi - \tfrac{1}{2})^4 \Big )  \, dx,
\end{align*}
where the second term represents a coupling between the two scalar fields $\phi$ and $\psi$ and the last term is a fourth order series expansion of the electrostatic interactions. The third energetic contribution accounts for the immiscibility between the polymers and the solvent, as well as superficial deformations like bending and stretching:
\[
E_{\mathrm{Stretch}}(\phi, \psi) = \int_\Omega 
\frac{\Lambda}{2} 
|\nabla (\psi - \tfrac{1}{2})|^2 + \frac{\lambda}{2} |\Delta (\psi - \tfrac{1}{2})|^2 +  \sigma \phi \Delta (\psi - \tfrac{1}{2})\, dx
\]
with the last term providing a coupling between the polymer order parameter $\phi$ and a measure of the local curvature $\Delta \psi$, see \cite{Martinez,Morales} for further details. Depending on the phase given by the value parameter $\phi$, a different sign of the ``spontaneous curvature" is preferred. We consider setting $\lambda_0 = \Lambda + 2 \lambda$ to obtain the energy functional \eqref{Energy} used for the subsequent mathematical analysis, as well as to recover the setting considered in \cite{Martinez,Morales}. A significant drawback of smooth potentials such as \eqref{poly} is their inability to ensure that the solutions adhere to the essential physical constraint $(\phi, \psi) \in \mathcal{K}$. One main aim of this work is to address this limitation through the use of suitable singular potentials ensuring the physical validity of the solutions. In particular, instead of the quartic polynomial function $F_0$ in \eqref{poly}, we suggest the singular form
\begin{equation}\label{log}
F_0(\phi, \psi) = \frac{\theta}{2} \Big [ \Pi \Big (\tfrac{1}{2}(1-\psi+\phi) \Big )  + \Pi \Big (\tfrac{1}{2}(1-\psi-\phi) \Big ) + \Pi(\psi) \Big ],
\end{equation}
where $\Pi(s) := s \ln s$ and $\theta$ plays the role of the absolute temperature, so that finite values are attained when
\[
1-\psi + \phi \geq 0, \quad 1- \psi - \phi \geq 0, \quad \psi \geq 0 \quad \Longleftrightarrow \quad (\phi, \psi) \in \mathcal{K}.
\]
In the formal deep quench limit $\theta \to 0$, we arrive at
\begin{equation}\label{obs}
F_0(\phi, \psi) = \mathbb{I}_{\mathcal{K}}(\phi,\psi) = \begin{cases} 0 & \text{ if } (\phi,\psi) \in \mathcal{K}, \\
+\infty & \text{ otherwise},
\end{cases}
\end{equation}
where $\mathbb{I}_X$ denotes the indicator function of the set $X$. Our theoretical investigations address potentials $F = F_0 + F_1$ where the convex part $F_0$ is taken either as the logarithmic function \eqref{log} or the indicator function \eqref{obs}, while the non-convex perturbation $F_1$ can be taken as in \eqref{poly}. We shall henceforth refer to $F$ as the logarithmic potential if $F_0$ is of the form \eqref{log} and as the obstacle potential if $F_0$ is of the form \eqref{obs}.

Despite both the Cahn--Hilliard and Swift--Hohenberg are fourth order equations, the former is a $H^{-1}$-gradient flow while the latter is a $L^2$-gradient flow of their respective energy functionals.  Hence, \eqref{CHSH} can be interpreted as a $H^{-1} \times L^2$ gradient flow of a suitable energy functional as shown in Section~\ref{sec:model:der}.  In particular, this differs from the conventional multicomponent Cahn--Hilliard systems \cite{EllGar,EllLuck}, and it turns out that the situation when considering Swift--Hohenberg equation with singular terms is more involved compared to the second order case with the Allen--Cahn equation. Similarly observed in \cite{CMP}, a weak solution to \eqref{CHSH} is defined based on a variational inequality for \eqref{CHSH:z}. While this is natural if $F_0$ is the indicator function \eqref{obs}, for the logarithmic function \eqref{log} it is weaker than the conventional variational solutions for similar systems. Nevertheless, our chief result ensures well-posedness to the coupled system \eqref{CHSH} with either \eqref{log} or \eqref{obs}.

The rest of the paper is organized as follows: in Section~\ref{sec:main} we provide a derivation of \eqref{CHSH} and list the main results, whose proofs can be found in Section~\ref{sec:analysis}. We introduce a fully discrete and unconditionally stable numerical scheme in Section~\ref{sec:numerics} and present various numerical simulations showing that the Cahn--Hilliard--Swift--Hohenberg system models complex pattern formation scenarios.

\section{Model derivation, main assumptions and results}\label{sec:main}

\subsection{Notation}
Let $\Omega$ be a bounded domain in $\erre^d$, where $d \in \{2,3\}$.
The Lebesgue measure of $\Omega$ and the Hausdorff measure of $\pd \Omega$ are denoted by $|\Omega|$ and $|\pd \Omega|$, respectively.

For any Banach space $X$, the norm of $X$ is represented as $\norma{\cdot}_X$, its dual space is denoted as $X^*$, and the duality pairing between $X^*$ and $X$ is given by $\<\cdot,\cdot>_X$.  In the case where $X$ is a Hilbert space, the inner product is denoted by $(\cdot,\cdot)_X$.

For each $1 \leq p \leq \infty$, $k \geq 0$, and $s>0$, the standard Lebesgue and Sobolev spaces defined on $\Omega$ are denoted as $L^p(\Omega)$, and $W^{k,p}(\Omega)$, with their respective norms $\norma{\cdot}_{L^p(\Omega)}$, and $\norma{\cdot}_{W^{k,p}(\Omega)}$, respectively. In some instances, we use $\norma{\cdot}_{L^p}$ instead of $\norma{\cdot}_{L^p(\Omega)}$, and employ a similar shorthand for other norms. We adopt the standard convention $H^k(\Omega):= W^{k,2}(\Omega)$ for all $k\in\mathbb{N}$, and denote the mean value of a function $f \in L^1(\Omega)$ and a functional $h \in H^1(\Omega)^*$ as 
\[
\inn{f}_\Omega := \frac{1}{|\Omega|} \int_\Omega f \,dx, \quad \inn{h}_\Omega := \frac{1}{|\Omega|} \inn{h,1}_{H^1}.
\]
\subsection{Model derivation}\label{sec:model:der}
As previously mentioned, we consider the following energy functional
\begin{align}\label{Energy}
E(\phi, \psi) = \int_\Omega \frac{\eps}{2}|\nabla \phi|^2 + \frac{\lambda}{2}|(\Delta + \omega^2) (\psi-\tfrac{1}{2})|^2 + F(\phi, \psi) - \sigma \nabla \phi \cdot \nabla \psi \, dx
\end{align}
with fixed constants $\omega \in \RRR$, $\lambda > 0$ and $\sigma \in \RRR$, and in the subsection we take a potential function $F$ which is differentiable. For $m \in \RRR$, we define
\[
H^1(\Omega)_{m} := \{ f \in H^1(\Omega) \, : \, \inn{f}_\Omega = m \}, \quad H^1(\Omega)^*_{0} := \{ h \in H^1(\Omega)^* \, : \, \inn{h}_{\Omega} = 0 \}
\]
and the operator $\mathcal{N} : H^1(\Omega)^*_{0} \to H^1(\Omega)_{0}$ defined as the map $h \mapsto \mathcal{N}(h)$ with $\mathcal{N}(h)$ as the solution to the variational equality
\[
\int_\Omega \nabla \mathcal{N}(h) \cdot \nabla \zeta \,dx = \inn{h,\zeta}_{H^1} \quad \forall \zeta \in H^1(\Omega),
\]
which can be interpreted as the inverse Neumann--Laplacian operator. Based on this operator, we consider the inner product
\begin{align}\label{innerprod}
\inn{(f,h),(\gamma,v)}_{H^1(\Omega)^*_{0} \times L^2(\Omega)} := \int_\Omega \nabla \mathcal{N}(f) \cdot \nabla \mathcal{N}(\gamma) + h v \, dx
\end{align}
for $f, \gamma \in H^1(\Omega)^*_0$ and $h,v \in L^2(\Omega)$. 
We now consider $E$ to be defined on $H^{1}(\Omega)_m \times H^2_{\bnn}(\Omega)$ with $ H^2_{\bnn}(\Omega) := \{ f \in H^2(\Omega) \, : \, {\dn} f = 0 \text{ on } \pd \Omega\}.$
Then, for arbitrary $\zeta \in H^1(\Omega)_{0}$ and $\eta \in H^2_{\bnn}(\Omega)$, we compute the first variation of $E$ with respect to $(\phi, \psi)$ in the direction $(\zeta, \eta)$ as
\begin{align*}
\frac{\delta E}{\delta (\phi,\psi)}((\phi, \psi))[(\zeta,\eta)] & = \int_\Omega \eps \nabla \phi \cdot \nabla \zeta + F_{,\phi}(\phi,\psi) \zeta - \sigma \nabla \zeta \cdot \nabla \psi \, dx \\
& \quad + \int_\Omega \lambda (\Delta + \omega^2) (\psi-\tfrac{1}{2}) \, (\Delta +\omega^2) \eta - \sigma {\nabla \phi \cdot  \nabla \eta }+ F_{,\psi}(\phi, \psi) \eta \, dx.
\end{align*}
We aim to identify the gradient of $E$, expressed as the pair $(p,q)$, with respect to the inner product \eqref{innerprod}. Namely, we are looking for the pair $(p,q) \in H^1(\Omega)^*_{0} \times L^2(\Omega)$ fulfilling
\[
\inn{(p,q),(\zeta, \eta)}_{H^1(\Omega)^*_{0} \times L^2(\Omega)} = \frac{\delta E}{\delta (\phi,\psi)}(\phi, \psi)[(\zeta,\eta)] \quad \text{ for every } \zeta \in H^1(\Omega)_0,  \eta \in H^2_{\bnn}(\Omega).
\]  
Then, from the definition of $\mathcal{N}$, we find that 
\begin{align*}
\int_\Omega \mathcal{N}(p) \zeta + q \eta \, dx & = \int_\Omega \eps \nabla \phi \cdot \nabla \zeta + F_{,\phi}(\phi,\psi) \zeta - \sigma \nabla \psi \cdot \nabla \zeta \, dx \\
& \quad + \int_\Omega \lambda (\Delta + \omega^2) (\psi-\tfrac{1}{2}) \, (\Delta +\omega^2) \eta - \sigma {\nabla \phi \cdot  \nabla \eta }+ F_{,\psi}(\phi, \psi) \eta \, dx.
\end{align*}
For arbitrary $\tilde{\zeta} \in H^1(\Omega)$ we insert $\zeta = \tilde{\zeta} - \inn{\tilde{\zeta}}_\Omega \in H^1(\Omega)_0$ into the above identity, and then performing integration by parts on the right-hand side yields
\begin{align*}
& \int_\Omega (\mathcal{N}(p) + \inn{F_{,\phi}(\phi, \psi)}_\Omega )\tilde{\zeta} + q \eta \, dx \\
& \quad = \int_\Omega \Big ( - \eps \Delta \phi + F_{,\phi}(\phi, \psi) + \sigma \Delta \psi \Big ) \tilde{\zeta} \, dx + \int_{\pd \Omega} (\eps \dn \phi - \sigma \dn \psi) \tilde{\zeta} \, ds \\
& \qquad + \int_\Omega \Big ( \lambda (\Delta + \omega^2)^2 (\psi-\tfrac{1}{2}) + \sigma \Delta \phi + F_{,\psi}(\phi, \psi) \Big ) \eta \, dx 
\\ & \qquad 
- \int_{\pd \Omega} \Big (\lambda \dn((\Delta + \omega^2) \psi + \sigma \dn \phi \Big ) \eta \, dS,
\end{align*}
for arbitrary $\tilde{\zeta} \in H^1(\Omega)$, and $\eta \in H^2_{\bnn}(\Omega)$, and we noted that a boundary term of the form $\lambda (\Delta + \omega^2) (\psi-\tfrac{1}{2}) \dn \eta$ vanished due to the fact that $\eta \in H^2_{\bnn}(\Omega)$. By the fundamental theorem of calculus of variations we have the identifications
\begin{align*}
\mu & := \mathcal{N}(p) + \inn{F_{,\phi}(\phi, \psi)}_\Omega = -\eps \Delta \phi + F_{,\phi}(\phi, \psi) + \sigma \Delta \psi, \\
z & := q = \lambda (\Delta + \omega^2)^2 (\psi-\tfrac{1}{2})  + F_{,\psi}(\phi, \psi) + \sigma \Delta \phi,
\end{align*}
along with the boundary conditions
\[
\dn \phi = 0, \quad \dn \psi = 0, \quad \dn \Delta \psi = 0.
\]
Thus, the gradient flow
\[
\inn{(\pd_t \phi, \pd_t \psi), (\zeta, \eta)}_{H^1(\Omega)^*_0 \times L^2(\Omega)} = - \frac{\delta E}{\delta (\phi, \psi)} (\phi, \psi) [(\zeta, \eta)],\]
for arbitrary $\zeta \in H^1(\Omega)_0$ and $\eta \in H^2_{\bnn}(\Omega)$ can be expressed as
\begin{align*}
& \int_\Omega \nabla \mathcal{N}(\pd_t \phi) \cdot \nabla \mathcal{N}(\zeta) + \pd_t \psi \eta \, dx \\
& \quad = \int_\Omega \mathcal{N}(\pd_t \phi) \zeta + \pd_t \psi \eta \, dx = \int_\Omega -(\mu - \inn{F_{,\phi}(\phi, \psi)}_\Omega) \zeta - z \eta \, dx.
\end{align*}
Substituting $\eta = \tilde{\eta} - \inn{\tilde{\eta}}_\Omega$ for arbitrary $\tilde{\eta} \in H^1(\Omega)$ then leads to the identification
\[
\mathcal{N}(\pd_t \phi) =  - \mu + \inn{\mu}_\Omega , \quad \pd_t \psi = - z,
\]
where by the definition of $\mathcal{N}$ we infer from the first equation:
\[
\pd_t \phi = \Delta \mu \quad \text{ in } \Omega, \qquad \dn \mu = 0 \quad \text{ on } \pd \Omega.
\]
\begin{remark}
Alternate choices of boundary conditions to \eqref{CHSH:bc} for similar types of equations can be found in \cite[Sec.~9]{Espath}.
\end{remark}

\subsection{Assumptions and main results}
We make the following assumptions:
\begin{enumerate}[label=$(\mathrm{A \arabic*})$, ref = $\mathrm{A \arabic*}$]
\item \label{ass:dom} $\Omega \subset \RRR^{d}$, $d \in \{2,3\}$, is a bounded domain which is either convex or has $C^{1,1}$ boundary $\pd \Omega$.
\item \label{ass:F} The potential function $F = F_0 + F_1$ is a sum of $F_0 : \RRR^2 \to [0,\infty)$ taking the form \eqref{log} or \eqref{obs}, while $F_1 \in C^{1}(\mathcal{K})$.
\item \label{ass:ini} The initial conditions satisfy $\phi_0 \in H^1(\Omega)$, $\psi_0 \in {H^2_{\bnn}(\Omega)}$, with $(\phi_0, \psi_0) \in \mathcal{K}$ for a.e.~$x \in \Omega$ and $\inn{\phi_0}_\Omega \in (-1,1)$.
\end{enumerate}

In order to introduce an appropriate notion of solution to \eqref{CHSH} with singular potentials, we first make the following definition.
\begin{defn}[Admissible function pair]\label{defn:add}
\
\begin{itemize}
\item For the logarithmic potential \eqref{log}, we say that a pair of functions $(\zeta, \eta)$ is log-admissible if
\[
\pi(\eta) {\in L^1(Q)}, \quad \pi \big ( \tfrac{1}{2}(1+\zeta - \eta) \big ){\in L^1(Q)}, \quad \pi \big ( \tfrac{1}{2}(1-\zeta - \eta ) \big ) \in L^1(Q),
\]
where $\pi(s) = 1 + \ln s$.
\item For the obstacle potential \eqref{obs}, we say that a pair of functions $(\zeta, \eta)$ is obstacle-admissible if
\[
(\zeta, \eta) \in \mathcal{K} \quad \text{ for a.e.~} (x,t) \in Q.
\]
\end{itemize}
\end{defn}

\begin{defn}[Variational solution]\label{defn:var}
A quadruple of functions $(\phi, \psi, {\mu}, z)$ is a variational solution to \eqref{CHSH} if
\begin{enumerate}
\item[(i)] they satisfy the regularities
\begin{align*}
\phi & \in L^\infty(0,T;H^1(\Omega)) \cap L^2(0,T;H^2_{\bnn}(\Omega)) \cap H^1(0,T;(H^1(\Omega))^*), \\
\psi & \in L^\infty(0,T;H^2_{\bnn}(\Omega)) \cap H^1(0,T;L^2(\Omega)), \\
\Delta \psi & \in L^2(0,T;H^1(\Omega)), \\
\mu & \in L^2(0,T;H^1(\Omega)), \\
z& \in L^2(0,T;L^2(\Omega)).
\end{align*}
\item[(ii)] The initial conditions are attained: $\phi(0) = \phi_0$ and $\psi(0) = \psi_0$ a.e.~in $\Omega$.
\item[(iii)] For a.e.~$t \in (0,T)$, arbitrary $v \in L^2(Q)$ and $u \in L^2(0,T;H^1(\Omega))$, it holds that
\begin{subequations}
\begin{alignat}{2}
0 & =  \inn{\pd_t \phi,u}_{H^1} + \int_\Omega \nabla \mu \cdot \nabla u \, dx, \label{w:1} \\
0 & = \int_\Omega \pd_t \psi v + z v \, dx. \label{w:2} 
\end{alignat}
\end{subequations}
\item[(iv-a)] If $F_0$ is the logarithmic potential \eqref{log}, then $(\phi, \psi)$ is log-admissible in the sense of Definition \ref{defn:add} and for a.e.~$t \in (0,T)$ and arbitrary log-admissible pair $(\zeta, \eta) \in (L^2(0,T;H^1(\Omega)))^2$:
\begin{equation}\label{w:3}
\begin{aligned}
0 & \leq \int_\Omega \big (F_{0,\phi}(\phi, \psi) + F_{1,\phi}(\phi, \psi) - \mu \big ) (\zeta - \phi) + \nabla ( \eps \phi - \sigma \psi)\cdot \nabla (\zeta - \phi) \, dx  \\
 & \quad + \int_\Omega  \big (F_{0,\psi}(\phi, \psi) + F_{1,\psi}(\phi, \psi) - z  \big ) (\eta - \psi) - \sigma \nabla \psi  \cdot \nabla  (\eta - \psi) \, dx \\
 & \quad + \int_\Omega \big (2 \lambda \omega^2 \Delta \psi + \lambda \omega^4 (\psi-\tfrac{1}{2}) \big ) (\eta - \psi) - \lambda \nabla \Delta \psi \cdot \nabla (\eta - \psi) \, dx.
\end{aligned}
\end{equation}
\item[(iv-b)] If $F_0$ is the obstacle potential \eqref{obs}, then $(\phi, \psi)$ is obstacle-admissible in the sense of Definition \ref{defn:add} and for a.e.~$t \in (0,T)$ and arbitrary obstacle-admissible pair $(\zeta, \eta) \in (L^2(0,T;H^1(\Omega))^2$:
\begin{equation}\label{w:3:obs}
\begin{aligned}
0 & \leq \int_\Omega \big (F_{1,\phi}(\phi, \psi) - \mu \big ) (\zeta - \phi) 
+ \nabla (\eps \phi - \sigma \psi)\cdot \nabla (\zeta - \phi)
\, dx \\
 & \quad + \int_\Omega  \big (F_{1,\psi}(\phi, \psi) - z \big ) (\eta - \psi) 
- \sigma \nabla \psi  \cdot \nabla  (\eta - \psi) \, dx \\
 & \quad + \int_\Omega \big (2 \lambda \omega^2 \Delta \psi + \lambda \omega^4 (\psi-\tfrac{1}{2}) \big ) (\eta - \psi) - \lambda \nabla \Delta \psi \cdot \nabla (\eta - \psi) \, dx.
\end{aligned}
\end{equation}
\end{enumerate}
\end{defn}

\begin{remark}\label{rem:varineq}
The equations \eqref{CHSH:w} and \eqref{CHSH:z} are formulated together as a variational inequality. For the logarithmic potential \eqref{log}, it turns out that we can show $F_{0,\phi}(\phi, \psi) \in L^2(Q)$ and hence item (iv-a) in Definition \ref{defn:var} can be replaced by the requirement that $(\phi, \psi)$ is log-admissible and satisfies
\begin{align}\label{mu:equ:ae}
& \mu = - \eps \Delta \phi + F_{0,\phi}(\phi, \psi) + F_{1,\phi}(\phi, \psi) + \sigma \Delta \psi \quad \text{ a.e.~in } Q, 
\end{align}
and
\begin{equation}\label{log:VI:2}
\begin{aligned}
& 0 \leq \int_\Omega \big (F_{0,\psi}(\phi, \psi) + F_{1,\psi}(\phi, \psi) - z \big ) (\eta - \psi) - \sigma \nabla \psi \cdot \nabla (\eta - \psi) \, dx \\
& \quad \quad + \int_\Omega \big ( 2 \lambda \omega^2 \Delta \psi + \lambda \omega^4 (\psi - \tfrac{1}{2}) \big ) (\eta - \psi) - \lambda \nabla \Delta \psi \cdot \nabla (\eta - \psi) \, dx
\end{aligned}
\end{equation}
holding for a.e.~$t \in (0,T)$ and arbitrary $\eta \in L^2(0,T;H^1(\Omega))$ such that $(\phi, \eta)$ is log-admissible, i.e., $\pi(\tfrac{1}{2}(1+\phi - \eta)) \in L^1(Q)$ and $\pi(\tfrac{1}{2}(1-\phi - \eta)) \in L^1(Q)$.
\end{remark}


Our main result is formulated as follows.

\begin{thm}[Well-posedness of variational solutions]\label{thm:wellposed}
Under assumptions \eqref{ass:dom}-\eqref{ass:ini}, there exists a unique variational solution $(\phi, \psi, {\mu}, z)$ to the system \eqref{CHSH} in the sense of Definition \ref{defn:var}, and for any pair of variational solutions $\{(\phi_i, \psi_i, {\mu}_i, z_i)\}_{i=1,2}$ with initial data $\{(\phi_{0,i}, \psi_{0,i})\}_{i=1,2}$ fulfilling \eqref{ass:ini}, there exist a positive constant $C$ independent of their differences $\hat{\phi} := \phi_1 - \phi_2$, $\hat{\psi} := \psi_1 - \psi_2$, $\hat{\mu} := \mu_1 - \mu_2$, $\hat{z} = z_1 - z_2$, such that 
\begin{equation}\label{cts:dep}
\begin{aligned}
 & \sup_{t \in (0,T]} \Big ( \| \nabla \mathcal{N}\big(
\hat{\phi} - \inn{\hat{\phi}}_\Omega \big )(t) \|^2 + \| \hat{\psi}(t) \|^2 \Big ) + \int_0^T \| \hat{\phi} \|_{H^1}^2 + \| \hat{\psi} \|_{H^2}^2 \, dt \\
& \quad \leq C \Big ( \| \nabla \mathcal{N}\big( \hat{\phi}_0 - \inn{\hat{\phi}_0}_\Omega \big) \|^2 + \| \hat{\psi}_0 \|^2 \Big ).
\end{aligned}
\end{equation}
\end{thm}

We have the following connection between the logarithmic  potential and the obstacle potential.

\begin{thm}[Deep quench limit]
For $\theta \in (0,1]$ we denote by $(\phi_\theta, \psi_\theta, {\mu}_\theta, z_\theta)$ to be a variational solution to \eqref{CHSH} for the logarithmic potential \eqref{log} originating from the initial conditions $(\phi_0, \psi_0)$ fulfilling \eqref{ass:ini}. Then, as $\theta \to 0$, 
\begin{align*}
\phi_\theta \to \phi_* & \quad \text{ weakly* in } L^\infty(0,T;H^1(\Omega))  \cap H^1(0,T;(H^1(\Omega))^*), \\
\psi_\theta \to \psi_* & \quad \text{ weakly* in } L^\infty(0,T;{H^2_{\bnn}(\Omega)}) \cap H^1(0,T;L^2(\Omega)), \\
{\mu}_\theta \to {\mu}_* &  \quad \text{ weakly in } L^2(0,T;H^1(\Omega)), \\
z_\theta \to z_* & \quad \text{ weakly in } L^2(0,T;L^2(\Omega)),
\end{align*}
where $(\phi_*, \psi_*, {\mu}_*, z_*)$ is the unique variational solution to \eqref{CHSH} with obstacle potential \eqref{obs} originating from the same initial data. Furthermore, there exists a positive constant $C$, independent of $\theta$, such that 
\begin{equation}\label{rate}
\begin{aligned}
&  \sup_{t \in (0,T]} \Big ( \| \nabla \mathcal{N}(\phi_*(t) - \phi_\theta(t)) \|^2 + \| \psi_*(t) -\psi_\theta(t) \|^2 \Big ) \\
& \quad + \int_0^T \| \phi_* - \phi_\theta \|_{H^1}^2 + \| \psi_* - \psi_\theta \|_{H^2}^2 \, dt \leq C \theta.
\end{aligned}
\end{equation}
\end{thm}

\section{Mathematical analysis}\label{sec:analysis}
\subsection{Existence of variational solutions}
We first treat the case of the logarithmic potential \eqref{log} and defer the case of the obstacle potential \eqref{obs} to Section \ref{sec:obs}. The first step is to regularize the nonlinearity by formulating an appropriate approximation scheme.

\subsubsection{Approximation scheme}
For $N \in \mathbb{N}$ we consider the fourth-order Taylor approximation $\Pi_N$ of $\Pi$ given by:
\begin{align}\label{Pi}
\Pi_N(s) = \begin{cases}
\Pi(s) = s \ln(s) & \text{ for } s \geq \frac{1}{N}, \\
\sum_{j=0}^4 \frac{1}{j!} \Pi^{(j)}(\frac{1}{N}) (s - \frac{1}{N})^j & \text{ for } s \leq \frac{1}{N}.
\end{cases}
\end{align}
We note that there exist a constant $d_1 \geq 0$ such that 
\begin{align}\label{Pi:lb}
\Pi_N(s) \geq - d_1 \quad \forall s \in \RRR.
\end{align}
Let us now show another useful property that will be used later on.
\begin{lem}
Setting $\pi_N(s) = \Pi_N'(s)$, for any $u \in (0,1)$ there exist constants $d_2>0$, $d_3 \geq 0$ dependent on $u$ but independent of $N$ such that for $N$ sufficiently large,
\begin{align}\label{MZ:ineq}
|\pi_N(s)| \leq d_2 \pi_N(s)(s - u) + d_3 \quad \forall s \in \RRR.
\end{align}
\end{lem}

\begin{proof}
For fixed $u \in (0,1)$ and for all $N > \frac{2}{u}$, we can choose $d_2 = \frac{2}{u}$ so that  $d_2(s-u) \leq -1$ for all $s \leq \frac{1}{N}$ and
\begin{align}\label{MZ:1}
d_2 \pi_N(s)(s-u) \geq |\pi_N(s)|.
\end{align}
For the remaining case we first note that the solvability of the nonlinear equation $s (2 + \ln(s)) = y$ for any $y \in (0, {\frac 32})$ holds with positive solutions. Then, setting $d_2 = \frac{2}{u}$ we notice that for $\frac{1}{N} \leq s \leq e^{-1}$ we have $|1+\ln(s)| = -1 - \ln(s)$ and the function
$f(s) = \frac{2}{u}(1+\ln(s))(s-u) - |1+\ln(s)|$ admits a minimum at $s_* \in (0,1)$ which also solves $s_*(2+\ln(s_*)) = \frac{u}{2}$. Hence, there exists a constant $d \geq 0$ such that 
\begin{align}\label{MZ:2}
\frac{2}{u}\pi_N(s)(s-u) - |\pi_N(s)| \geq -d .
\end{align}
Lastly for $s \geq e^{-1}$, we have $|1+\ln(s)| = 1+\ln(s)$ and analogously the function $g(s) = \frac{2}{u}(1+\ln(s))(s-u) - 1 - \ln(s)$ admits a minimum at $s^* \in (0,1)$ which also solves $s^* (2+\ln(s^*)) = \frac{3u}{2}$. Likewise, we can find a constant $d \geq 0$ such that \eqref{MZ:2} is fulfilled. Hence, \eqref{MZ:ineq} holds for the approximation $\pi_N$. 
\end{proof}

We then define
\[
F_0^N(r,s) := \theta \Big [ \Pi_N \big ( \tfrac{1}{2}(1+r-s) \big ) + \Pi_N \big ( \tfrac{1}{2}(1-r-s) \big) + \Pi_N(s) \Big ]
\quad \forall r,s \in \RRR,
\]
and from the above definition of $\Pi_N$ we deduce via Young's inequality and the fact $\Pi^{(4)}(\frac{1}{N}) = 2N^3 > 0$ that there exist constants $d_4 > 0$ and $d_5 \geq 0$ independent of $N \in \mathbb{N}$ such that 
\begin{align}\label{log:app} 
F_0^N(r,s) \geq d_4 (|r|^4 + |s|^4) - d_5 \quad {\forall r,s \in \RRR}.
\end{align}
In Theorem~\ref{thm:wellposed}, the solution $(\phi, \psi)$ only take values in the admissible set $\mathcal{K}$. Hence, only $F_{1,\phi}$ and $F_{1,\psi}$ restricted to $\mathcal{K}$ enter into the definition of a variational solution.  We can therefore extend $F_1$ from $\mathcal{K}$ to the whole of $\RRR^2$ such that $F_1$, $F_{1,\phi}$ and $F_{1,\psi}$ are bounded. Then, replacing $F_0$ with $F_0^N$ leads to the following approximate problem expressed in strong form:
\begin{subequations}\label{CHSH:approx}
\begin{alignat}{2}
\pd_t \phi_N & = \Delta {\mu}_N
\quad && \text{ in } Q, \label{app:1} \\
{\mu}_N & = - \eps \Delta \phi_N + F^{N}_{0,\phi}(\phi_N, \psi_N) + F_{1,\phi}(\phi_N, \psi_N) + \sigma \Delta \psi_N
\quad && \text{ in } Q, \label{app:2}  \\
\pd_t \psi_N & = - z_N \quad && \text{ in } Q,  \label{app:3} \\
z_N & = \lambda (\Delta + \omega^2)^2 (\psi_N-\tfrac{1}{2})  +  F^N_{0,\psi}(\phi_N, \psi_N) + F_{1,\psi}(\phi_N, \psi_N) + \sigma \Delta \phi_N
\quad && \text{ in } Q,  \label{app:4}
\\ \label{app:5}
{\dn} \phi_N &  = {\dn} {\mu}_N =  {\dn} \psi_N = {\dn} \Delta \psi_N = 0 \quad && \text{ on } \Sigma, \\ \label{app:6}
\phi_N(0) &  =\phi_0, \quad \psi_N(0)= \psi_0 \quad && \text{ in } \Omega.
\end{alignat}
\end{subequations}
The existence of a weak solution tuple $(\phi_N, \psi_N, {\mu}_N, z_N)$ can be established by a standard Galerkin approximation. We omit the details here as in the next section we will derive uniform estimates that can also be used as a foundation for the existence proof of \eqref{CHSH:approx} via a Galerkin approximation.

\subsubsection{Uniform estimates}
In the sequel the symbol $C$ denotes nonnegative constants independent of $N$ whose value may change from line to line and also within the same line. We test \eqref{app:1} with ${\mu}_N$, \eqref{app:2} with $\pd_t \phi_N$, \eqref{app:3} with $z_N$ and \eqref{app:4} with $\pd_t \psi_N$, respectively. Without entering the details, let us highlight that all the mentioned testing procedures can be justified within a rigorous framework as mentioned above. Then, upon summing, we arrive at the energy identity
\begin{align*}
\frac{d}{dt} E_N(\phi_N, \psi_N) + \| \nabla {\mu}_N \|^2 + \| z_N \|^2 = 0,
\end{align*}
where $E_N$ is the energy functional 
\begin{align*}
E_N(\phi_N, \psi_N) & = \int_\Omega \frac{\eps}{2} |\nabla \phi_N|^2 + \frac{\lambda}{2} |(\Delta + \omega^2) (\psi_N-\tfrac{1}{2})|^2  \, dx \\
& \quad +  \int_\Omega F^N_0(\phi_N, \psi_N) + F_{1}(\phi_N, \psi_N) - \sigma \nabla \phi_N \cdot \nabla \psi_N \, dx.
\end{align*}
Upon integrating in time, for arbitrary $t \in (0,T]$, the above identity gives
\begin{align}\label{est:1}
E_N(\phi_N(t), \psi_N(t)) + \int_0^t \| \nabla {\mu}_N \|^2 + \| z_N \|^2 \, dt = E_N(\phi_0, \psi_0).
\end{align}
Note that by assumption \eqref{ass:ini} for the initial conditions in \eqref{app:5}, it holds that 
\[
E_N(\phi_0, \psi_0) \leq C \big ( \| \phi_0 \|_{H^1}, \| \psi_0 \|_{H^2} \big ).
\]
Thanks to the Neumann boundary conditions, we have the interpolation inequality and elliptic regularity estimates, see, e.g.,~ \cite[Thm.~2.4.2.7]{Gris} for $C^{1,1}$-domains or \cite[Thm.~3.2.1.3]{Gris} for convex domains:
\begin{align}\label{inter:1}
\| \nabla \psi_N \|^2 \leq \| \Delta \psi_N \| \| \psi_N \|, \quad \| \psi_N \|_{H^2} \leq C(\Omega) \big ( \| \Delta \psi_N \| + \| \psi_N \| \big ),
\end{align}
which leads to the lower bound 
\begin{align*}
\|(\Delta + \omega^2) (\psi_N - \tfrac{1}{2}) \|^2 & \geq \| \Delta \psi_N \|^2 + \omega^4 \| \psi_N -\tfrac{1}{2} \|^2 - 2 \omega^2 \| \Delta \psi_N \| \| \psi_N - \tfrac{1}{2} \| \\
& \geq \frac{1}{4} \| \Delta \psi_N \|^2 - 3\omega^4 \| \psi_N -\tfrac{1}{2} \|^2.
\end{align*}
Furthermore, for the term with indefinite sign $-\sigma \nabla \phi_N \cdot \nabla \psi_N$ in $E_N$, we compute a lower bound:
\begin{equation}\label{curvature:term}
\begin{aligned}
- \sigma \int_\Omega  \nabla \phi_N \cdot \nabla (\psi_N -\tfrac{1}{2}) \, dx & \geq - |\sigma| \| \nabla \phi_N \| \| \Delta \psi_N \|^{1/2} \| \psi_N -\tfrac{1}{2} \|^{1/2} \\
& \geq - \frac{\eps}{4} \| \nabla \phi_N \|^2 - \frac{\lambda}{4} \| \Delta \psi_N \|^2 - \frac{|\sigma|^4}{\lambda \eps^2} \| \psi_N -\tfrac{1}{2} \|^2.
\end{aligned}
\end{equation}
Using the lower bound in \eqref{log:app} for $F_0^N$, by Young's inequality we see that 
\begin{align*}
\frac{1}{2} \int_\Omega F_0^N(\phi_N, \psi_N) \, dx & \geq \frac{d_4}{4} \| \psi_N - \tfrac{1}{2} \|_{L^4}^4 - C \geq \Big ( \frac{|\sigma|^4}{\lambda \eps^2} + 3 \omega^4 \Big ) \| \psi_N - \tfrac{1}{2} \|^2 - C,
\end{align*}
and hence we deduce the lower bound for $E_N$:
\begin{equation}\label{lowerb}
\begin{aligned}
E_N(\phi_N, \psi_N) & \geq \frac{\eps}{4} \| \nabla \phi_N \|^2 + \frac{\lambda}{4} \| \Delta \psi_N \|^2  \\
& +   \frac{1}{2}\int_\Omega F_0^N(\phi_N, \psi_N) \, dx +  \int_\Omega F_1(\phi_N, \psi_N) \, dx - C,
\end{aligned}
\end{equation}
where $C$ is a positive constant independent of $N$, $\phi_N$ and $\psi_N$. Next, we test \eqref{app:1} with $1/|\Omega|$ and \eqref{app:3} with $\psi_N$ to obtain
\begin{align}\label{mass:cons}
\inn{\phi_N(t)}_\Omega = \inn{\phi_0}_\Omega \quad \forall t \in (0,T],
\end{align}
and
\begin{align}\label{psi:L2}
\frac{1}{2} \| \psi_N(t) \|^2 \leq \frac{1}{2} \|\psi_0 \|^2 + \frac{1}{2} \int_0^t \| z_N \|^2 \, dt + \frac{1}{2} \int_0^t \| \psi_N \|^2 \, dt.
\end{align}
Adding this to \eqref{est:1}, applying Gronwall's inequality and the lower bound \eqref{lowerb} in conjunction with Poincar\'e's inequality for $\phi_N$ and the elliptic regularity estimate \eqref{inter:1} for $\psi_N$, as well as comparison argument in \eqref{app:3}, we deduce the uniform estimates
\begin{equation}\label{est:2}
\begin{aligned}
& \| \phi_N \|_{L^\infty(0,T;H^1)} + \| \psi_N \|_{L^\infty(0,T;H^2)} + \| \pd_t \psi_N \|_{L^2(Q)} \\
& \quad + \| \nabla {\mu}_N \|_{L^2(Q)} + \| z_N \|_{L^2(Q)} + {\norma{F_0^N(\phi_N, \psi_N)}_{\L\infty {\Lx1}}} \leq C.
\end{aligned}
\end{equation}
Since $\inn{\phi_N}_\Omega = \inn{\phi_0}_\Omega \in (-1,1)$, we see that the constant $\nu := \frac{1}{2}(1- |\inn{\phi_0}_\Omega|)$ satisfies the properties
\[
\nu \in (0,1) \quad \text{ and } \quad \frac{1}{2} (1 - {\nu} \pm \inn{\phi_0}_\Omega) \in (0,1).
\]
Then, we consider testing \eqref{app:2} with $\phi_N - \inn{\phi_N}_\Omega$ and \eqref{app:4} with $\psi_N - {\nu} $, which leads to
\begin{align*}
& \int_\Omega \eps | \nabla \phi_N|^2 + F_{0,\phi}^N(\phi_N, \psi_N) (\phi_N - \inn{\phi_N}_\Omega) \, dx \\
& \quad = \int_\Omega ({\mu}_N - \inn{{\mu}_N}_\Omega) (\phi_N - \inn{\phi_N}_\Omega) - \big ( F_{1,\phi}(\phi_N, \psi_N) + \sigma \Delta \psi_N)  (\phi_N - \inn{\phi_N}_\Omega) \, dx \\
& \quad \leq C \big ( \| \nabla {\mu}_N \| +  \| F_{1,\phi}(\phi_N, \psi_N) \| + \sigma \| \Delta \psi_N \| \big )  \| \nabla \phi_N \| \\
& \quad \leq C(1 + \| \nabla {\mu}_N \|),
\end{align*}
and
\begin{align*}
& \int_\Omega  \lambda |(\Delta + \omega^2) \psi_N|^2 + F_{0,\psi}^N (\phi_N, \psi_N)( \psi_N- {\nu} ) \, dx \\
& \quad = \int_\Omega \sigma \nabla \phi_N \cdot \nabla \psi_N + (z_N - F_{1,\psi}(\phi_N, \psi_N))(\psi_N - {\nu} ) \, dx \\
& \qquad   + \int_\Omega \tfrac{1}{2} \lambda \omega^2 (\Delta + \omega^2)(\psi_N - \nu) + \lambda \nu(\Delta + \omega^2) \psi_N + \tfrac{1}{2} \lambda \nu \, dx  \\
& \quad \leq C( 1 + \| z_N \|),
\end{align*}
where we have used the uniform estimates \eqref{est:2}, as well as the boundedness of $F_1$, $F_{1,\phi}$ and $F_{1,\psi}$ to deduce that $F_{1,\phi}(\phi_N, \psi_N)$ and $F_{1,\psi}(\phi_N, \psi_N)$ are uniformly bounded in $L^\infty(0,T;L^2(\Omega))$. Using \eqref{mass:cons} and adding these inequalities, we find that 
\begin{equation}\label{F0_der}
\begin{aligned}
& \int_\Omega F_{0,\phi}^N(\phi_N, \psi_N) (\phi_N - \inn{\phi_0}_\Omega) + F_{0,\psi}^N(\phi_N, \psi_N)(\psi_N - {\nu} ) \, dx \\
& \quad \leq C ( 1 + \| \nabla {\mu}_N \| + \| z_N \| ).
\end{aligned}
\end{equation}
A direct computation shows the left-hand side can be expressed as
\begin{align*}
& \int_\Omega \frac{\theta}{2} \big (\pi_N \big ( \tfrac{1}{2} (1-\psi_N + \phi_N) \big ) - \pi_N \big ( \tfrac{1}{2} (1-\psi_N - \phi_N) \big ) \big )(\phi_N - \inn{\phi_0}_\Omega) \, dx \\
& \quad + \int_\Omega \frac{\theta}{2} (\pi_N \big ( \tfrac{1}{2} (1-\psi_N + \phi_N) \big ) + \pi_N \big (\tfrac{1}{2} (1-\psi_N - \phi_N) \big ) - \pi_N(\psi_N) ) \\
& \qquad \times ((1-\psi_N) - (1- {\nu} )) \, dx \\
& = \int_\Omega \theta (\pi_N \big ( \tfrac{1}{2}(1-\psi_N + \phi_N) \big ) [\tfrac{1}{2}(1-\psi_N + \phi_N) -\tfrac{1}{2} (1-{\nu}  + \inn{\phi_0}_\Omega)] \, dx \\
& \quad + \int_\Omega \theta \pi_N \big ( \tfrac{1}{2}(1-\psi_N - \phi_N) \big ) [ \tfrac{1}{2} (1-\psi_N - \phi_N) - \tfrac{1}{2}(1- {\nu} -\inn{\phi_0}_\Omega)] \, dx \\
& \quad + \int_\Omega \theta \pi_N(\psi_N) (\psi_N - {\nu} )\, dx.
\end{align*}
As $\tfrac{1}{2}(1-\nu \pm \inn{\phi_0}_\Omega) \in (0,1)$ and $\nu \in (0,1)$, by invoking \eqref{MZ:ineq} we deduce that 
\begin{align*}
& \| \pi_N \big ( \tfrac{1}{2}(1+\phi_N - \psi_N) \big ) \|_{L^1} + \| \pi_N \big ( \tfrac{1}{2}(1-\phi_N - \psi_N) \big ) \|_{L^1} + \| \pi_N(\psi_N) \|_{L^1} \\
& \quad \leq C( 1 + \| \nabla {\mu}_N \| + \| z_N \|),
\end{align*}
which implies
\begin{align}\label{est:3}
\| F_{0,\phi}^N(\phi_N, \psi_N) \|_{L^2(0,T;L^1)} + \| F_{0,\psi}^N(\phi_N, \psi_N) \|_{L^2(0,T;L^1)} \leq C.
\end{align}
Consequently, integrating \eqref{app:2} yields the estimate on the mean value of ${\mu}_N$:
\[
\|\inn{{\mu}_N}_\Omega \|_{L^2(0,T)} \leq \| F_{0,\phi}^N(\phi_N, \psi_N) \|_{L^2(0,T;L^1)} + \| F_{1,\phi}(\phi_N, \psi_N) \|_{L^2(0,T;L^1)} \leq C,
\]
and by the Poincar\'e inequality we obtain 
\begin{align}\label{est:4}
\| {\mu}_N \|_{L^2(Q)} \leq C.
\end{align}
Then, testing \eqref{app:2} with $- \Delta \phi_N$ and \eqref{app:4} with $- \Delta \psi_N$, upon summing we obtain
\begin{align*}
& \int_\Omega F_{0,\phi \phi}^N (\phi_N, \psi_N) \nabla \phi_N \cdot \nabla \phi_N + \nabla F_{0,\phi \psi}^N( \phi_N, \psi_N) \nabla \psi_N \cdot \nabla \phi_N \, dx \\
& \qquad + \int_\Omega  F_{0,\psi \phi}^N(\phi_N, \psi_N) \nabla \phi_N \cdot \nabla \psi_N + F_{0,\psi \psi}^N(\phi_N, \psi_N) \nabla \psi_N \cdot \nabla \psi_N \, dx \\
& \qquad + \eps \| \Delta \phi_N \|^2 + \lambda \| \nabla \Delta \psi_N \|^2 \\
& \quad = \int_\Omega \Big (F_{1,\phi}(\phi_N, \psi_N) - {\mu}_N + \sigma \Delta \psi_N \Big ) \Delta \phi_N \, dx \\
& \qquad + \int_\Omega \Big ( F_{1,\psi}(\phi_N, \psi_N) + \sigma \Delta \phi_N - z_N + 2 \lambda \omega^2 \Delta \psi_N + \lambda \omega^4 (\psi_N-\tfrac{1}{2}) \Big ) \Delta \psi_N \, dx \\
& \quad \leq \frac{1}{2} \| \Delta \phi_N \|^2 + C \big ( 1 + \| {\mu}_N \|^2 + \| z_N \|^2 + \| \nabla F_{1}(\phi_N, \psi_N) \|^2 \big ).
\end{align*}
Using the convexity of $F_0^N$ the sum of the first and second lines on the left-hand side is non-negative. Recalling that $F_{1,\phi}(\phi_N, \psi_N)$ and $F_{1,\psi}(\phi_N, \psi_N)$ are bounded in $L^\infty(0,T;L^2(\Omega))$, we then infer the estimate
\begin{align*}
\| \Delta \phi_N \|_{L^2(0,T;L^2)}^2 + \| \nabla \Delta \psi_N \|_{L^2(0,T;L^2)}^2 \leq C,
\end{align*}
and by invoking elliptic regularity estimate \eqref{inter:1} we obtain
\begin{align}\label{est:5}
\| \phi_N \|_{L^2(0,T;H^2)}  \leq C.
\end{align}
Through a comparison of terms in \eqref{app:2} we deduce also that
\begin{align}\label{est:6}
\| F_{0,\phi}^N (\phi_N, \psi_N) \|_{L^2(Q)} \leq C.
\end{align}
Lastly, testing \eqref{app:1} with an arbitrary test function $\zeta \in L^2(0,T;H^1(\Omega))$ yields
\begin{align}\label{est:7}
\| \pd_t \phi_N \|_{L^2(0,T;(H^1)^*)} \leq C,
\end{align}
while from the uniform estimate \eqref{est:2} for $\pd_t \psi_N$ we readily see that 
\begin{align}\label{est:8}
\|  \pd_t \Delta \psi_N \|_{L^2(0,T;H^2_{\bnn}(\Omega)^*)} \leq \| \pd_t \psi_N \|_{L^2(Q)} \leq C.
\end{align}

\begin{remark}
With a more regular domain boundary $\pd \Omega$, it is possible to obtain a uniform boundedness of $\psi_N$ in $L^2(0,T;H^3(\Omega))$.
\end{remark}

\subsubsection{Passing to the limit}\label{sec:lim}
From the uniform estimates \eqref{est:2}, \eqref{est:3}--\eqref{est:7} and \eqref{est:8}, we deduce the existence of limit functions $\phi$, $\psi$, ${\mu}$ and $z$, as well as a non-relabelled subsequence $N \to \infty$, such that as $N \to \infty$,
\begin{equation}\label{compact}
\begin{aligned}
\phi_N \to \phi &\quad  \text{ weakly* in } L^\infty(0,T;H^1(\Omega)) \cap L^2(0,T;H^2_{\bnn}(\Omega)) \cap H^1(0,T;H^1(\Omega)^*), \\
\phi_N \to \phi & \quad \text{ strongly in } C^0([0,T];L^s(\Omega)) \cap L^2(0,T;W^{1,s}(\Omega)) \text{ and a.e.~in } Q, \\
\psi_N \to \psi & \quad \text{ weakly* in } L^\infty(0,T;H^2_{\bnn}(\Omega)) \cap H^1(0,T;L^2(\Omega)), \\
\psi_N \to \psi & \quad \text{ strongly in } C^0([0,T];W^{1,s}(\Omega)) \text{ and a.e.~in }Q, \\
\Delta \psi_N \to \Delta \psi &\quad  \text{ weakly in } L^2(0,T;H^1(\Omega)) \cap H^1(0,T;H^2_{\bnn}(\Omega)^*), \\
\Delta \psi_N \to \Delta \psi &\quad  \text{ strongly in }  L^2(0,T;L^s(\Omega)), \\
\mu_N \to \mu &\quad  \text{ weakly in } L^2(0,T;H^1(\Omega)), \\
z_N \to z &\quad  \text{ weakly in } L^2(Q),
\end{aligned}
\end{equation}
for any $s < \infty$ in two dimensions and any $s \in [2,6)$ in three dimensions. Note that the initial conditions are attained by virtue of the continuity properties $\phi \in C^0([0,T];L^2(\Omega))$ and $\psi \in C^0([0,T];H^1(\Omega))$. By the generalized dominated convergence theorem and the boundedness of $F_1$, $F_{1,\phi}$ and $F_{1,\psi}$  we deduce as $N \to \infty$
\[
F_{1,\phi}(\phi_N, \psi_N) \to F_{1,\phi}(\phi, \psi) \,\text{ and }\, F_{1,\psi}(\phi_N, \psi_N) \to F_{1,\psi}({\phi, \psi}) \,\, \text{ strongly in } L^2(Q).
\]
From the a.e.~convergence of $\phi_N$ and $\psi_N$, we deduce that $F_{0,\phi}^N(\phi_N, \psi_N) \to F_{0,\phi}(\phi, \psi)$ a.e.~in $Q$.  Together with the uniform estimate \eqref{est:6}, invoking Vitali's convergence theorem yields the strong convergence of $F_{0,\phi}^N(\phi_N, \psi_N)$ to $F_{0,\phi}(\phi, \psi)$ in $L^q(Q)$ for $q \in [1,2)$.  This allows us to identify the weak limit of $F_{0,\phi}^N(\phi_N, \psi_N)$ in $L^2(Q)$ as $F_{0,\phi}(\phi,\psi)$ and obtain 
\begin{align}\label{F0phi:weak}
F_{0,\phi}^N(\phi_N, \psi_N) \to F_{0,\phi}(\phi, \psi) \text{ weakly in } L^2(Q).
\end{align}
However, as we only have a uniform estimate of $F_{0,\psi}^N(\phi_N, \psi_N)$ in $L^2(0,T;L^1(\Omega))$, this is insufficient to identify the limit of $F_{0,\psi}^N(\phi_N, \psi_N)$ as $N \to \infty$. This primarily explains why, in the limit, we can just achieve a variational inequality rather than an equality. Nevertheless, we can show that the limit functions satisfy the physical property $(\phi, \psi) \in \mathcal{K}$ for a.e.~$(x,t) \in Q$. Owning to the explicit expression for $\Pi_N$, recall that $d_1\geq0$, we have for $s \leq \frac{1}{N}$,
\begin{align*}
\Pi_N(s) = \frac{2N^3}{4!}\Big (s-\frac{1}{N}\Big )^4 - \frac{N^2}{3!} \Big (s-\frac{1}{N}\Big )^3 + \frac{N}{2}\Big (s-\frac{1}{N}\Big )^2 + \Big (\ln \frac{1}{N} + 1 \Big ) \Big (s - \frac{1}{N}\Big ) + \frac{1}{N} \ln \frac{1}{N}.
\end{align*}
By Young's inequality we deduce that 
\[
\frac{2N^3}{4!}\Big (s-\frac{1}{N}\Big )^4 - \frac{N^2}{3!}\Big (s-\frac{1}{N}\Big )^3 + \frac{N}{2}\Big (s-\frac{1}{N}\Big )^2 \geq 0 \quad \forall s \in \RRR,
\]
and hence, due to the lower bound in \eqref{Pi:lb}, for $N$ sufficiently large we obtain
\begin{align*}
& \int_Q (\Pi_N(\psi_N) + d_1) \, dx \, dt \geq \int_{\{(x,t) \in Q \, : \, \psi_N(x,t) < 0 \}} (\Pi_N(\psi_N) + d_1) \, dx \, dt \\
& \quad \geq \Big ( \ln \frac{1}{N} + 1\Big ) \int_{\{(x,t) \in Q \, : \,  \psi_N(x,t) < 0 \}} \psi_N \, dx \, dt - \frac{1}{N} |\Omega| T \\
& \quad =  \left | \ln \frac{1}{N} + 1 \right | \int_Q (-\psi_N)_+ \, dx \, dt - \frac{1}{N} |\Omega |T,
\end{align*}
where, for a given scalar function $f$, $(f)_+ = \max(f, 0)$ denotes the positive part of $f$, and $(-f)_+ = \max(-f, 0) = - \min(f, 0)$ is the negative part of $f$. Combining this with the uniform estimate \eqref{est:2} for $F_0^N$, we infer that 
\begin{align*}
& C \geq \int_Q F_0^N(\phi_N, \psi_N) \, dx\, dt \\
& \quad \geq \theta \left | \ln \frac{1}{N} + 1 \right | \int_Q (-\psi_N)_+ + \big (- \tfrac{1}{2}(1+\phi_N - \psi_N) \big)_+ + (- \tfrac{1}{2}(1-\phi_N - \psi_N) \big)_+ \, dx \, dt \\
& \qquad - 3 \theta \frac{1}{N} |\Omega |T.
\end{align*}
Since $3 \theta \frac{1}{N} |\Omega |T \leq C$, we obtain
\begin{equation}\label{est:log}
\begin{aligned}
& \| (-\psi_N)_+ \|_{L^1(0,T;L^1)} + \| (- \tfrac{1}{2}(1+\phi_N - \psi_N) \big)_+ \|_{L^1(0,T;L^1)} \\
& \quad + \|(- \tfrac{1}{2}(1-\phi_N - \psi_N) \big)_+ \|_{L^1(0,T;L^1)} \leq \frac{C}{\theta |\ln \tfrac{1}{N} + 1|},
\end{aligned}
\end{equation}
where the right-hand side vanishes as $N \to \infty$. By Fatou's lemma we deduce that the limit functions $\phi$ and $\psi$ satisfy
\[
(-\psi)_+ = 0, \quad (- \tfrac{1}{2}(1+\phi - \psi) \big)_+ = 0, \quad (- \tfrac{1}{2}(1-\phi - \psi) \big)_+ = 0 \quad \text{ a.e.~in } Q,
\]
which in turn implies 
\begin{align}\label{K:sat}
\psi \geq 0, \quad \frac{1}{2}(1+ \phi - \psi) \geq 0, \quad \frac{1}{2}(1-\phi -\psi) \geq 0 \quad \text{ a.e.~in } Q,
\end{align}
meaning that $(\phi, \psi) \in \mathcal{K}$ a.e.~in $Q$ as claimed.

Lastly, for arbitrary test functions $v \in L^2(Q)$, $u \in L^2(0,T;H^1(\Omega))$ and arbitrary log-admissible test function pair $(\zeta, \eta) \in (L^2(0,T;H^1(\Omega))^2$, we test \eqref{app:1} with $u$, \eqref{app:3} with $v$, \eqref{app:2} with $\zeta - \phi$ and \eqref{app:4} with $\eta - \psi$. Then, integrating by parts and upon adding the resulting equalities involving \eqref{app:2} and \eqref{app:4} we obtain
\begin{subequations}
\begin{alignat}{2}
0 & = \int_Q \pd_t \phi_N u + \nabla {\mu}_N \cdot \nabla u \, dx \, dt, \label{lim:1} \\
0 & = \int_Q \pd_t \psi_N v + z_N v \, dx \, dt, \label{lim:2} \\
0 & = \int_Q  \big (F_{0,\phi}^N(\phi_N, \psi_N) + F_{1,\phi}(\phi_N, \psi_N) - \mu_N  \big ) (\zeta - \phi_N) \, dx \, dt \label{lim:3}  \\
\notag & \quad + \int_Q \nabla (\eps \phi_N + \sigma \psi_N) \cdot \nabla (\zeta - \phi_N) \, dx \, dt \\
\notag & \quad + \int_Q  - (\lambda \nabla \Delta \psi_N + \sigma \nabla \phi_N) \cdot \nabla (\eta - \psi_N) + \big (F_{1,\psi}(\phi_N, \psi_N) - z_N \big ) (\eta - \psi_N)  \, dx \, dt \\
\notag & \quad + \int_Q \big (2 \lambda \omega^2 \Delta \psi_N + \lambda \omega^4 (\psi_N-\tfrac{1}{2}) + F_{0,\psi}^N(\phi_N, \psi_N) \big ) (\eta - \psi_N)  \, dx \, dt.
\end{alignat}
\end{subequations}
Monotonicity of $\pi_N$ yields
\[
(\pi_N(r) - \pi_N(s))(r - s) \geq 0 \quad \forall r, s \in \RRR,
\]
and a short calculation reveals that for log-admissible test function pair $(\zeta, \eta)$ it holds that
\begin{align}\label{F0N:mono}
(F_{0,\phi}^N(\zeta, \eta) - F_{0,\phi}^N(\phi_N, \psi_N)) (\zeta - \phi_N) + (F_{0,\psi}^N(\zeta, \eta) - F_{0,\psi}^N(\phi_N, \psi_N)(\eta - \psi_N) \geq 0.
\end{align}
Hence, we replace \eqref{lim:3} with the inequality
\begin{equation}\label{lim:3:alt}
\begin{aligned}
0 & \leq \int_Q  \big (F_{0,\phi}^N(\zeta, \eta) + F_{1,\phi}(\phi_N, \psi_N) 
- {\mu}_N  \big ) (\zeta - \phi_N) \, dx \, dt \\
& \quad + \int_Q \nabla (\eps \phi_N - \sigma \psi_N) \cdot \nabla (\zeta - \phi_N) \, dx \, dt \\
 & \quad + \int_Q  -(\lambda \nabla \Delta \psi_N + \sigma \nabla \phi_N)  \cdot \nabla (\eta - \psi_N) + \big (F_{1,\psi}(\phi_N, \psi_N) - z_N \big ) (\eta - \psi_N) \, dx \, dt \\
 & \quad + \int_Q \big (2 \lambda \omega^2 \Delta \psi_N + \lambda \omega^4 (\psi_N-\tfrac{1}{2}) + F_{0,\psi}^N(\zeta, \eta) \big ) (\eta - \psi_N)  \, dx \, dt.
\end{aligned}
\end{equation}
Passing to the limit $N \to \infty$ with the compactness assertions \eqref{compact} yields that the limit functions $\phi$, $\psi$, ${\mu}$ and $z$ satisfy
\begin{subequations}
\begin{alignat}{2}
0 & = \int_0^T  \inn{\pd_t \phi,u}_{H^1} \, dt + \int_Q\nabla {\mu} \cdot \nabla u \, dx \, dt, \label{wl:1} \\
0 & = \int_Q \pd_t \psi v + z v \, dx \, dt, \label{wl:2} \\
0 & \leq \int_Q  \big (F_{0,\phi}(\zeta, \eta) + F_{1,\phi}(\phi, \psi)  -\mu
\big ) (\zeta - \phi) 
+ \nabla (\eps \phi- \sigma \psi) \cdot \nabla  (\zeta - \phi) \, dx \, dt \label{wl:3} \\
\notag & \quad + \int_Q  -(\lambda \nabla \Delta \psi + \sigma \nabla \phi) \cdot \nabla (\eta - \psi) + \big (F_{1,\psi}(\phi, \psi) 
- z \big ) (\eta - \psi)  \, dx \, dt \\
\notag & \quad + \int_Q \big (2 \lambda \omega^2 \Delta \psi + \lambda \omega^4 (\psi-\tfrac{1}{2}) + F_{0,\psi}(\zeta, \eta) \big ) (\eta - \psi) \, dx \, dt,
\end{alignat}
\end{subequations}
for arbitrary $v \in L^2(Q)$, $u \in L^2(0,T;H^1(\Omega))$ and log-admissible test function pair $(\zeta, \eta) \in (L^2(0,T;H^1(\Omega))^2$. In the above we also used that
\[
|\pi_N(s)| \leq |\pi(s)| + 1 \quad \text{ for } s \in (0,1),
\]
so that for an arbitrary log-admissible test function pair $(\zeta, \eta)$ we have
\begin{align*}
|F_{0,\phi}^N(\zeta, \eta)| & \leq 2 + |\pi \big ( \tfrac{1}{2}(1+\zeta - \eta) \big )| + |\pi \big ( \tfrac{1}{2}(1-\zeta - \eta) \big )|, \\
|F_{0,\psi}^N(\zeta, \eta)| & \leq 3 + |\pi(\eta)| +  |\pi \big ( \tfrac{1}{2}(1+\zeta - \eta) \big )| + |\pi \big ( \tfrac{1}{2}(1-\zeta - \eta) \big )|,
\end{align*}
whence by the generalized Lebesgue dominated convergence theorem we obtain
\[
F_{0,\phi}^N(\zeta, \eta) \to F_{0,\phi}(\zeta,\eta) \quad \text{ and } \quad F_{0,\psi}^N(\zeta, \eta) \to F_{0,\psi}(\zeta, \eta) \quad \text{ strongly in } L^1(Q).
\]
Note that \eqref{wl:3} is an alternate variational inequality where we evaluated $F_{0,\phi}$ and $F_{0,\psi}$ at the log-admissible test function pair $(\zeta, \eta)$ instead at the solution pair $(\phi,\psi)$.  To recover \eqref{w:3} we argue similar to the ideas of \cite{CMP}. Let $(\alpha, \beta) \in (L^2(0,T;H^1(\Omega))^2$ be an arbitrary log-admissible test function pair, and we consider, for $\kappa \in (0,1]$, 
\[
\zeta_\kappa = (1-\kappa) \phi + \kappa \alpha \quad \text{ and } \quad \eta_\kappa = (1-\kappa) \psi + \kappa \beta.
\]
Then, it is clear that $(\zeta_\kappa, \eta_\kappa) \in \mathcal{K}$ a.e.~in $Q$, and by the convexity of $h: s \mapsto |\ln(s)|$ we can deduce that 
\begin{equation}\label{uniq:pre:1}
\begin{aligned}
& \big | \pi \big (\tfrac{1}{2} ((1-\kappa)(1\pm \phi - \psi) + \kappa (1 \pm \alpha - \beta)) \big) \big |  \\
& \quad  \leq 1 + h\big (\tfrac{1}{2} ((1-\kappa)(1\pm \phi - \psi) + \kappa (1 \pm \alpha - \beta)) \big) \\
& \quad \leq 1 + (1-\kappa) h \big ( \tfrac{1}{2} (1\pm \phi - \psi) \big) + \kappa h \big (\tfrac{1}{2} (1 \pm \alpha - \beta) \big)  \\
& \quad \leq 2 + |\pi \big (\tfrac{1}{2}(1\pm \phi - \psi) \big )| + |\pi\big ( \tfrac{1}{2} (1\pm \alpha - \beta) \big )| \in L^1(Q),
\end{aligned}
\end{equation} 
so that by a similar argument,
\begin{align}\label{uniq:pre:2}
|\pi ((1-\kappa) \psi + \kappa \beta)| \leq 2 + |\pi(\psi)| + |\pi(\beta)| \in L^1(Q).
\end{align}
Hence, $(\zeta_\kappa, \eta_\kappa)$ is a log-admissible test function pair in the sense of Definition \ref{defn:add}.  Substituting this choice of $\zeta_\kappa$ and $\eta_\kappa$ into \eqref{w:3} and dividing by $\kappa$ we find that 
\begin{equation}\label{uniq:pre:3}
\begin{aligned}
0 & \leq \int_Q  \big (F_{0,\phi}(\zeta_\kappa, \eta_\kappa) + F_{1,\phi}(\phi, \psi)  -\mu \big ) (\alpha - \phi) 
+ \nabla (\eps \phi - \sigma \psi) \cdot \nabla (\alpha - \phi) 
\, dx \, dt  \\
 & \quad + \int_Q \big (- \lambda \nabla \Delta \psi - \sigma \nabla \phi \big ) \cdot \nabla (\eta - \psi) + \big (F_{1,\psi}(\phi, \psi) - z \big ) (\beta - \psi) \, dx \, dt \\
 & \quad + \int_Q \big (2 \lambda \omega^2 \Delta \psi + \lambda \omega^4 (\psi-\tfrac{1}{2}) + F_{0,\psi}(\zeta_\kappa, \eta_\kappa) \big ) (\alpha - \psi) \, dx \, dt.
\end{aligned}
\end{equation}
By virtue of \eqref{uniq:pre:1} and \eqref{uniq:pre:2}{,} we infer that
\begin{align*}
& |F_{0,\phi}(\zeta_\kappa, \eta_\kappa)| + |F_{0,\psi}(\zeta_\kappa, \eta_\kappa)| \\
& \quad \leq C \big ( 1 + |\pi(\psi)| + |\pi(\beta)| +  |\pi \big (\tfrac{1}{2}(1\pm \phi - \psi) \big )| + |\pi\big ( \tfrac{1}{2} (1\pm \alpha - \beta) \big )| \big )
\end{align*}
uniformly in $\kappa \in (0,1]$. Hence, by the dominated convergence theorem we obtain, as $\kappa\to 0$,
\[
F_{0,\phi}(\zeta_\kappa, \eta_\kappa) \to F_{0,\phi}(\phi, \psi) \quad \text{ and } \quad F_{0,\psi}(\zeta_\kappa, \eta_\kappa) \to F_{0,\psi}(\phi, \psi) \quad \text{ strongly in } L^1(Q),
\]
and by passing to the limit $\kappa \to 0$ in \eqref{uniq:pre:3} we obtain \eqref{w:3}. This shows that $(\phi, \psi, {\mu}, z)$ is a variational solution to \eqref{CHSH} with logarithmic potential in the sense of Definition \ref{defn:var}.

\begin{remark}
By utilizing \eqref{compact} and \eqref{F0phi:weak}, we can pass to the limit in \eqref{app:2} to deduce that the limit functions $(\phi, \psi, \mu)$ satisfy \eqref{mu:equ:ae}. On the other hand, \eqref{log:VI:2} can be derived by choosing $\zeta = \phi$ in \eqref{w:3} whilst keeping $\eta$ arbitrary.
\end{remark}

\begin{remark}
We mention that a weaker variational inequality than \eqref{w:3} or \eqref{lim:3:alt} can also be derived. Let us use the notation $F_{\log} = F_0$ as there is no ambiguity. We start with \eqref{lim:3} with arbitrary test function $(\zeta, \eta) \in (L^2(0,T;H^1(\Omega))^2$ such that $F_{\log}(\zeta, \eta) \in L^1(Q)$. For instance, an obstacle-admission test function pair satisfies the requirement due to the continuity of $F_{\log}$ over $\mathcal{K}$. Using the convexity of $F_{\log}^N$ we have instead of \eqref{F0N:mono} the following relation
\[
\int_\Omega F_{\log}^N(\zeta, \eta) - F_{\log}^N(\phi_N, \psi_N) \, dx \geq \int_\Omega (F_{\log,\phi}^N(\phi_N, \psi_N)(\zeta - \phi_N) + F_{\log,\psi}^N(\phi_N, \psi_N)(\eta - \psi_N) \, dx.
\]
Then, instead of \eqref{lim:3:alt} we obtain the variational inequality
\begin{equation}\label{lim:3:alt:weak}
\begin{aligned}
0 & \leq \int_Q  \big ( F_{1,\phi}(\phi_N, \psi_N) - {\mu}_N  \big ) (\zeta - \phi_N) + \nabla (\eps \phi_N - \sigma \psi_N) \cdot \nabla (\zeta - \phi_N) \, dx \, dt \\
 & \quad + \int_Q  -(\lambda \nabla \Delta \psi_N + \sigma \nabla \phi_N)  \cdot \nabla (\eta - \psi_N) + \big (F_{1,\psi}(\phi_N, \psi_N) - z_N \big ) (\eta - \psi_N) \, dx \, dt \\
 & \quad + \int_Q \big (2 \lambda \omega^2 \Delta \psi_N + \lambda \omega^4 (\psi_N-\tfrac{1}{2}) \big ) (\eta - \psi_N) + F_{\log}^N(\zeta, \eta) - F_{\log}^N(\phi_N, \psi_N) \, dx \, dt,
\end{aligned}
\end{equation}
holding for arbitrary $(\zeta, \eta) \in (L^2(0,T;H^1(\Omega))^2$ such that $F_{\log}(\zeta, \eta) \in L^1(Q)$.  Passing to the limit $N \to \infty$ yields 
\begin{equation}\label{w:3:alt:weak}
\begin{aligned}
0 & \leq \int_\Omega  \big ( F_{1,\phi}(\phi, \psi) - \mu  \big ) (\zeta - \phi) + \nabla (\eps \phi - \sigma \psi) \cdot \nabla (\zeta - \phi) \, dx \\
 & \quad + \int_\Omega  -(\lambda \nabla \Delta \psi + \sigma \nabla \phi)  \cdot \nabla (\eta - \psi) + \big (F_{1,\psi}(\phi, \psi) - z \big ) (\eta - \psi) \, dx \\
 & \quad + \int_\Omega \big (2 \lambda \omega^2 \Delta \psi + \lambda \omega^4 (\psi-\tfrac{1}{2}) \big ) (\eta - \psi)  + F_{\log}(\zeta, \eta) - F_{\log}(\phi, \psi) \, dx,
\end{aligned}
\end{equation}
holding for a.e.~$t \in (0,T)$ and arbitrary $(\zeta, \eta) \in (L^2(0,T;H^1(\Omega))^2$ such that $F_{\log}(\zeta, \eta) \in L^1(Q)$.  
\end{remark}

\subsection{Continuous dependence and uniqueness}\label{sec:cts}
Let now $(\phi_1, \psi_1, {\mu}_1, z_1)$ and $(\phi_2, \psi_2, {\mu}_2, z_2)$ be two variational solutions to \eqref{CHSH} with logarithmic potential corresponding to initial data $(\phi_{0,1}, \psi_{0,1})$ and $(\phi_{0,2}, \psi_{0,2})$, respectively. Consider \eqref{w:3} for $(\phi_1, \psi_1)$ with $\zeta = \phi_2$ and $\eta = \psi_2$, and likewise with the alternate variational inequality \eqref{wl:3} for $(\phi_2, \psi_2)$ with $\zeta = \phi_1$ and $\eta = \psi_1$. Upon summing the resulting inequalities we obtain for the differences $\hat{\phi} := \phi_1 - \phi_2$, $\hat{\psi} := \psi_1 - \psi_2$, $\hat{{\mu}} := {\mu}_1 - {\mu}_2$ and $\hat{z} := z_1 - z_2$ that
\begin{equation}\label{uniq:1}
\begin{aligned}
0 & \geq \int_\Omega \big (F_{1,\phi}(\phi_1, \psi_1) - F_{1,\phi}(\phi_2, \psi_2) - \hat{{\mu}}  \big ) \hat{\phi}
+ \nabla (\eps \hat{\phi} - \sigma \hat{\psi}) \cdot \nabla \hat{\phi}  \, dx \\
& \quad + \int_\Omega - \lambda \nabla \Delta \hat{\psi} \cdot \nabla \hat{\psi} + \big ( F_{1,\psi}(\phi_1, \psi_1) - F_{1,\psi}(\phi_2, \psi_2)  
- \hat{z} \big )\hat{\psi} - \sigma \nabla \hat{\phi} \cdot \nabla  \hat{\psi} \, dx \\
& \quad + \int_\Omega (2 \lambda \omega^2 \Delta \hat{\psi} + \lambda \omega^4 \hat{\psi} \big ) \hat{\psi} \, dx,
\end{aligned}
\end{equation}
where we had a cancellation of terms involving $F_0$.

Next, we consider the difference between \eqref{w:1} and \eqref{w:2} for the two solutions $(\phi_1, \psi_1, {\mu}_1, z_1)$ and $(\phi_2, \psi_2, {\mu}_2, z_2)$ to derive that 
\begin{align*}
0  = \inn{\pd_t \hat{\phi}, u}_{H^1} + \int_\Omega \nabla \hat{{\mu}} \cdot \nabla  u\, dx,
\quad 
0  = 
\iO \pd_t \hat{\psi} v
+  \hat{z} {v} \, dx,
\end{align*}
for any $u \in \Hx1,$ and $v \in \Lx2$. From the first equality, it readily follows that, for every $t \in [0,T]$, $\inn{\hat{\phi}(t)}_\Omega = \<\phi_{0,1}-\phi_{0,2}>_\Omega$. Next, we can consider the choices $u = \mathcal{N}(\hat{\phi} -\hat \phi_\Omega )$ and $v = \hat{\psi}$ to infer
\begin{align*}
0 & = \inn{\pd_t \hat{\phi}, \mathcal{N}(\hat{\phi}-\hat \phi_\Omega )}_{H^1} + \int_\Omega \hat{{\mu}} (\hat{\phi}-\hat \phi_\Omega ) \, dx = {\frac{1}{2}\frac{d}{dt} } \| \nabla \mathcal{N} (\hat{\phi}-\hat \phi_\Omega ) \|^2 + \int_\Omega \hat{{\mu}} (\hat{\phi}-\hat \phi_\Omega ) \, dx , \\
0 & = {\frac{1}{2}\frac{d}{dt} } \| \hat{\psi} \|^2 + \int_\Omega \hat{z} {\hat{\psi}} \, dx.
\end{align*}
Adding these to \eqref{uniq:1} and using the local Lipschitz continuity of $F_{1,\phi}$, $F_{1,\psi}$, as well as the boundedness of $\phi_i$ and $\psi_i$, $i = 1,2$, leads to 
\begin{equation}\label{uniq:2}
\begin{aligned}
& \frac{1}{2}\frac{d}{dt}  \Big ( \| \nabla \mathcal{N}( \hat \phi -\hat \phi_\Omega ) \|^2 + \| \hat{\psi} \|^2 \Big ) + \eps \| \nabla \hat{\phi} \|^2 + \lambda \| \Delta \hat{\psi} \|^2 + \lambda \omega^4 \| \hat{\psi} \|^2 \\
& \quad \leq \int_\Omega |F_{1,\phi}(\phi_1, \psi_1) - F_{1,\phi}(\phi_2, \psi_2)| |\hat{\phi} -\hat \phi_\Omega | + |F_{1,\psi}(\phi_1, \psi_1) - F_{1,\psi}(\phi_2, \psi_2)| |\hat{\psi}| \, dx \\
& \qquad + 2 |\sigma| \| \nabla \hat{\psi} \| \| \nabla \hat{\phi} \| + 2 \lambda \omega^2 \| \nabla \hat{\psi} \|^2 \\
& \quad \leq C \big ( \| \hat{\phi} -\hat \phi_\Omega \|^2 + \|\hat{\psi}\|^2 \big )
+ \frac{\lambda}{2} \| \Delta \hat{\psi} \|^2 + \frac{\eps}{4} \| \nabla \hat{\phi} \|^2 + C \| \hat{\psi} \|^2 \\
& \quad \leq C \Big (\| \nabla \mathcal{N}(\hat{\phi} -\hat \phi_\Omega)  \|^2 + \| \widehat{\psi} \|^2 \Big ) + \frac{\lambda}{2} \| \Delta \widehat{\psi} \|^2 + \frac{\eps}{2} \| \nabla \widehat{\phi} \|^2,
\end{aligned} 
\end{equation}
where we have used Young's inequality and the following:
\[
\| \hat{\phi} - \hat{\phi}_\Omega \|^2 = \int_\Omega \nabla \mathcal{N}(\hat{\phi} -\hat \phi_\Omega) \cdot \nabla \hat{\phi} \, dx \leq 
\|\nabla \mathcal{N}(\hat{\phi} -\hat \phi_\Omega)  \|\| \nabla \hat{\phi} \|.
\]
Applying the elliptic regularity estimate \eqref{inter:1} and the Gronwall inequality leads to \eqref{cts:dep}.

\subsection{Obstacle potential}\label{sec:obs}
The well-posedness of \eqref{CHSH} with the obstacle potential \eqref{obs} follows along similar lines of argument as in the proof for the logarithmic potential \eqref{log}.  {Thus, let us just} outline the essential modifications. In place of \eqref{Pi}{,} we set 
\begin{align}\label{Pi:obs}
\Pi_N(s) = \begin{cases}
0 & \text{ if } s \geq 0, \\
\tfrac{N}{{4!}} s^4 & \text{ if } s \leq 0.
\end{cases}
\end{align}
Then, it is clear that the lower bound \eqref{Pi:lb} is fulfilled, and by setting
\[
F_0^N(r,s) = \Pi_N \big ( \tfrac{1}{2}(1+r-s) \big ) + \Pi_N \big ( \tfrac{1}{2}(1-r-s) \big) + \Pi_N(s) \quad {\forall r,s \in \erre},
\] 
we see that \eqref{log:app} is also fulfilled with constants independent of $N \in \mathbb{N}$. Furthermore, for fixed $u \in (0,1)$ we can find a constant $C_u > 0$ such that $C_u u > 1$. Then, for any $s \leq 0$ we deduce that the function $f(s) = C_u s^3 (s-u) - |s|^3$ is non-negative and consequently we find an analogue to \eqref{MZ:ineq} 
\begin{align}\label{MZ:ineq:obs}
|\pi_N(s)| \leq C\pi_N(s)(s-u) \quad \forall s \in \RRR,
\end{align}
where the positive constant $C$ is independent of $N$.

Analogous to the proof for the logarithmic potential, we obtain from the approximate system \eqref{CHSH:approx}, now with $F_0^N$ defined as above, the uniform estimates \eqref{mass:cons} and \eqref{est:2}. Then, testing \eqref{app:2} with $\phi_N - \inn{\phi_0}_\Omega$ and \eqref{app:4} with $\psi_N - {\nu}$ we obtain from \eqref{F0_der} and \eqref{MZ:ineq:obs} that 
\begin{align*}
& \| \pi_N \big ( \tfrac{1}{2}(1+\phi_N - \psi_N) \big ) \|_{L^1} + \| \pi_N \big ( \tfrac{1}{2}(1-\phi_N - \psi_N) \big )\|_{L^1} + \| \pi_N \big ( \psi_N \big ) \|_{L^1} \\
& \quad \leq C \int_\Omega F_{0,\phi}^N(\phi_N, \psi_N)(\phi_N - \inn{\phi_0}_\Omega) + F_{0,\psi}^N(\phi_N, \psi_N)(\psi_N - {\nu}) \, dx \\
& \quad \leq C \big ( 1 + \| \nabla {\mu}_N \| + \| z_N \| \big ),
\end{align*}
which in turn leads to the uniform estimates \eqref{est:3} and \eqref{est:4}. Similarly, by the convexity of $F_0^N$ we obtain the regularity estimate \eqref{est:5}, as well as the remaining uniform estimates \eqref{est:6} and \eqref{est:7}.  It remains to show that the limit $(\phi, \psi)$ of $(\phi_N, \psi_N)$, along a non-relabelled subsequence, satisfies $(\phi, \psi) \in \mathcal{K}$ for a.e.~$(x,t) \in Q$. From the explicit formula in \eqref{Pi:obs} we infer that
\begin{align*}
\int_Q \Pi_N(\psi_N) \, dx \, dt \geq \int_{\{(x,t) \in Q \, : \, \psi_N(x,t) < 0 \}} \frac{N}{4!} |\psi_N|^4 \, dx \, dt = \frac{N}{4!} \int_{Q} (-\psi_N)_{+}^4 \, dx \,dt,
\end{align*}
and so by \eqref{est:2} we have, similar to \eqref{est:log},
\begin{align*}
&  \| (-\psi_N)_+ \|_{L^4(Q)} + \| (- \tfrac{1}{2}(1+\phi_N - \psi_N) \big)_+ \|_{L^4(Q)} + \| (- \tfrac{1}{2}(1-\phi_N - \psi_N) \big)_+ \|_{L^4(Q)} \\
& \quad \leq C N^{-\frac{1}{4}},
\end{align*}
where the right-hand side vanishes as $N \to \infty$. Fatou's lemma then implies that the limit functions $\phi$ and $\psi$ satisfy \eqref{K:sat}, that is $(\phi, \psi) \in \mathcal{K}$ a.e.~in $Q$.

Lastly, for an obstacle-admissible test function pair $(\zeta, \eta) \in L^2(0,T;H^1(\Omega))^2$, notice that from the definition \eqref{Pi:obs} it holds that 
\[
\pi_N(\eta) = 0, \quad \pi_N \big ( \tfrac{1}{2}(1+\zeta - \eta) \big ) = 0, \quad \pi_N \big ( \tfrac{1}{2}(1-\zeta - \eta) \big ) = 0,
\]
and hence
\[
F_{0,\phi}^N(\zeta, \eta) = 0, \quad F_{0,\psi}^N(\zeta, \eta) = 0.
\]
This leads the following analogue of \eqref{lim:3:alt}:
\begin{align*}
0 & \leq \int_Q  \big (F_{1,\phi}(\phi_N, \psi_N) - {\mu}_N  \big ) (\zeta - \phi_N) 
+ \nabla (\eps \phi_N - \sigma \psi_N) \cdot \nabla (\zeta - \phi_N)\, dx \, dt \\
 & \quad + \int_Q  -\lambda \nabla \Delta \psi_N  \cdot \nabla (\eta - \psi_N) + \big (F_{1,\psi}(\phi_N, \psi_N)  - z_N \big ) (\eta - \psi_N)  \, dx \, dt \\
 & \quad + \int_Q   - \sigma \nabla \psi_N \cdot \nabla (\eta - \psi_N)  + \big (2 \lambda \omega^2 \Delta \psi_N + \lambda \omega^4 (\psi_N-\tfrac{1}{2}) \big ) (\eta - \psi_N)  \, dx \, dt,
\end{align*}
and with the compactness assertions \eqref{compact} we find that in the limit $N \to \infty$ the limit solution pair $(\phi, \psi)$ satisfy the variational inequality \eqref{w:3:obs}.  This completes the proof of existence. For continuous dependence and uniqueness of variational solution, the proof proceeds exactly as in Section \ref{sec:cts} and so we omit the details.

\subsection{Deep quench limit}
\subsubsection{Weak convergence}
For $\theta \in (0,1]$, we denote by $(\phi_\theta, \psi_\theta, \mu_\theta, z_\theta)$ as the variational solution to \eqref{CHSH} with logarithmic potential $F_0 = F_{\log}$ obtained through Theorem \ref{thm:wellposed}. From \eqref{compact} we see that $\phi_N \Delta \psi_N \to \phi_\theta \Delta \psi_\theta$ strongly in $L^1(0,T;L^1(\Omega))$, and hence it holds that for a.e.~$t \in (0,T)$,
\[
\int_\Omega \phi_N(t) \Delta \psi_N(t) \, dx \to \int_\Omega \phi_\theta(t) \Delta \psi_\theta(t) \, dx.
\]
Then, we revisit \eqref{est:1} and find that after neglecting the non-negative $F_{\log}^N$, employing the boundedness of $F_1$, and performing an integration by parts, where for arbitrary $t \in (0,T]$,  
\begin{equation}\label{E:lb:unif}
\begin{aligned}
& \int_\Omega \frac{\eps}{2} |\nabla \phi_N(t)|^2 + \frac{\lambda}{2} |(\Delta + \omega^2) (\psi_N(t) - \tfrac{1}{2})|^2  + \sigma \phi_N(t) \Delta \psi_N(t) \, dx \\
& \qquad + \int_0^t \| \nabla \mu_N \|^2 + \| z_N \|^2 \, dt  \leq C \big ( \| \phi_0 \|_{H^1} , \| \psi_0 \|_{H^2} \big )
\end{aligned}
\end{equation} 
with a positive constant $C$ independent of $\theta \in (0,1]$. Invoking the compactness properties listed in \eqref{compact} and the weak lower semicontinuity of the norms we deduce from \eqref{E:lb:unif} that for a.e.~$t \in (0,T)$,
\begin{equation}\label{Deep:est}
\begin{aligned}
& \int_\Omega \frac{\eps}{2} |\nabla \phi_\theta(t)|^2 + \frac{\lambda}{2} |(\Delta + \omega^2) (\psi_\theta(t) - \tfrac{1}{2})|^2  + \sigma \phi_\theta(t) \Delta \psi_\theta(t) \, dx \\
& \qquad + \int_0^t \| \nabla \mu_\theta \|^2 + \| z_\theta \|^2 \, dt  \leq C \big ( \| \phi_0 \|_{H^1} , \| \psi_0 \|_{H^2} \big ).
\end{aligned}
\end{equation}
Together with Young's inequality and the property $(\phi_\theta, \psi_\theta) \in \mathcal{K}$ a.e.~in $Q$, we infer from \eqref{Deep:est} that
\begin{equation}\label{Deep:est:2}
\begin{aligned}
& \| \phi_\theta \|_{L^\infty(0,T;H^1)}^2 + \|  \psi_\theta \|_{L^\infty(0,T;H^2)}^2 + \| \nabla \mu_\theta \|_{L^2(Q)}^2 \\
& \quad + \| z_\theta \|_{L^2(Q)}^2 + \| \pd_t \psi_\theta \|_{L^2(Q)}^2 + \| \pd_t \phi_\theta \|_{L^2(0,T;(H^1)^*)}^2 \leq C,
\end{aligned}
\end{equation}
with a positive constant $C$ independent of $\theta \in (0,1]$, where the uniform estimates on the time derivatives are inferred from a comparison of terms in \eqref{w:1} and \eqref{w:2}.

Next, in \eqref{w:3} for $(\phi_\theta, \psi_\theta, \mu_\theta, z_\theta)$ we consider $\zeta = \inn{\phi_\theta}_\Omega = \inn{\phi_0}_\Omega$ and $\eta = \nu = \frac{1}{2}(1 - |\inn{\phi_0}_\Omega|)$, integrating by parts and employing the boundedness property for $(\phi_\theta, \psi_\theta)$, $F_{1,\phi}(\phi_\theta, \psi_\theta)$ and $F_{1,\psi}(\phi_\theta, \psi_\theta)$ leads to 
\begin{equation}\label{Deep:est:3}
\begin{aligned}
& \int_\Omega F_{\log,\phi}(\phi_\theta, \psi_\theta)(\phi_\theta - \inn{\phi_\theta}_\Omega)  + F_{\log,\psi}(\phi_\theta, \psi_\theta) (\psi_\theta - \nu) \, dx \\
& \quad \leq C \big (1 +  \| \nabla \mu_\theta \| \| \nabla \phi_\theta \| + \| z_\theta \| + \| \nabla \psi_\theta \|^2 + \| \Delta \psi_\theta \|^2 + \| \nabla \phi_\theta \| \| \nabla \psi_\theta \| \big ) \\
& \quad \leq C \big ( 1 + \| \nabla \mu_\theta \| + \| z_\theta \| \big ).
\end{aligned}
\end{equation}
Invoking the analogue of \eqref{MZ:ineq} for $\Pi(s) = s \ln s$ and $\pi(s) = \Pi'(s) = 1 + \ln(s)$, cf.~\cite[Prop.~A.1]{Miran}: there exist constants $C_1 > 0$ and $C_2 \geq 0$ depending on $u \in (0,1)$ such that
\[
|\Pi(s)| + |\pi(s)| \leq C_1 \pi(s)(s- u) + C_2 \quad \forall s \in (0,1),
\]
we may deduce from \eqref{Deep:est:3} that
\begin{equation}\label{Deep:est:5}
\begin{aligned}
& \| F_{\log}(\phi_\theta, \psi_\theta) \|_{L^2(0,T;L^1)}  \\
& \quad + \| F_{\log,\phi}(\phi_\theta, \psi_\theta) \|_{L^2(0,T;L^1)} + \| F_{\log,\psi}(\phi_\theta, \psi_\theta) \|_{L^2(0,T;L^1)} \leq C.
\end{aligned}
\end{equation}
Then, integrating \eqref{mu:equ:ae} over $\Omega$ yields
\[
| \inn{\mu_\theta}_\Omega| \leq C \| F_{\log,\phi}(\phi_\theta, \psi_\theta) \|_{L^1} + C\| F_{1,\phi}(\phi_\theta, \psi_\theta) \|_{L^1},
\]
and by \eqref{Deep:est:5} we deduce that $\inn{\mu_\theta}_\Omega$ is uniformly bounded in $L^2(0,T)$.  Hence, by the Poincar\'e inequality we obtain
\begin{align}\label{Deep:est:6}
\| \mu_\theta \|_{L^2(Q)} \leq C.
\end{align}
The uniform estimates \eqref{Deep:est:2} and \eqref{Deep:est:6} allows us to deduce, along a non-relabelled subsequence $\theta \to 0$, the existence of limit functions $(\phi_*, \psi_*, \mu_*, z_*)$ such that
\begin{equation}\label{deep:compact}
\begin{aligned}
\phi_\theta \to \phi_* & \text{ weakly* in } L^\infty(0,T;H^1(\Omega)) \cap H^1(0,T;H^1(\Omega)^*), \\
\phi_\theta \to \phi_* & \text{ strongly in } C^0([0,T];L^s(\Omega))  \text{ and a.e.~in } Q, \\
\psi_\theta \to \psi_* & \text{ weakly* in } L^\infty(0,T;H^2_{\bnn}(\Omega)) \cap H^1(0,T;L^2(\Omega)), \\
\psi_\theta \to \psi_* & \text{ strongly in } C^0([0,T];W^{1,s}(\Omega)) \text{ and a.e.~in }Q, \\
\mu_\theta \to \mu_* & \text{ weakly in } L^2(0,T;H^1(\Omega)), \\
z_\theta \to z_* & \text{ weakly in } L^2(Q),
\end{aligned}
\end{equation}
for any $s < \infty$ in two dimensions and any $s \in [2,6)$ in three dimensions, along with $(\phi_*, \psi_*) \in \mathcal{K}$ for a.e.~$(x,t) \in Q$ as well as attainment of the initial conditions $\phi_*(0) = \phi_0$ and $\psi_*(0) = \psi_0$. Passing to the limit in \eqref{w:1}-\eqref{w:2} for $(\phi_\theta, \psi_\theta, \mu_\theta, z_\theta)$ yields the analogous identities for $(\phi_*, \psi_*, \mu_*, z_*)$. Next, we consider an arbitrary obstacle-admissible test function pair $(\zeta, \eta) \in L^2(0,T;H^1(\Omega) \times H^2_{\bnn}(\Omega))$ in \eqref{w:3:alt:weak} for $(\phi_\theta, \psi_\theta, \mu_\theta, z_\theta)$ expressed as
\begin{align*}
0 & \leq \int_Q  \big ( F_{1,\phi}(\phi_\theta, \psi_\theta) - \mu_\theta  \big ) (\zeta - \phi_\theta) + \nabla (\eps \phi_\theta - \sigma \psi_\theta) \cdot \nabla (\zeta - \phi_\theta) \, dx \, dt \\
 & \quad + \int_Q   \Delta \psi_\theta \Delta (\eta - \psi_\theta)  - \sigma \nabla \phi_\theta  \cdot \nabla (\eta - \psi_\theta) + \big (F_{1,\psi}(\phi_\theta, \psi_\theta) - z_\theta \big ) (\eta - \psi_\theta) \, dx \, dt \\
 & \quad + \int_Q \big (2 \lambda \omega^2 \Delta \psi_\theta + \lambda \omega^4 (\psi_\theta-\tfrac{1}{2}) \big ) (\eta - \psi_\theta)  + F_{\log}(\zeta, \eta) - F_{\log}(\phi_\theta, \psi_\theta) \, dx \, dt.
\end{align*}
Due to the definition of $F_0$ in \eqref{log} and the continuity of $F_{\log}$ over $\mathcal{K}$ we see that 
\[
\Big |\int_Q F_{\log}(\zeta, \eta) - F_{\log}(\phi_\theta, \psi_\theta) \, dx \, dt \Big | \leq C \theta \to 0
\]
as $\theta \to 0$. Employing the compactness assertions in \eqref{deep:compact} and weak lower semicontinuity of the Bochner norms, we obtain as $\theta \to 0$
\begin{align*}
0 & \leq \int_Q  \big ( F_{1,\phi}(\phi_*, \psi_*) - \mu_*  \big ) (\zeta - \phi_*) + \nabla (\eps \phi_* - \sigma \psi_*) \cdot \nabla (\zeta - \phi_*) \, dx \, dt \\
 & \quad + \int_Q   \Delta \psi_* \Delta (\eta - \psi_*)  - \sigma \nabla \phi_*  \cdot \nabla (\eta - \psi_*) + \big (F_{1,\psi}(\phi_*, \psi_*) - z_*\big ) (\eta - \psi_*) \, dx \, dt \\
 & \quad + \int_Q \big (2 \lambda \omega^2 \Delta \psi_* + \lambda \omega^4 (\psi_* -\tfrac{1}{2}) \big ) (\eta - \psi_*)  \, dx \, dt.
\end{align*}
Via a similar calculation to the proof of uniqueness in Section \ref{sec:cts} we find that $(\phi_*, \psi_*, \mu_*, z_*)$ is the unique variational solution to \eqref{CHSH} with the obstacle potential \eqref{obs}, whence in fact $\Delta \psi_* \in L^2(0,T;H^1(\Omega))$. Furthermore, by uniqueness of variational solutions we infer that the whole sequence $\{(\phi_\theta, \psi_\theta, \mu_\theta, z_\theta)\}_{\theta \in (0,1]}$ converges, and by the density of $L^2(0,T;H^2_{\bnn}(\Omega))$ in $L^2(0,T;H^1(\Omega))$ we recover \eqref{w:3:obs} holding for arbitrary obstacle-admissible test function pair $(\zeta, \eta) \in L^2(0,T;H^1(\Omega))^2$.

\subsubsection{Convergence rate}
Let $(\phi_\theta, \psi_\theta, {\mu}_\theta, z_\theta)$ be the variational solution to \eqref{CHSH} with the logarithmic potential \eqref{log} associated with the initial conditions $(\phi_0, \psi_0)$, and let $(\phi_*, \psi_*, {\mu}_*, z_*)$ be the variational solution to \eqref{CHSH} with the obstacle potential \eqref{obs} associated with  the same initial conditions.  We denote by $\hat{\phi}_\theta := \phi_* - \phi_\theta$, $\hat{\psi}_\theta := \psi_* - \psi_\theta $, $\hat{{\mu}}_\theta := {\mu}_* - {\mu}_\theta$, and $\hat{z}_\theta := z_* - z_\theta$  the differences between variational solutions and incidentally remark that $\hat{\phi}_\theta $ is of zero mean value as $\inn{\hat{\phi}_\theta}_\Omega = \inn{\phi_* - \phi_\theta}_\Omega =\inn{\phi_0}_\Omega- \inn{\phi_0}_\Omega =0$. 

Similar to the proof of uniqueness, we consider the variational inequality \eqref{w:3:obs} for $(\phi_*, \psi_*, {\mu}_*, z_*)$ with test function pair $(\zeta, \eta) = (\phi_\theta, \psi_\theta)$ which we note is obstacle-admissible, as well as the alternate variational inequality \eqref{w:3:alt:weak} for $(\phi_\theta, \psi_\theta, {\mu}_\theta, z_\theta)$ with test function pair $(\zeta, \eta) = (\phi_*, \psi_*)$ which satisfies $F_{\log}(\phi_*, \psi_*) \in L^1(Q)$. Then, upon summing, we obtain
\begin{equation*}
\begin{aligned} 
& \int_\Omega \big (F_{1,\phi}(\phi_*, \psi_*) - F_{1,\phi}(\phi_\theta, \psi_\theta) - \hat{\mu}_\theta \big ) \hat{\phi}_\theta 
+ \nabla (\eps \hat{\phi}_\theta- \sigma \hat{\psi}_\theta) \cdot \nabla \hat{\phi}_\theta  \, dx \\
& \qquad + \int_\Omega - \lambda \nabla \Delta \hat{\psi}_\theta \cdot \nabla \hat{\psi}_\theta + \big ( F_{1,\psi}(\phi_*, \psi_*) - F_{1,\psi}(\phi_\theta, \psi_\theta)  - \hat{z}_\theta \big )\hat{\psi}_\theta 
- \sigma \nabla \hat{\phi}_\theta  \cdot \nabla \hat{\psi}_\theta \\
& \qquad + \int_\Omega (2 \lambda \omega^2 \Delta \hat{\psi}_\theta + \lambda \omega^4 \hat{\psi}_\theta \big ) \hat{\psi}_\theta  \, dx \\
& \quad \leq \int_\Omega F_{\log}(\phi_*, \psi_*) - F_{\log}(\phi_\theta, \psi_\theta) \, dx \leq C \theta,
\end{aligned}
\end{equation*}
where for the right-hand side we have used the definition \eqref{log} and the continuity of $F_{\log}$ over $\mathcal{K}$. Analogously, from the difference between \eqref{w:1} and \eqref{w:2} for $(\phi_*, \psi_*, {\mu}_*, z_*)$ and $(\phi_\theta, \psi_\theta, {\mu}_\theta, z_\theta)$, and choosing $u = \mathcal{N}(\hat{\phi}_\theta)$ and $v = \hat{\psi}_\theta$, we obtain
\begin{align*}
0 = {\frac{1}{2}\frac{d}{dt} } \| \nabla \mathcal{N} (\hat{\phi}_\theta) \|^2 + \int_\Omega \hat{{\mu}}_\theta \hat{\psi}_\theta \, dx , \quad  0= {\frac{1}{2}\frac{d}{dt} } \| \hat{\psi}_\theta \|^2 + \int_\Omega \hat{z}_\theta \hat{\phi}_\theta \, dx.
\end{align*}
Then, upon adding these inequalities we deduce similar to Section \ref{sec:cts}, 
\begin{align*}
& {\frac{1}{2}\frac{d}{dt} } \Big ( \| \nabla \mathcal{N}( \hat{\phi}_\theta) \|^2 + \| \hat{\psi}_\theta \|^2 \Big ) + \frac{1}{4} \| \nabla \hat{\phi}_\theta \|^2 + \frac{\lambda}{2} \| \Delta \hat{\psi}_\theta \|^2  \\ & \quad 
\leq C \Big ( \| \nabla \mathcal{N}( \hat{\phi}_\theta) \|^2 + \| \hat{\psi}_\theta \|^2 \Big ) + C \theta
\end{align*}
with constants independent of $\theta$. Application of the Gronwall inequality and the elliptic regularity estimate \eqref{inter:1} leads to \eqref{rate}.

\section{Numerical discretization}\label{sec:numerics}

In this section we introduce a finite element approximation of the system
\eqref{CHSH} with the obstacle potential \eqref{obs} based on a
suitable variational formulation that is then discretized with piecewise linear
finite elements. We establish an unconditional stability result, introduce an
iterative solution method for implementation and then present several numerical simulations, which
exhibit a wide range of complex pattern formations.

\subsection{Weak formulation}
For notational convenience, we let $(\cdot,\cdot)$ denote the $L^{2}$--inner 
product on $\Omega$, and define
\begin{equation*} 
K:=\{\,(\eta_1,\eta_2) \in [H^1(\Omega)]^2:
(\eta_1(x),\eta_2(x)) \in{\cal K}
\mbox{ a.e. in } \Omega \,\}.
\end{equation*}
Furthermore, we introduce an auxiliary variable $q = -\Delta \psi$ and consider the following variational formulation for \eqref{CHSH} with the obstacle 
potential \eqref{obs}. Find $(\phi,\psi) \in L^2(0,T;K) \cap (H^1(0,T;(H^1(\Omega))^*))^2$, $z \in L^2(0,T;L^2(\Omega))$ and $(\mu, q) \in (L^2(0,T;H^1(\Omega)))^2$ such that for almost all $t \in (0,T)$
\begin{subequations}\label{weak:CHSH:obs:V2}
\begin{alignat}{2}
0 & = \inn{\pd_t \phi, u} + (\nabla {\mu}, \nabla u), \\
0 & = \inn{\pd_t \psi, v} + (z, v), \\
 0 & \leq (\nabla \phi - \sigma \nabla \psi, \nabla (\eta - \phi)) +  ( F_{1,\phi}(\phi, \psi) - {\mu}, \eta - \phi  ) \\
\notag & \quad +  (\lambda \omega^4 (\psi-\tfrac{1}{2}) - z + F_{1,\psi}(\phi, \psi) , \zeta - \psi  ) \\ 
\notag & 
\quad + (\lambda \nabla q - 2 \lambda \omega^2 \nabla \psi  -\sigma \nabla \phi, \nabla (\zeta - \psi)), \\
0 & =  (\nabla \psi, \nabla \theta) - (q, \theta),
\end{alignat}
\end{subequations}
for all $(\eta, \zeta) \in K$ and 
$(u, v, \theta) \in H^1(\Omega) \times L^2(\Omega) \times H^1(\Omega)$. This weak formulation can be derived from Definition \ref{defn:var} with the help of the new variable $q = - \Delta \psi$. For the numerical approximation we prefer the formulation \eqref{weak:CHSH:obs:V2} because this weak formulation can be solved with piecewise linear continuous functions at the discrete level.

\subsection{Finite element approximation}
We assume that $\Omega$ is a polyhedral domain and
let $\mathcal{T}_{h}$ be a regular triangulation of $\Omega$ into disjoint open simplices. Associated with $\mathcal{T}_h$ is the piecewise linear finite element space
\begin{align*}
S^{h} = \left \{\zeta \in C^{0}(\overline\Omega) : \,  \zeta_{\vert_{o}} \in P_{1}(o) \, \forall o \in \mathcal{T}_{h} \right \},
\end{align*}
where we denote by $P_{1}(o)$ the set of all affine linear functions on $o$,
cf.\ \cite{Ciarlet78}. 
In addition, we define 
\[
K^h = K \cap S^h,
\]
and let $(\cdot,\cdot)^{h}$ be the usual mass lumped $L^{2}$--inner product on
$\Omega$ associated with $\mathcal{T}_{h}$.
Finally, $\tau$ denotes a chosen uniform time step size.

For what follows we assume that $F_1$ can be decomposed into
$F_1 = F_1^+ + F_1^-$, with $F_1^+$ being convex and $F_1^-$ being concave.
For example, in the case \eqref{poly} we set
\begin{subequations} \label{eq:F1pm}
\begin{align}
F_1^+(\phi, \psi) & = \frac{C_F}{2} (\phi^2 + \psi^2), \\
F_1^-(\phi, \psi) & = - \frac{\alpha}{2} \phi^2 - \frac{g}{3} (\psi-\tfrac{1}{2})^3 - \frac{\gamma}{2} (\psi-\tfrac{1}{2})^2 + \frac{\delta}{2} \phi^2 (\psi-\tfrac{1}{2}) - \frac{C_F}{2} (\phi^2 + \psi^2),
\end{align}
\end{subequations}
where $C_F \geq 0$ is a constant chosen sufficiently large. In fact, choosing 
\begin{equation} \label{eq:CF}
C_F = \max \left (0,  \frac{3}{2}|\delta| - \alpha, |\delta| + |g| - \gamma \right )
\end{equation}
ensures that the Hessian
\[
H = \begin{pmatrix} - \alpha - C_F + \delta (\psi - \tfrac{1}{2}) & \delta \phi \\ \delta \phi & - \gamma - 2 g(\psi - \tfrac{1}{2}) - C_F \end{pmatrix} = \begin{pmatrix} H_{11} & H_{12} \\ H_{12} & H_{22} \end{pmatrix}
\]
of $F_1^-$ satisfies 
$H_{11} \leq - |\delta|$, $H_{12} \leq - |\delta|$, 
$H_{11} H_{22} \geq \delta^2$,  
and thus $\tr H \leq 0$ and $\det H \geq 0$ in $\mathcal{K}$. It follows that
$F_1^-$ with the choice \eqref{eq:CF} is indeed concave in $\mathcal{K}$.

Then our finite element approximation of \eqref{weak:CHSH:obs:V2} is given as
follows. Let $(\phi^0_h,\psi^0_h) \in K^h$. Then, for $n \geq 0$ and
given $(\phi^n_h,\psi^n_h) \in K^h$, find
$(\phi^{n+1}_h,\psi^{n+1}_h) \in K^h$ and 
$(\mu^{n+1}_h, z^{n+1}_h, q^{n+1}_h) \in [S^h]^3$ such that
\begin{subequations}\label{eq:fea}
\begin{alignat}{2}
0 & = \tfrac{1}{\tau}(\phi^{n+1}_h - \phi^n_h, u_h)^h + (\nabla {\mu}^{n+1}_h, \nabla u_h), \label{d:1} \\
0 & = \tfrac{1}{\tau}(\psi^{n+1}_h - \psi^n_h, v_h)^h + (z^{n+1}_h, v_h)^h,  \label{d:2} \\
\label{d:3} 0 & \leq (\eps \nabla \phi^{n+1}_h - \tfrac{\sigma}{2} \nabla (\psi^{n+1}_h + \psi^n_h), \nabla (\eta_h - \phi^{n+1}_h)) \\
\notag & \quad +  (F_{1,\phi}^+(\phi^{n+1}_h, \psi^{n+1}_h) + F_{1,\phi}^-(\phi^n_h, \psi^n_h) - {\mu}^{n+1}_h, \eta_h - \phi^{n+1}_h  )^h \\
\notag & \quad +  ( \lambda \omega^4 (\psi^{n+1}_h - \tfrac{1}{2}) - z^{n+1}_h + F_{1,\psi}^+(\phi^{n+1}_h, \psi^{n+1}_h) + F_{1,\psi}^-(\phi^n_h, \psi^n_h), \zeta_h - \psi^{n+1}_h  )^h \\
\notag & \quad + (\lambda \nabla q^{n+1}_h - 2 \lambda \omega^2 \nabla \psi^{n+1}_h - \tfrac{\sigma}{2} \nabla (\phi^{n+1}_h + \phi^{n}_h), \nabla (\zeta_h - \psi^{n+1}_h)), \\
0 & = (\nabla \psi^{n+1}_h, \nabla \theta_h) - (q^{n+1}_h, \theta_h)^h,\label{d:4}
\end{alignat}
\end{subequations}
for all $(\eta_h, \zeta_h) \in K^h$ and $(u_h, v_h, \theta_h) \in [S^h]^3$.

The convex-concave splitting of $F_1$ allows for an unconditional
stability estimate for the discrete energy
\begin{align} \label{eq:En}
E^n = \frac{\eps}{2} \| \nabla \phi^n_h \|^2 + (F_{1}(\phi^n_h, \psi^n_h), 1)^h - \sigma (\phi^n_h, q^n_h)^h + \frac{\lambda}{2} \| \omega^2 (\psi^{n}_h - \tfrac{1}{2})  - q^n_h \|_h^2,
\end{align}
where $\|u_h\|_h = [(u, u)^h]^\frac12$, and where $q^0_h \in S^h$ is defined by
\eqref{d:4} with $n+1$ replaced by $0$.
Note that since $q^n_h$ approximates $q = -\Delta \psi$, 
the energy \eqref{eq:En} is a discrete analogue of \eqref{eq:Energy}. 

\begin{thm}
Let $(\phi^{n+1}_h,\psi^{n+1}_h,\mu^{n+1}_h, z^{n+1}_h, q^{n+1}_h)$ be a 
solution to \eqref{eq:fea}. Then it holds that
\begin{equation} \label{eq:thmstab}
E^{n+1} + \tau \| \nabla {\mu}^{n+1}_h \|^2 + \tau \| z^{n+1}_h \|^2_h \leq E^n.
\end{equation}
Moreover, there exists a constant $E_{\min} \in \mathbb R$ such that
\[
E^{n} \geq E_{\min}+  \frac{\lambda}{8} \| q_h^n \|_h^2 + \frac{\eps}{2} \| \nabla \phi_h^n \|^2 \quad\text{ for all } n\geq 0.
\]
\end{thm}
\begin{proof}
Choosing $u_h = \tau {\mu}^{n+1}_h$ in \eqref{d:1}, 
$v_h = \tau z^{n+1}_h$ in \eqref{d:2} and
$(\zeta_h, \eta_h) = (\phi^n_h, \psi^n_h)$ in \eqref{d:3}, 
we obtain upon summing
\begin{equation}\label{stab:1}
\begin{aligned}
0 & \geq (\eps \nabla \phi^{n+1}_h, \nabla (\phi^{n+1}_h - \phi^n_h)) +  (F_{1,\phi}^+(\phi^{n+1}_h, \psi^{n+1}_h)  + F_{1,\phi}^-(\phi^n_h, \psi^n_h) , \phi^{n+1}_h - \phi^n_h  )^h \\
& \quad + \tau  \| \nabla {\mu}^{n+1}_h \|^2 - \tfrac{\sigma}{2}( \nabla (\psi^{n+1}_h + \psi^n_h), \nabla (\phi^{n+1}_h - \phi^n_h)) \\
& \quad +  (\lambda \omega^4 (\psi^{n+1}_h - \tfrac{1}{2})  + F_{1,\psi}^+(\phi^{n+1}_h, \psi^{n+1}_h) + F_{1,\psi}^-(\phi^n_h, \psi^n_h), \psi^{n+1}_h - \psi^n_h )^h  \\
& \quad + \tau \| z^{n+1}_h \|_h^2+ (\lambda \nabla q^{n+1}_h - 2 \lambda \omega^2 \nabla \psi^{n+1}_h - \tfrac{\sigma}{2} \nabla (\phi^{n+1}_h + \phi^{n}_h), \nabla (\psi^{n+1}_h - \psi^n_h)).
\end{aligned}
\end{equation}
Furthermore, we see that 
\begin{equation} \label{eq:stab1a}
\begin{aligned}
&  \tfrac{\sigma}{2} (\nabla (\psi^{n+1}_h + \psi^n_h), \nabla (\phi^{n+1}_h - \phi^n_h)) + \tfrac{\sigma}{2} (\nabla (\phi^{n+1}_h+ \phi^n_h), \nabla (\psi^{n+1}_h - \psi^n_h))  \\
 & \quad = \sigma (\nabla \psi^{n+1}_h, \nabla \phi^{n+1}_h) - \sigma (\nabla \psi^n_h, \nabla \phi^n_h) = 
\sigma (\phi^{n+1}_h, q^{n+1}_h)^h - \sigma (\phi^n_h, q^n_h)^h,
\end{aligned}
\end{equation}
while by the convexity of $F_1^+$ and concavity of $F_1^-$ we have that
\begin{equation} \label{eq:stab1b}
\begin{aligned}
&  (F_{1,\phi}^+(\phi^{n+1}_h, \psi^{n+1}_h), \phi^{n+1}_h - \phi^n_h  )^h  +   (F_{1,\psi}^+(\phi^{n+1}_h, \psi^{n+1}_h), \psi^{n+1}_h- \psi^n_h  )^h  \\
& \qquad +  (F_{1,\phi}^-(\phi^{n}_h, \psi^{n}_h), \phi^{n+1}_h - \phi^n_h  )^h  +   (F_{1,\psi}^-(\phi^{n}_h, \psi^{n}_h), \psi^{n+1}_h - \psi^n_h  )^h  \\
& \quad \geq (F_{1}^+(\phi^{n+1}_h, \psi^{n+1}_h) , 1)^h - (F_1^+(\phi^n_h, \psi^n_h), 1)^h \\
& \qquad + (F_{1}^-(\phi^{n+1}_h, \psi^{n+1}_h) , 1)^h - (F_1^-(\phi^n_h, \psi^n_h), 1)^h  \\
& \quad 
=  (F_{1}(\phi^{n+1}_h, \psi^{n+1}_h) - F_1(\phi^n_h, \psi^n_h), 1 )^h.
\end{aligned}
\end{equation}
Combining \eqref{stab:1}, \eqref{eq:stab1a} and \eqref{eq:stab1b} yields
\begin{equation}\label{eq:stab1c}
\begin{aligned}
0 & \geq \frac{\eps}{2} \| \nabla \phi^{n+1}_h \|^2 - \frac{\eps}{2} \| \nabla \phi^n_h \|^2 + \frac{\eps}{2} \| \nabla ( \phi^{n+1}_h - \phi^n_h) \|^2  
\\
& \quad +   (F_{1}(\phi^{n+1}_h, \psi^{n+1}_h) - F_1(\phi^n_h, \psi^n_h), 1 )^h + \tau \| \nabla {\mu}^{n+1}_h \|^2 + \tau \| z^{n+1}_h \|_h^2 \\
& \quad  
- \sigma (\phi^{n+1}_h, q^{n+1}_h)^h + \sigma (\phi^n_h, q^n_h)^h +  (\lambda \omega^4 (\psi^{n+1}_h - \tfrac{1}{2})  , \psi^{n+1}_h - \psi^n_h )^h \\ & \quad 
+ (\lambda \nabla q^{n+1}_h - 2 \lambda \omega^2 \nabla \psi^{n+1}_h , \nabla (\psi^{n+1}_h - \psi^n_h)).
\end{aligned}
\end{equation}
Moreover, taking the difference between \eqref{d:4} at instance $n+1$ and $n$, and choosing $\theta = - \lambda q^{n+1}_h$ in the subsequent difference yields
\begin{align*}\label{stab:2}
(\nabla (\psi^{n+1}_h - \psi^n_h), \lambda \nabla q^{n+1}_h) = 
\lambda (q^{n+1}_h - q^n_h, q^{n+1}_h)^h 
= \frac\lambda2 \left( \| q^{n+1}_h \|_h^2 - \| q^n_h \|_h^2 
+ \| q^{n+1}_h - q^n_h \|_h^2\right).
\end{align*}
Hence, together with \eqref{eq:stab1c}, we obtain
\begin{equation}\label{uncond:stab:1}
\begin{aligned}
0 & \geq \frac{\eps}{2} \| \nabla \phi^{n+1}_h \|^2 - \frac{\eps}{2} \| \nabla \phi^n_h \|^2 + \frac{\eps}{2} \| \nabla ( \phi^{n+1}_h - \phi^n_h) \|^2
\\
& \quad+  (F_{1}(\phi^{n+1}_h, \psi^{n+1}_h) - F_1(\phi^n_h, \psi^n_h), 1 )^h \\
& \quad + \tau \| \nabla {\mu}^{n+1}_h \|^2 + \tau \| z^{n+1}_h \|_h^2 
- \sigma (\phi^{n+1}_h, q^{n+1}_h)^h + \sigma (\phi^n_h, q^n_h)^h \\
& \quad + \frac{\lambda}{2} \Big (\| q^{n+1}_h \|_h^2 - 2 \omega^2 \| \nabla \psi^{n+1}_h \|^2 +  \omega^4 \| \psi^{n+1}_h-\tfrac{1}{2} \|_h^2 \Big )\\
& \quad - \frac{\lambda}{2} \Big (\| q^n_h \|_h^2 - 2 \omega^2 \| \nabla \psi^n_h \|^2 +  \omega^4 \| \psi^{n}_h - \tfrac{1}{2} \|_h^2 \Big ) \\
& \quad + \frac{\lambda}{2} \Big ( \| q^{n+1}_h - q^n_h \|_h^2 - 2 \omega^2 \| \nabla (\psi^{n+1}_h - \psi^n_h ) \|^2 + \omega^4 \| \psi^{n+1}_h - \psi^n_h \|_h^2 \Big ).
\end{aligned}
\end{equation}
Using the relation 
$\| \nabla \psi^k_h \|^2 = (q^k_h, \psi^{k}_h - \tfrac{1}{2} )^h$
for $k=n$ and $k=n+1$, recall \eqref{d:4}, we see that 
\begin{align*}
 \| q^k_h \|_h^2 - 2 \omega^2 \| \nabla \psi^k_h \|^2 
+ \omega^4 \| \psi^k_h - \tfrac12 \|_h^2 & = 
( |q^k_h|^2 - 2\omega^2 q^k_h (\psi^{k}_h - \tfrac{1}{2}) 
+ \omega^4 |\psi^{k}_h - \tfrac{1}{2} |^2, 1 )^h \\
& = \| \omega^2 (\psi^{k}_h - \tfrac{1}{2})  - q^k_h \|_h^2 
\end{align*}
for $k \in \{n, n+1\}$, and similarly
\begin{align*}
& \| q^{n+1}_h - q^n_h \|_h^2 - 2 \omega^2 \| \nabla (\psi^{n+1}_h - \psi^n_h ) \|^2 + \omega^4 \| \psi^{n+1}_h - \psi^n_h \|_h^2 \\
& \quad = \| q^{n+1}_h - q^n_h \|_h^2 - 2 \omega^2 (q^{n+1}_h - q^{n}_h, \psi^{n+1}_h - \psi^{n}_h)^h + \omega^4 \| \psi^{n+1}_h - \psi^n_h \|_h^2 \\
& \quad = \| \omega^2 (\psi^{n+1}_h - \psi^n_h)  - (q^{n+1}_h - q^n_h) \|_h^2.
\end{align*}
This allows us to express \eqref{uncond:stab:1} as
\begin{align*}
0 & \geq E^{n+1} - E^n + \frac{\eps}2 \| \nabla (\phi^{n+1}_h - \phi^n_h) \|^2 
+ \tau \| \nabla \mu^{n+1}_h \|^2 + \tau \| z^{n+1}_h \|^2_h \\
& \quad + \frac{\lambda}{2}  \| \omega^2 (\psi^{n+1}_h - \psi^n_h)  - (q^{n+1}_h - q^n_h) \|_h^2,
\end{align*}
which proves the desired result \eqref{eq:thmstab}. 
The lower bound follows from the fact that $(\phi^n_h, \psi^n_h) \in K^h$, since then \begin{align*}
E^n & \geq |\Omega| \min_{\mathcal K} F_1 - \frac{2\sigma^2}\lambda |\Omega|
- \frac\lambda8\|q^n_h\|_h^2
+ \frac{\lambda}{2} \| \omega^2 (\psi^{n}_h - \tfrac{1}{2})  - q^n_h \|_h^2 + \frac{\eps}{2} \| \nabla \phi_h^n \|^2 \\ 
&
\geq |\Omega| \min_{\mathcal K} F_1 - \frac{2\sigma^2}\lambda |\Omega|
- \frac\lambda8\|q^n_h\|_h^2 + \frac{\lambda}{2} \left(
\omega^2 \| \psi^{n}_h - \tfrac{1}{2} \|_h - \| q^n_h \|_h \right)^2 + \frac{\eps}{2} \| \nabla \phi_h^n \|^2 \\ 
&
\geq |\Omega| \min_{\mathcal K} F_1 - \frac{2\sigma^2}\lambda |\Omega| 
- \frac\lambda8\|q^n_h\|_h^2 + \frac{\lambda}{2} \left(
\tfrac12 \| q^n_h \|_h^2 - 2 \omega^4 \| \psi^{n}_h - \tfrac{1}{2} \|_h^2 
\right) + \frac{\eps}{2} \| \nabla \phi_h^n \|^2 \\ & 
\geq |\Omega| \min_{\mathcal K} F_1 - \frac{2\sigma^2}\lambda |\Omega| - \frac{\lambda\omega^4}{4} |\Omega| + \frac{\lambda}{8} \| q_h^n \|_h^2 + \frac{\eps}{2} \| \nabla \phi_h^n \|^2.
\end{align*}
This completes the proof.
\end{proof}

\subsection{Iterative solution method}
Following \cite{grain} we now discuss a possible algorithm to solve the
resulting system of algebraic equations for 
$(\phi^{n+1}_h,\psi^{n+1}_h,\mu^{n+1}_h, z^{n+1}_h, q^{n+1}_h)$
arising at each time level from the finite element approximation 
\eqref{eq:fea}. To this end, let $\mathcal J$ denotes the number of nodes of
$\mathcal T^h$. Then, on 
introducing the obvious notations, the system \eqref{eq:fea} can be written
in matrix-vector form as follows. Find 
$(\Phi^{n+1}, \Psi^{n+1}) \in \mathcal K^{\mathcal J}$ and 
$W^{n+1}, Z^{n+1}, Q^{n+1} \in \mathcal R^\mathcal J$ such that
\begin{subequations} \label{eq:blockVI}
\begin{align} 
& M \Phi^{n+1} + \tau A W^{n+1} = M \Phi^n, \\
& M \Psi^{n+1} + \tau M Z^{n+1} = M \Psi^n, \\
& (\eta_h - \Phi^{n+1})^\top (\eps A \Phi^{n+1} - \tfrac\sigma2 A \Psi^{n+1} 
+  M F^+_{1,\phi}(\Phi^{n+1}, \Psi^{n+1}) - M W^{n+1})  \\ & \quad
+ (\zeta_h - \Psi^{n+1})^\top ( \lambda \omega^4 M \Psi^{n+1} 
- \tfrac\sigma2 A \Phi^{n+1} - M Z^{n+1}
+  M F^+_{1,\psi}(\Phi^{n+1}, \Psi^{n+1}) 
\nonumber \\ & \hspace{4cm}
+ \lambda A Q^{n+1} 
- 2 \lambda \omega^2 A \Psi^{n+1}) \nonumber \\ & \ \geq
(\eta_h - \Psi^{n+1})^\top R^n + (\zeta_h - \Psi^{n+1})^\top S^n \quad
\forall \ (\eta_h, \zeta_h) \in \mathcal K^{\mathcal J}, \nonumber \\
& A \Psi^{n+1} - M Q^{n+1} = 0, 
\end{align}
\end{subequations}
where $M$ and $A$ denote the lumped mass matrix and stiffness matrix, 
respectively, $R^n = \frac\sigma2 A \Psi^n -  M F_{1,\phi}^-(\Phi^n, \Psi^n)$ and 
$S^n = \frac12 \lambda \omega^4 M \underline{1} - M F_{1,\psi}^-(\Phi^n, \Psi^n) + \frac\sigma2 A \Phi^n$.

Let $A = A_D - A_L - A_L^\top$ and recall that $M$ is a diagonal matrix. 
Then we can formulate a ``Gauss--Seidel type'' iterative scheme as follows.
Given $(\Phi^{n+1,0}, \Psi^{n+1,0}) \in \mathcal K^{\mathcal J}$ and 
$W^{n+1,0}, Z^{n+1,0}, Q^{n+1,0} \in \mathcal R^\mathcal J$, for $k\geq 0$
find $(\Phi^{n+1,k+1}, \Psi^{n+1,k+1}) \in \mathcal K^{\mathcal J}$ and 
$W^{n+1,k+1}, Z^{n+1,k+1}, Q^{n+1,k+1} \in \mathcal R^\mathcal J$ such that
\begin{subequations} \label{eq:GS}
\begin{align} 
& M \Phi^{n+1,k+1} + \tau (A_D - A_L) W^{n+1,k+1} 
= M \Phi^n + \tau A_L^\top W^{n+1,k}, \\
& M \Psi^{n+1,k+1} + \tau M Z^{n+1,k+1} = M \Psi^n, \\
& (\eta_h - \Phi^{n+1,k+1})^\top (\eps (A_D - A_L) \Phi^{n+1,k+1} 
- \tfrac\sigma2 (A_D - A_L) \Psi^{n+1,k+1}  \\ & \hspace{3cm}
+  M F^+_{1,\phi}(\Phi^{n+1,k+1}, \Psi^{n+1,k+1}) - M W^{n+1,k+1}) \nonumber \\ & \quad
+ (\zeta_h - \Psi^{n+1,k+1})^\top ( \lambda \omega^4 M \Psi^{n+1,k+1} 
- \tfrac\sigma2 (A_D - A_L) \Phi^{n+1,k+1}
- M Z^{n+1,k+1} \nonumber \\ & \hspace{3cm}
+  M F^+_{1,\psi}(\Phi^{n+1,k+1}, \Psi^{n+1,k+1}) 
\nonumber \\ & \quad
+ \lambda (A_D - A_L) Q^{n+1,k+1} 
- 2 \lambda \omega^2 (A_D - A_L) \Psi^{n+1,k+1}) \nonumber \\ & \ \geq
(\eta_h - \Phi^{n+1,k+1})^\top (R^n + \eps A_L^\top \Phi^{n+1,k} 
- \tfrac\sigma2 A_L^\top \Psi^{n+1,k}) \nonumber \\ & \qquad
+ (\zeta_h - \Psi^{n+1,k+1})^\top (S^n 
- \frac\sigma2 A_L^\top \Phi^{n+1,k}
+ \lambda A_L^\top Q^{n+1,k}
- 2 \lambda \omega^2 A_L^\top \Psi^{n+1,k}) \nonumber\\ & \hspace{8cm}
\forall\ (\eta_h, \zeta_h) \in \mathcal K^{\mathcal J}, \nonumber \\
& (A_D-A_L) \Psi^{n+1,k+1} - M Q^{n+1,k+1} = A_L^\top \Psi^{n+1,k}.
\end{align}
\end{subequations}

From now on we fix our discussion to the choice \eqref{eq:F1pm}. Then
\eqref{eq:GS} can be explicitly solved for $j=1,\ldots,{\cal J}$.
To this end let 
\begin{align*}
r_1 & = M \Phi^n + \tau (A_L W^{n+1,k+1} + A_L^\top W^{n+1,k}), \\
r_2&  = R^n + \eps A_L \Phi^{n+1,k+1} + \eps A_L^\top \Phi^{n+1,k} 
- \tfrac\sigma2 A_L \Psi^{n+1,k+1}
- \tfrac\sigma2 A_L^\top \Psi^{n+1,k}), \\
r_3 & = M \Psi^n, \\
r_4 & = S^n 
- \tfrac\sigma2 A_L \Phi^{n+1,k+1}
- \tfrac\sigma2 A_L^\top \Phi^{n+1,k}
+ \lambda (A_L Q^{n+1,k+1} + A_L^T Q^{n+1,k}) \\
& \quad - 2 \lambda \omega^2 (A_L \Psi^{n+1,k+1} + A_L^\top \Psi^{n+1,k}), \\
r_5 &= A_L \Psi^{n+1,k+1} + A_L^\top \Psi^{n+1,k}.
\end{align*}
Then $(\Phi^{n+1,k+1}_j, \Psi^{n+1,k+1}_j)$ is the solution of the following problem: 
Find $(\Phi_j, \Psi_j) \in \mathcal{K}$ such that, for every $(\eta_h,\zeta_h) \in \mathcal{K}$,
\begin{align} \label{eq:VIj}
(\eta_h - \Phi_j)^\top (\alpha_{11} \Phi_j - \alpha_{12}\Psi_j
- \beta_1) 
+
(\zeta_h - \Psi_j)^\top (\alpha_{22} \Psi_j -\alpha_{12}\Phi_j
- \beta_2) \geq 0,
\end{align}
where $\alpha_{12} = \frac\sigma2 A_{jj}$.
The values of $\alpha_{11}, \beta_1, \alpha_{22}, \beta_2$ can be identified 
from the above, on writing, for $j=1,...,{\cal J}$,
\begin{equation} \label{eq:WZQj}
\begin{aligned}
W^{n+1,k+1}_j & = ([r_1]_j - M_{jj} \Phi^{n+1,k+1}_j) / (\tau A_{jj}), \\
Z^{n+1,k+1}_j & = ([r_3]_j - M_{jj} \Psi^{n+1,k+1}_j) / (\tau M_{jj}), \\
Q^{n+1,k+1}_j & = - ([r_5]_j - A_{jj} \Psi^{n+1,k+1}_j) / M_{jj},
\end{aligned}
\end{equation}
and then substituting these into the variational inequality in \eqref{eq:GS}.
In fact, overall we obtain
\begin{align*} 
\alpha_{11} & = \eps A_{jj} +  C_F M_{jj} + M_{jj}^2 / (\tau A_{jj}), \\
\beta_1 & = [r_2]_j + M_{jj} [r_1]_j / (\tau A_{jj}), \\
\alpha_{22} & 
= (\lambda^\frac12\omega^2M_{jj}^\frac12 - \lambda^\frac12 A_{jj}
M_{jj}^{-\frac12})^2
+ C_F M_{jj} + M_{jj} / \tau, \\
\beta_2 & = [r_4]_j + [r_3]_j / \tau + \lambda A_{jj} [r_5]_j / M_{jj}.
\end{align*}
We can rewrite \eqref{eq:VIj} as
\begin{equation} \label{eq:VIAj}
\binom{\eta_h - \Phi_j}{\zeta_h-\Psi_j}^T \mathfrak A \left[ 
\binom{\Phi_j}{\Psi_j} - \mathfrak A^{-1} \binom{\beta_1}{\beta_2} \right]
\geq 0
\quad \forall\ (\eta_h,\zeta_h) \in \mathcal{K},
\end{equation}
where $\mathfrak{A}=\binom{\alpha_{11}\ -\alpha_{12}}{-\alpha_{12}\ \alpha_{22}}$.
The matrix is symmetric positive definite 
if $\alpha_{12}^2 < \alpha_{11}\alpha_{22}$, which is guaranteed as long as the
time step size $\tau$ is chosen sufficiently small.
In that case, the unique solution to \eqref{eq:VIAj} is
\begin{equation*}
 (\Phi_j, \Psi_j) = {\rm P}_{\cal K}^{\mathfrak A} \left (
\mathfrak A^{-1} \binom{\beta_1}{\beta_2} \right ), \label{PCK}
\end{equation*}
where ${\rm P}_{\cal K}^{\mathfrak A}(x_1,x_2)$ is the orthogonal projection
of the point
$\underline{x}=(x_1,x_2)^\top \in {\mathbb R}^2$ onto ${\cal K}$
with respect to the inner product
$\langle \underline{p},\underline{q}\rangle_{\mathfrak{A}}
:= \underline{p}^T \mathfrak{A}\,\underline{q}$.
The projection $\underline{y}={\rm P}_{\cal K}^{\mathfrak A}(\underline{x})$
can be computed as follows.
\begin{enumerate}
\itemsep1mm
\item If $\underline{x} \in {\cal K}$,
then $\underline{y}=\underline{x}$,~else

\item \quad If $x_2 \leq 0$ then $\underline{y}:=(
      \max\{-1,\min\{x_1 - \frac\sigma{2\alpha_{11}} x_2,1\}\},0 )^\top$, ~else
\item \quad \quad If $x_1 \geq 0$ then $\underline{v}:=(1,-1)^\top$,
      else $\underline{v}:=-(1,1)^\top$.
\item \quad \quad $\gamma:=\displaystyle
             \frac{\langle \underline{x}-(0,1)^\top,\underline{v}\rangle_{\mathfrak A}}
             {\|\underline{v}\|_{\mathfrak A}^2}$.
\item \quad \quad $\underline{y} := (0,1)^\top + \min\{\max\{\gamma,0\},1\}\,
\underline{v}$.
\end{enumerate}

Having found $(\Phi^{n+1,k+1}_j, \Psi^{n+1,k+1}_j) = (\Phi_j, \Psi_j)$ 
from \eqref{eq:VIj}, we can then find $W^{n+1,k+1}_j$, $Z^{n+1,k+1}_j$ and
$Q^{n+1,k+1}_j$ via \eqref{eq:WZQj}. Repeating this procedure for
$j=1,\ldots,\mathcal{J}$ we obtain the solution to \eqref{eq:GS}.
In practice the iteration \eqref{eq:GS} is performed until a suitable stopping
criterion is met.

\subsection{Numerical simulations}
We implemented the scheme \eqref{eq:fea} with the help of the finite element 
toolbox ALBERTA, see \cite{Alberta}. 
To increase computational efficiency, we employ adaptive meshes, which have a 
finer mesh size $h_{f}$ within the regions $|\phi^n_h|<1$ and a coarser 
mesh size $h_{c}$ in the regions $|\phi^n_h| = 1$, see 
\cite{voids3d,grain,voids} for a more detailed description. 

In all our numerical simulations we make use of the splitting \eqref{eq:F1pm} 
for $F_1$ as defined in \eqref{poly}, together with the choice \eqref{eq:CF} for the value of $C_F$.

Throughout the numerical experiments we set 
$\Omega = (-\frac12,\frac12)^d$ and use $\eps=\frac1{16\pi}$ and 
$\lambda=10^{-5}$.
For the computations for Cahn--Hilliard (CH) and 
Cahn--Hilliard--Swift--Hohenberg (CHSH) we let $\alpha=100$,
unless stated otherwise.
For the computations for Swift--Hohenberg (SH) and CHSH, 
unless otherwise stated, we always set $\omega=100$ and $\delta = \sigma = 0$.
Of course, for CH we set $\lambda = g = \gamma = \delta = \sigma = 0$,
while for SH we use $\alpha=\delta=\sigma=0$.
Moreover, for the initial data for CH and CHSH we always choose a random $\phi_0$ with zero mean value and values inside $[-0.01,0.01]$, while
for SH we simply set $\phi_0=0$. Similarly, for SH we choose a random $\psi_0$ 
with mean 0.5 and values inside $[0.49,0.51]$, while for CHSH we set
$\psi_0 = \frac12 (1 - |\phi_0|)$, and for CH we use $\psi_0=0$.

For demonstrative purposes, we begin with a simulation for CH in two dimensions. Then, the usual spinodal decomposition can be observed in
Figure~\ref{fig:Fig32_CH}. The color map shown in Figure~\ref{fig:Fig32_CH} will
be used throughout for the visualizations of $\phi^n_h$.
\begin{figure}
\center
\mbox{
\includegraphics[angle=-0,width=0.2\textwidth]{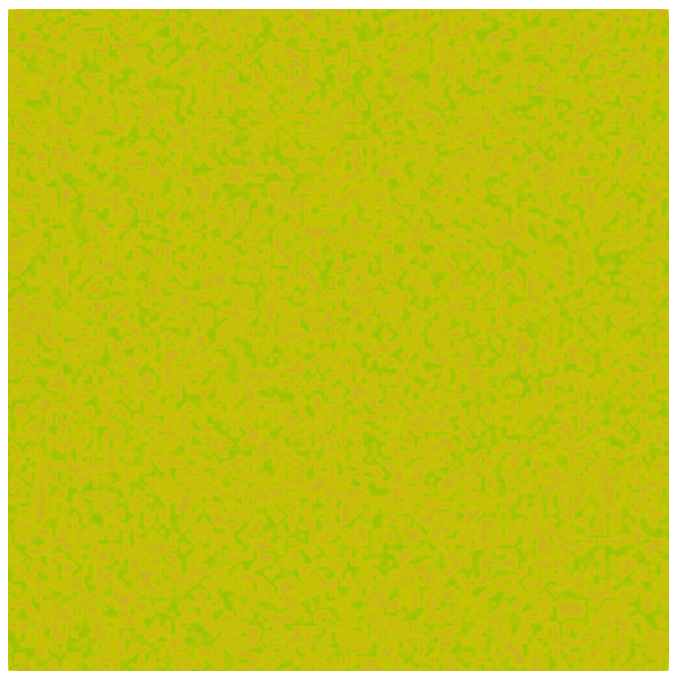}
\includegraphics[angle=-0,width=0.2\textwidth]{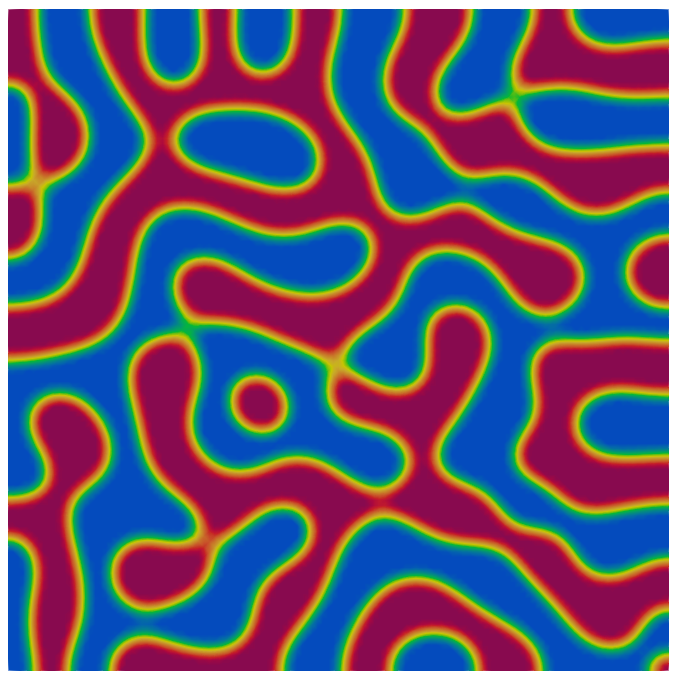}
\includegraphics[angle=-0,width=0.2\textwidth]{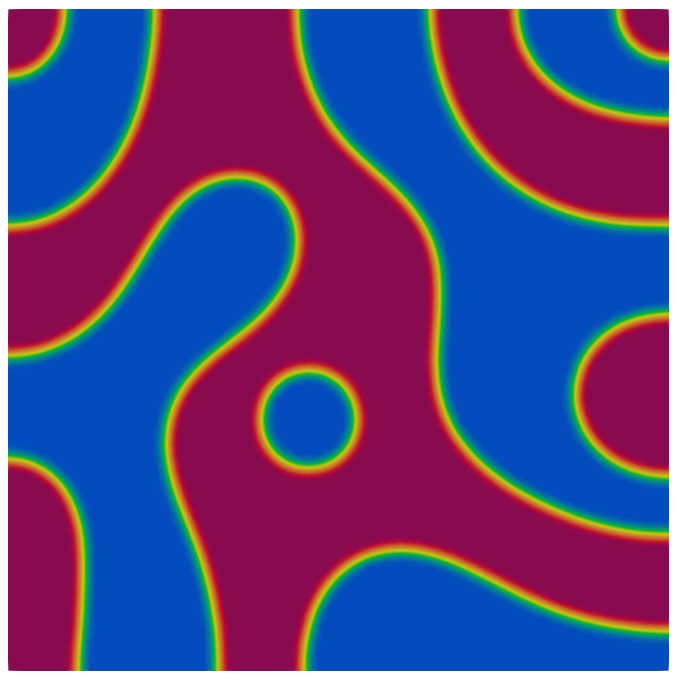}
\includegraphics[angle=-0,width=0.2\textwidth]{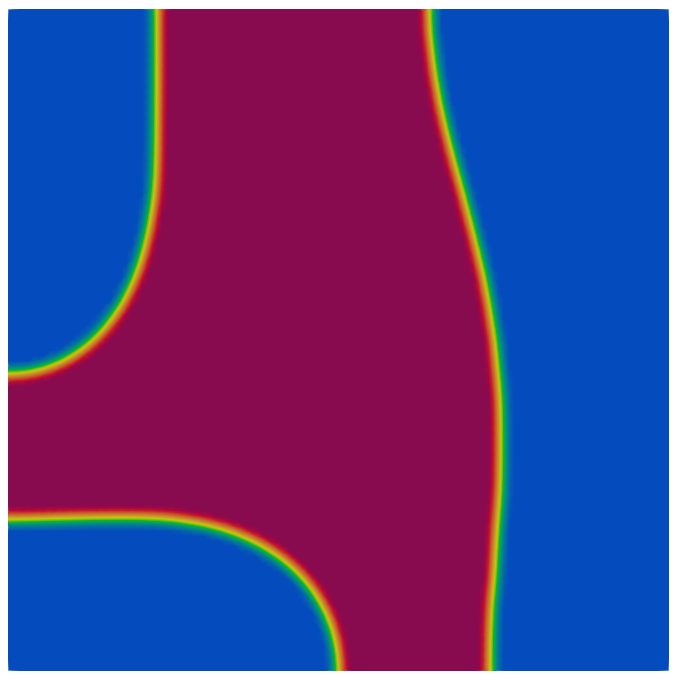}
\includegraphics[angle=-0,width=0.2\textwidth]{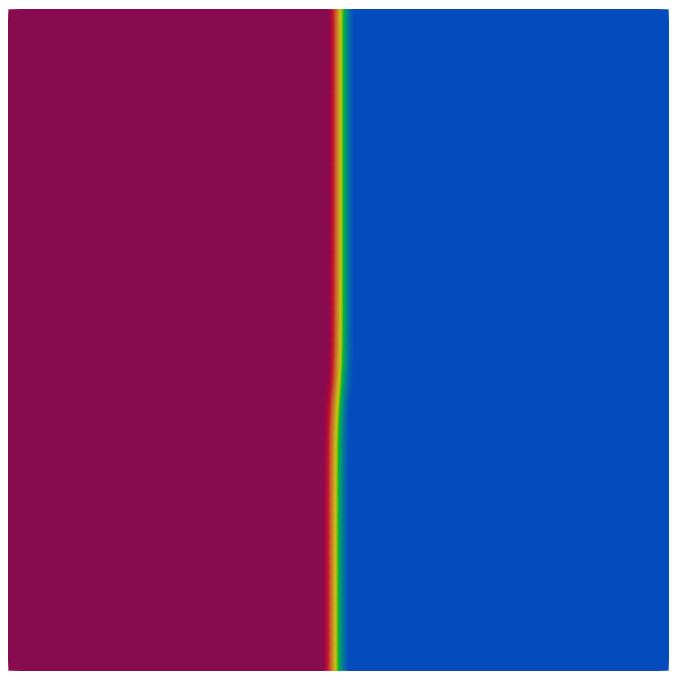}
} 
\includegraphics[angle=-0,width=0.3\textwidth]{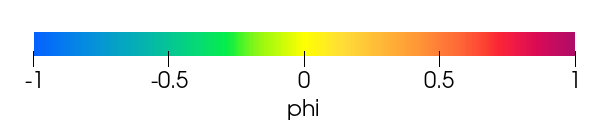}
\caption{Spinodal decomposition for CH.
We display $\phi^n_h$ at times $t=0$, $10^{-4}$, $0.001$, $0.01$, $1$.
}
\label{fig:Fig32_CH}
\end{figure}%
Next we consider some simulations for SH, in order to
obtain some insights into the role of the different parameters in the free
energy of the system.
Here we ran our finite element approximation for a very long time, 
until the numerical solutions $\psi^n_h$ have settled on a stable profile,
or changed only very little. These profiles, for different
parameters, are visualized in Figure~\ref{fig:CF_SH}.
The color map shown in Figure~\ref{fig:CF_SH} will
be used throughout for the visualizations of $\psi^n_h$.
In the first row of Figure~\ref{fig:CF_SH} we can see that increasing
the value of $\omega$ leads to a higher frequency of the observed
oscillations. In the second row we see that increasing $\gamma$, with
$\omega=100$ fixed, leads to more intricate patterns. Finally, the third row
demonstrates that increasing $|g|$, while keeping $\omega=100$ and $\gamma=10$
fixed, leads to the phase $\psi=1$ being preferred, so that small islands
of the phase $\psi=0$ are created. 
\begin{figure}
\center
\mbox{
\includegraphics[angle=-0,width=0.2\textwidth]{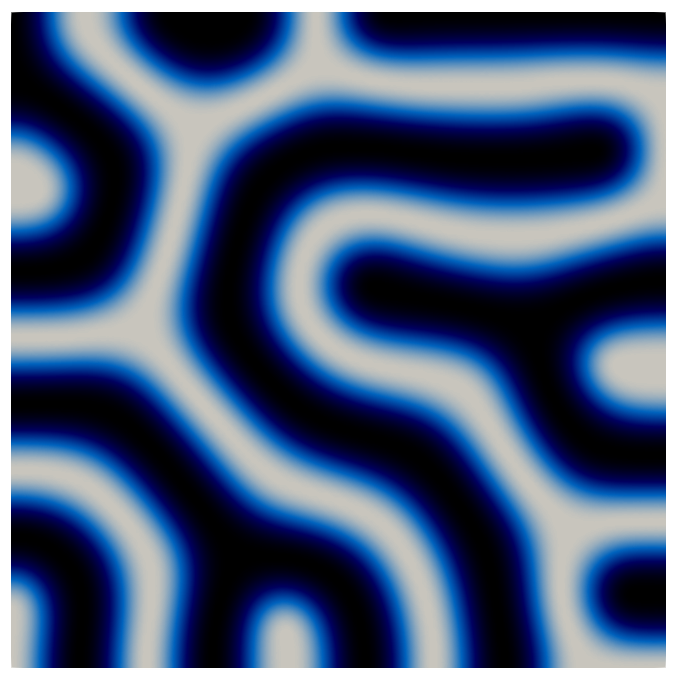}
\includegraphics[angle=-0,width=0.2\textwidth]{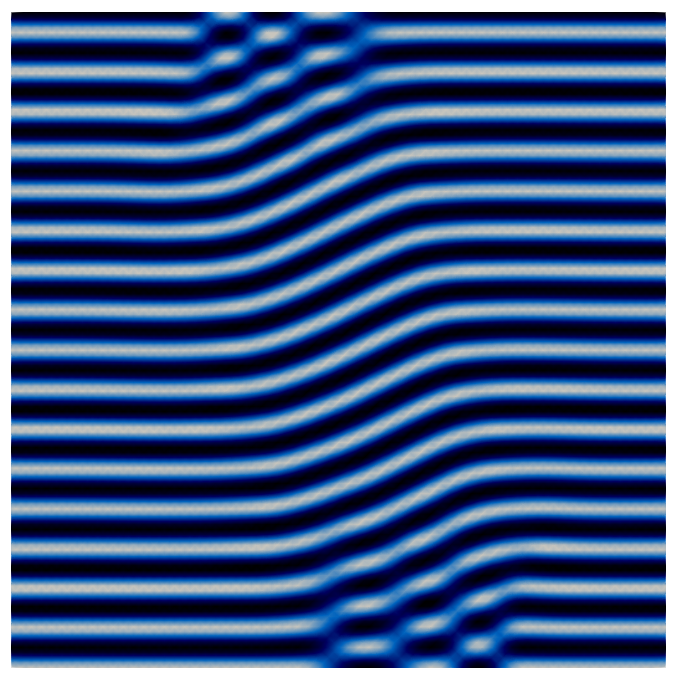}
\includegraphics[angle=-0,width=0.2\textwidth]{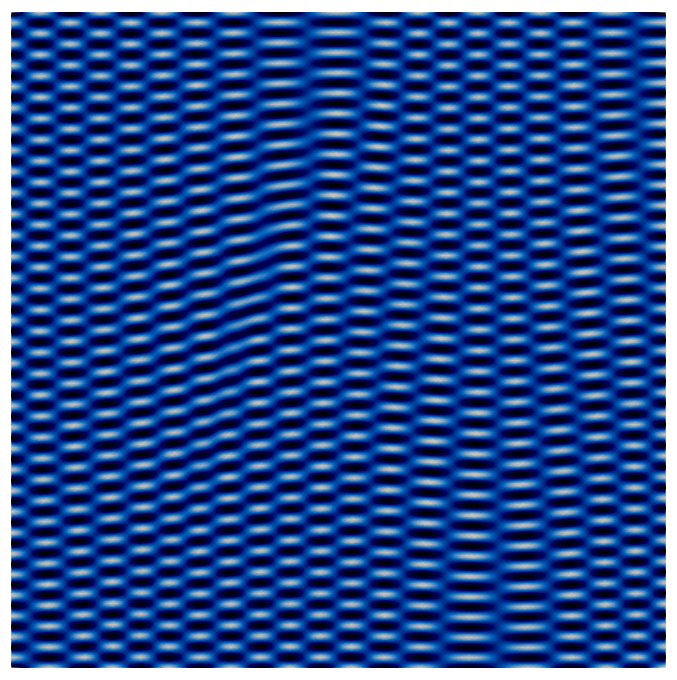}
\includegraphics[angle=-0,width=0.2\textwidth]{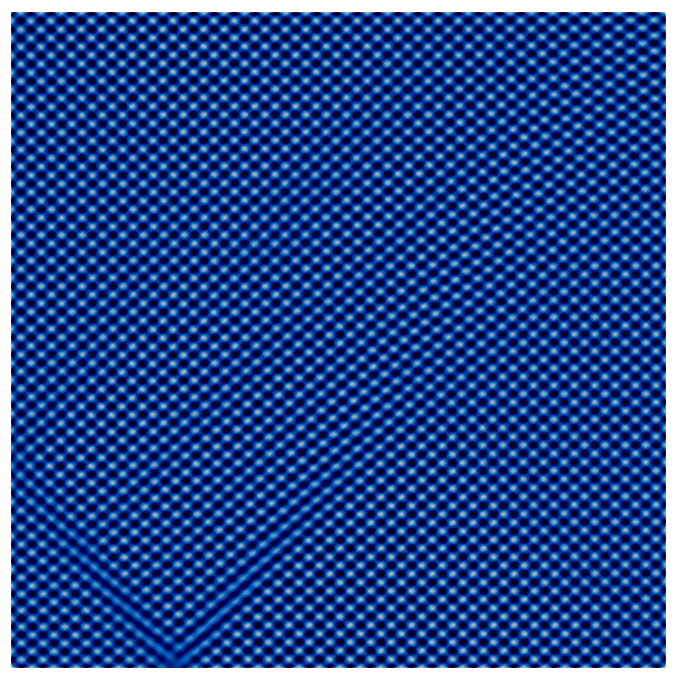}
}
\mbox{
\includegraphics[angle=-0,width=0.2\textwidth]{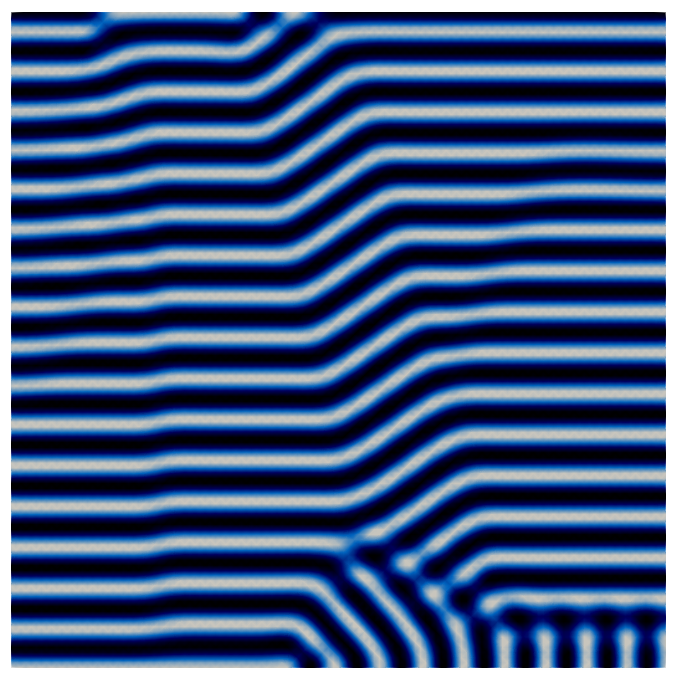}
\includegraphics[angle=-0,width=0.2\textwidth]{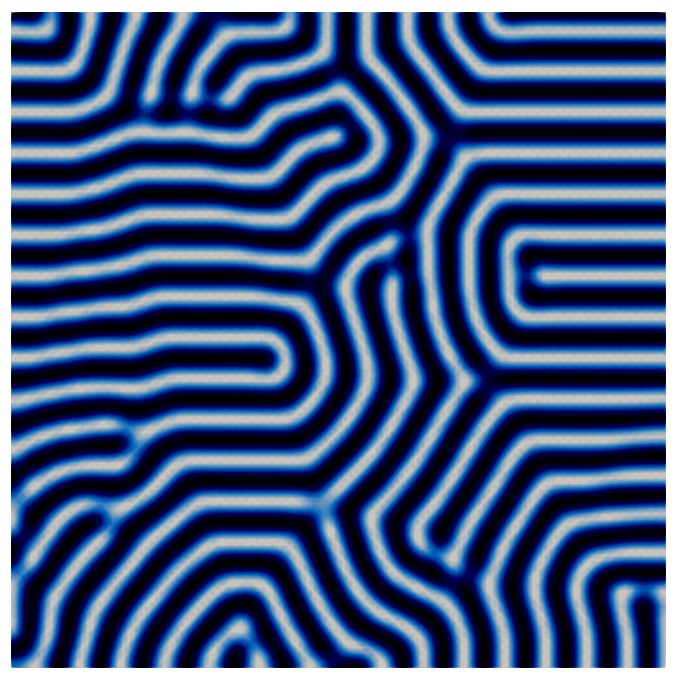}
\includegraphics[angle=-0,width=0.2\textwidth]{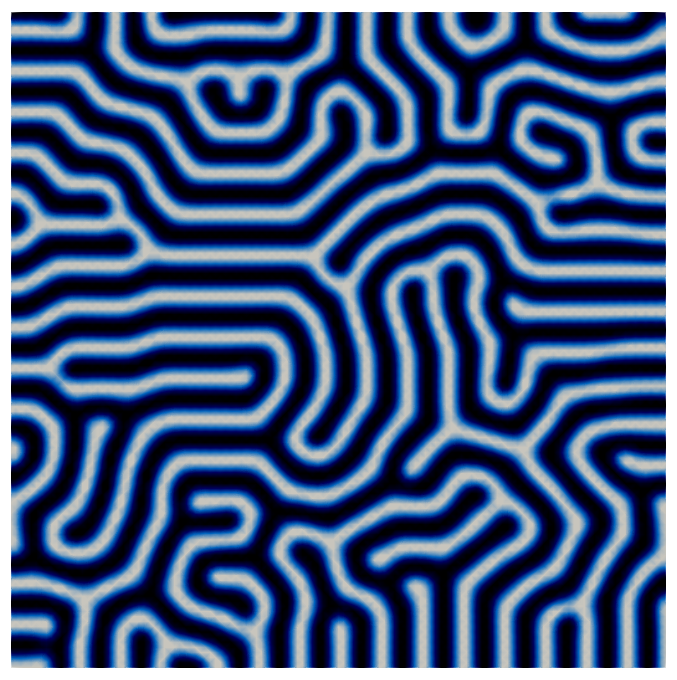}
\includegraphics[angle=-0,width=0.2\textwidth]{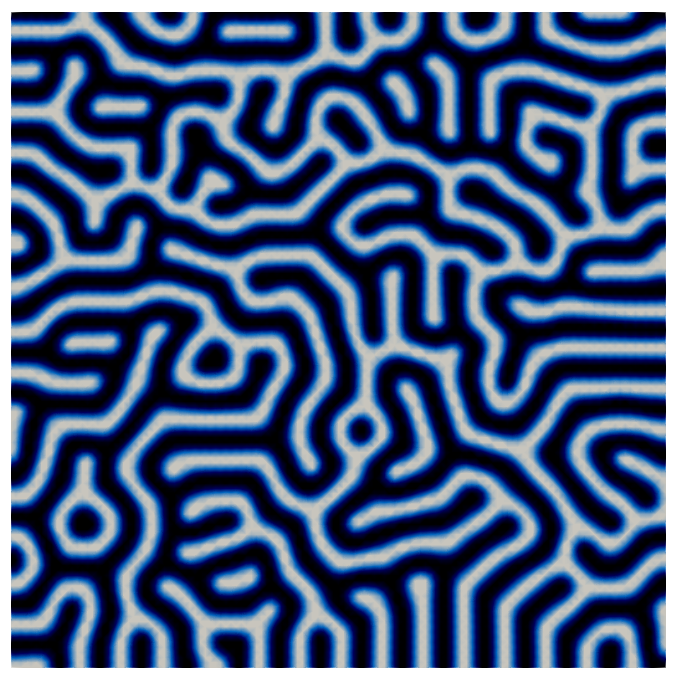}
}
\mbox{
\includegraphics[angle=-0,width=0.2\textwidth]{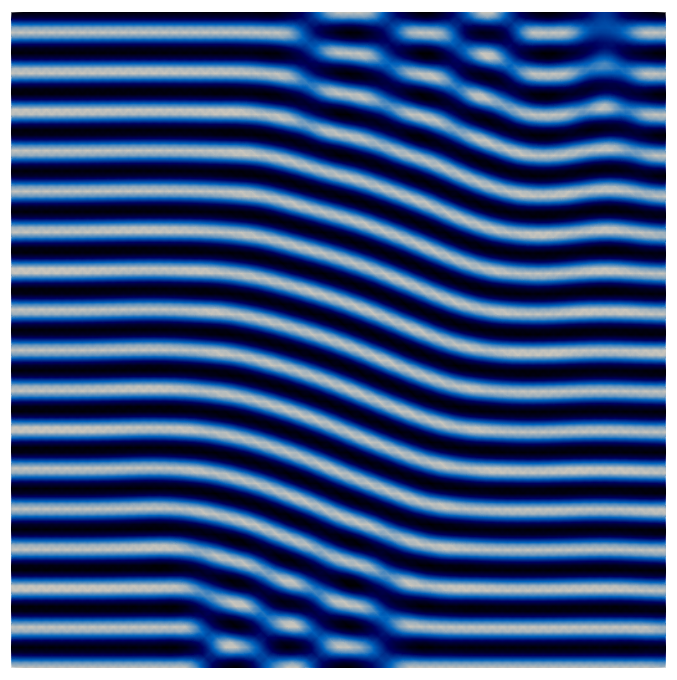}
\includegraphics[angle=-0,width=0.2\textwidth]{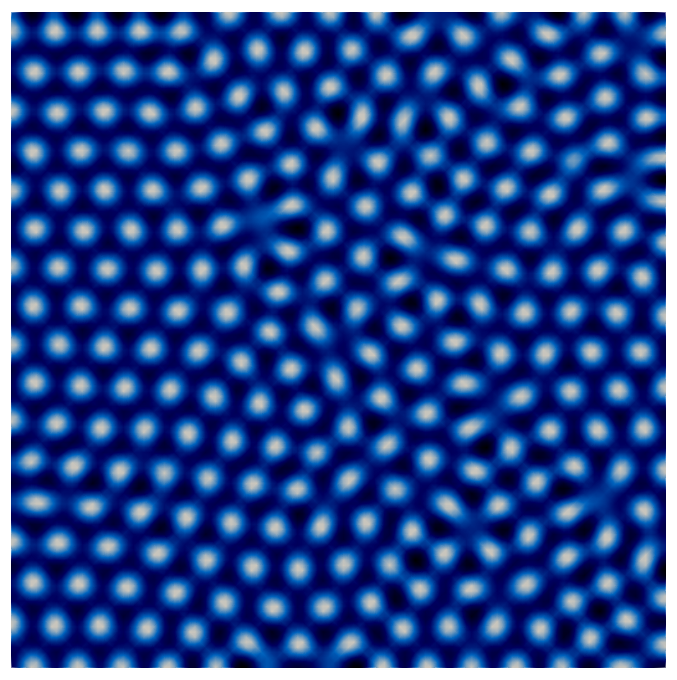}
\includegraphics[angle=-0,width=0.2\textwidth]{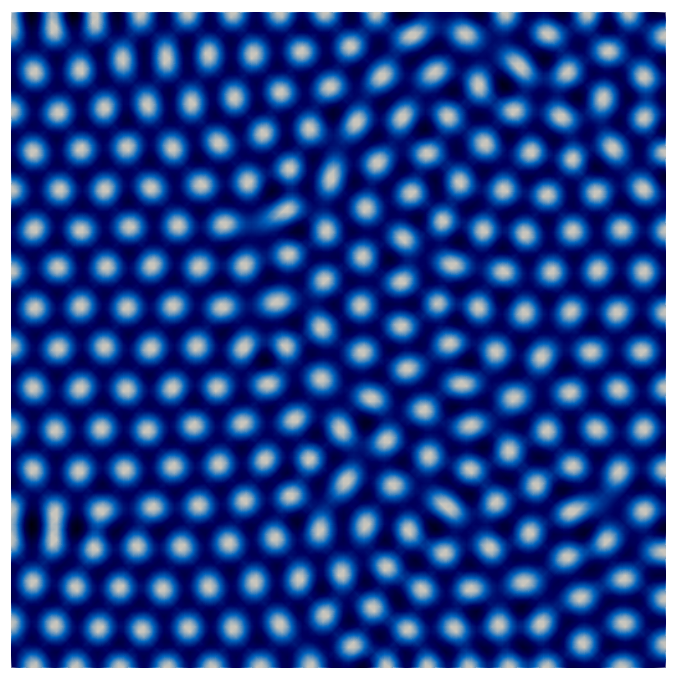}
\includegraphics[angle=-0,width=0.2\textwidth]{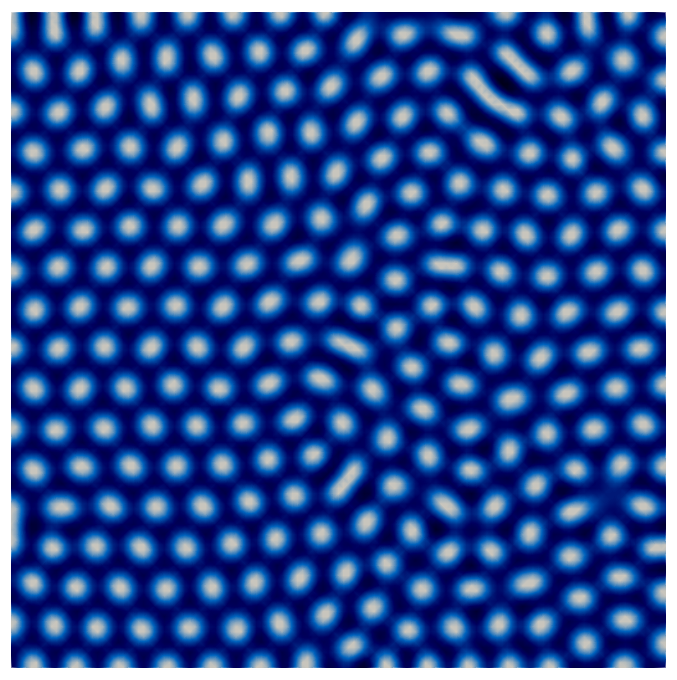}
}
\includegraphics[angle=-0,width=0.3\textwidth]{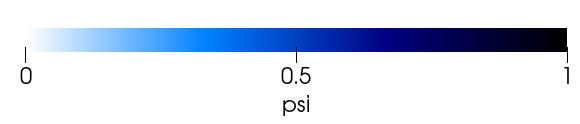}
\caption{Obtained steady states for SH.
First row: $g=0$, $\gamma = 10$ and $\omega=30, 100, 200, 300$.
Second row: $\omega=100$, $g=0$ and $\gamma=100, 200, 500, 1000$.
Third row: $\omega=100$, $\gamma=10$ and $g=-100,-150,-500,-1000$.
}
\label{fig:CF_SH}
\end{figure}%
If we now combine the parameters from Figure~\ref{fig:Fig32_CH} and
the last image in the second row of Figure~\ref{fig:CF_SH} for the full CHSH
model, we see a dramatically different evolution.
We refer to Figure~\ref{fig:Fig32_delta0} for the numerical results, which can be compared to Figure 7f in \cite{Morales_JTB}. As a comparison we show the time evolution for SH on its own in
Figure~\ref{fig:Fig32_SH}. Comparing the evolving patterns in
Figure~\ref{fig:Fig32_delta0}, with the pure CH evolution
in Figure~\ref{fig:Fig32_CH} and the pure SH evolution in
Figure~\ref{fig:Fig32_SH}, we note that only by combining the two gives rise
to the kind of complex patterns that motivates our current study.
\begin{figure}
\center
\mbox{
\includegraphics[angle=-0,width=0.2\textwidth]{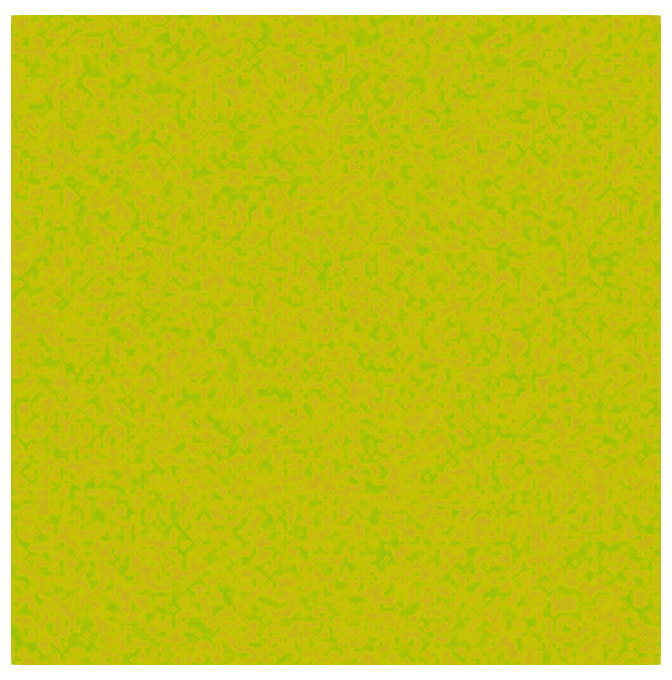}
\includegraphics[angle=-0,width=0.2\textwidth]{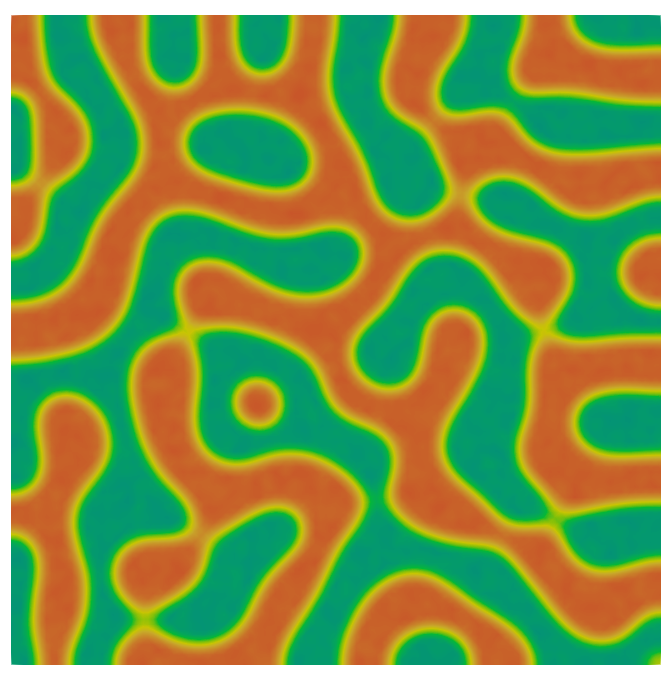}
\includegraphics[angle=-0,width=0.2\textwidth]{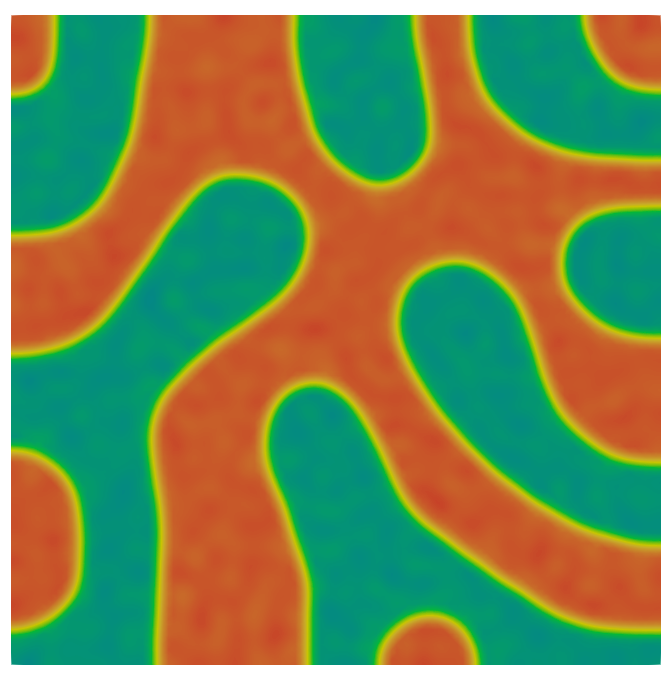}
\includegraphics[angle=-0,width=0.2\textwidth]{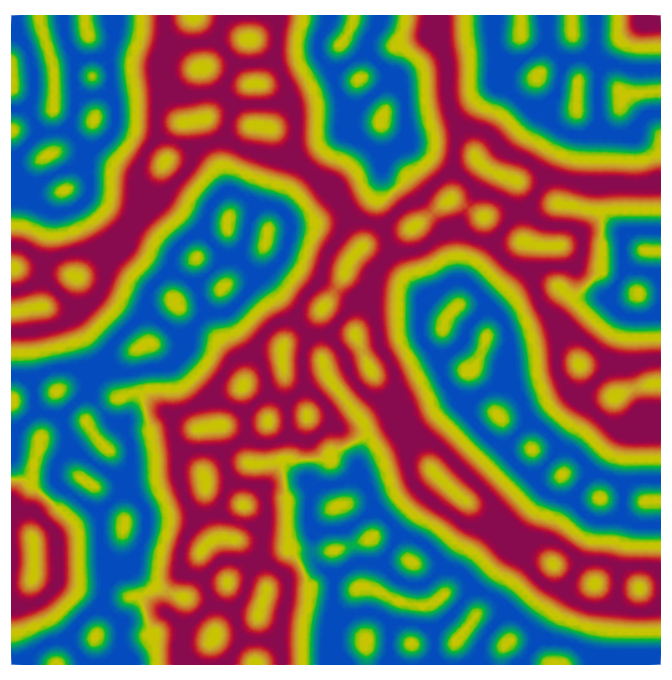}
\includegraphics[angle=-0,width=0.2\textwidth]{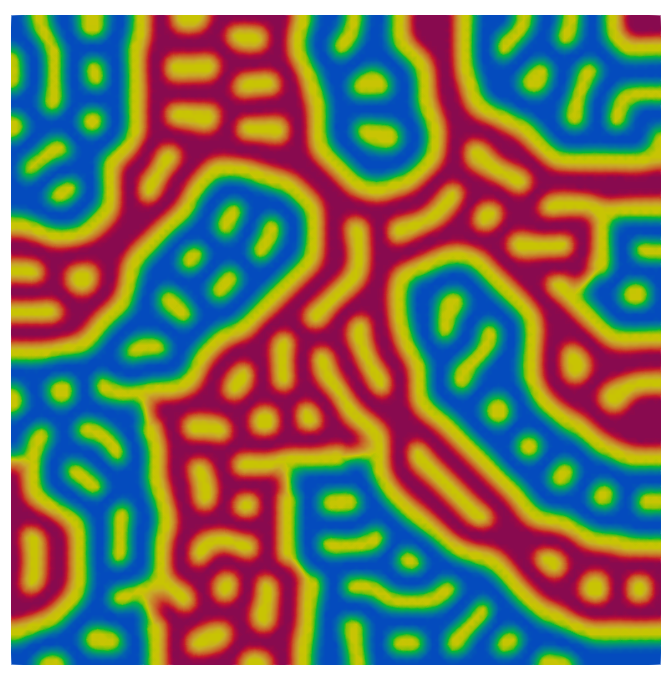}
} \\
\mbox{
\includegraphics[angle=-0,width=0.2\textwidth]{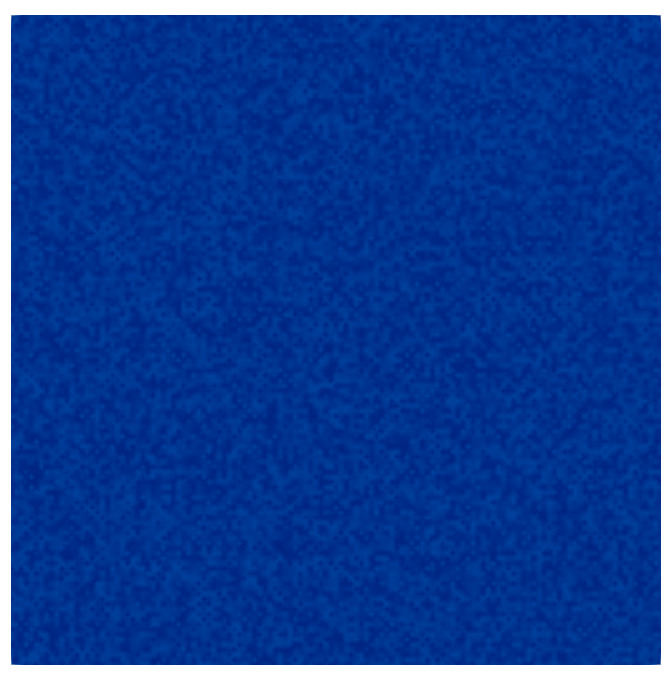}
\includegraphics[angle=-0,width=0.2\textwidth]{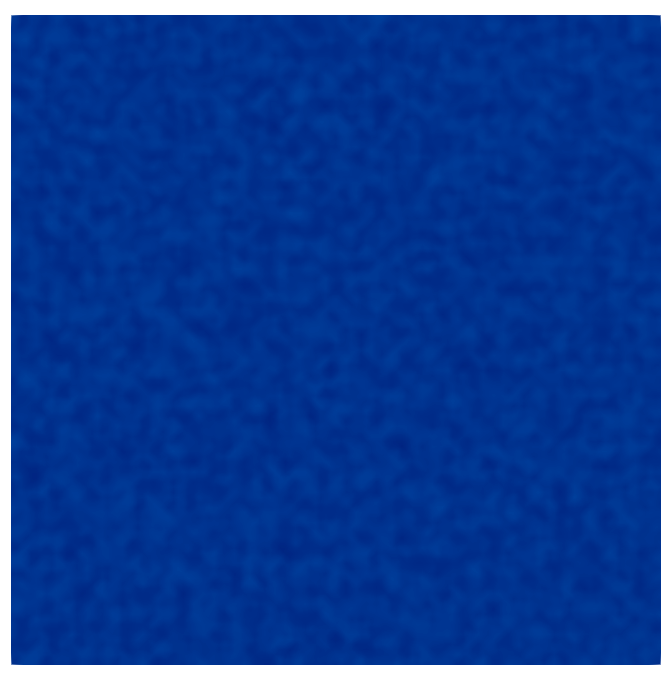}
\includegraphics[angle=-0,width=0.2\textwidth]{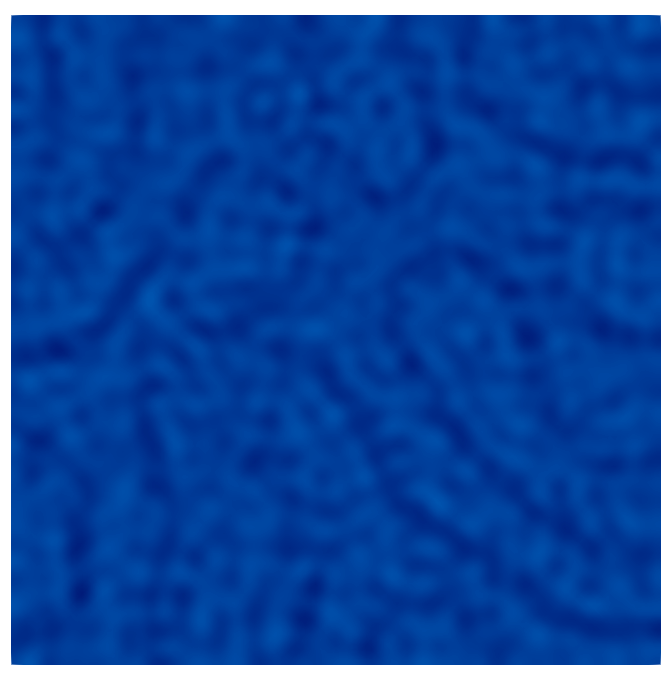}
\includegraphics[angle=-0,width=0.2\textwidth]{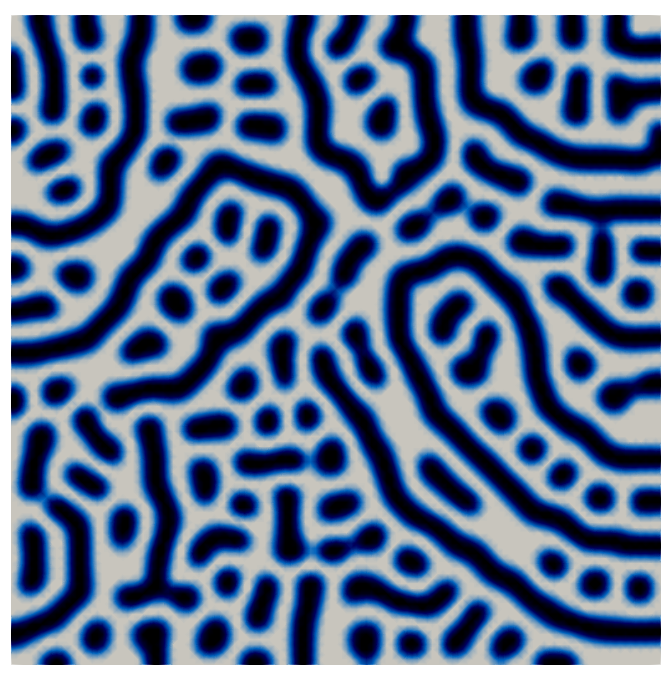}
\includegraphics[angle=-0,width=0.2\textwidth]{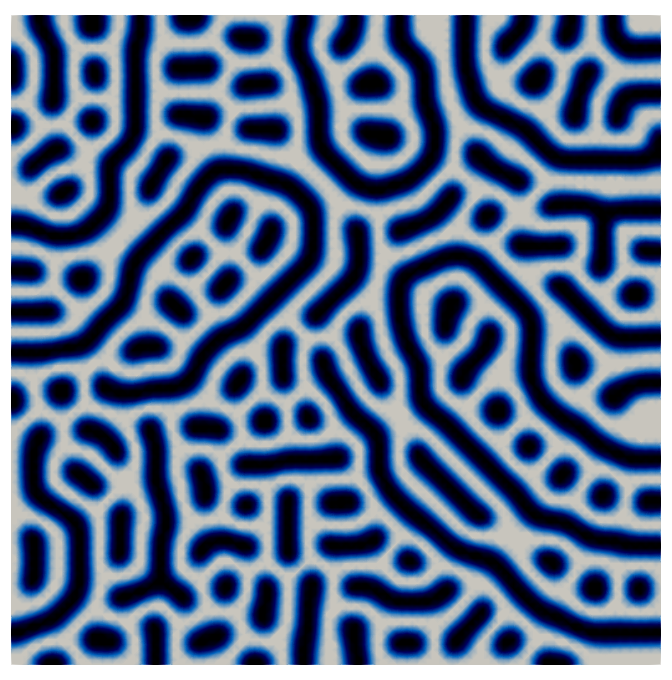}
} 
\caption{
Computation for CHSH with $g=0$, $\gamma=1000$.
We display $\phi_h^n$ at times $t=0$, $10^{-4}$, $0.001$, $0.01$, $0.1$.
Below we show $\psi_h^n$ at the same times.
}
\label{fig:Fig32_delta0}
\end{figure}%
\begin{figure}
\center
\mbox{
\includegraphics[angle=-0,width=0.2\textwidth]{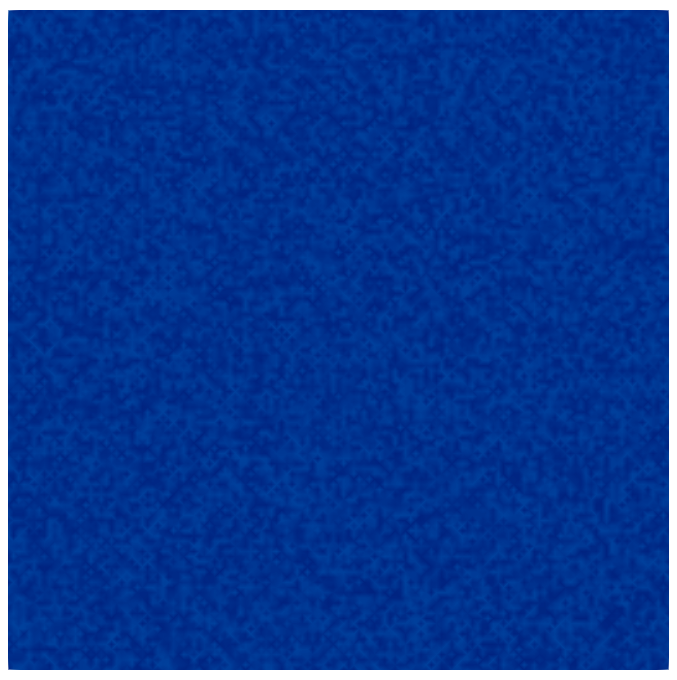}
\includegraphics[angle=-0,width=0.2\textwidth]{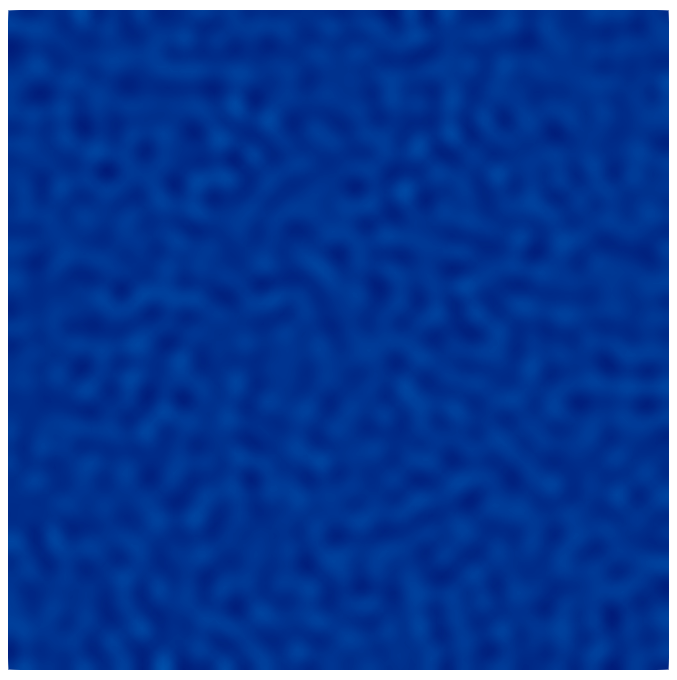}
\includegraphics[angle=-0,width=0.2\textwidth]{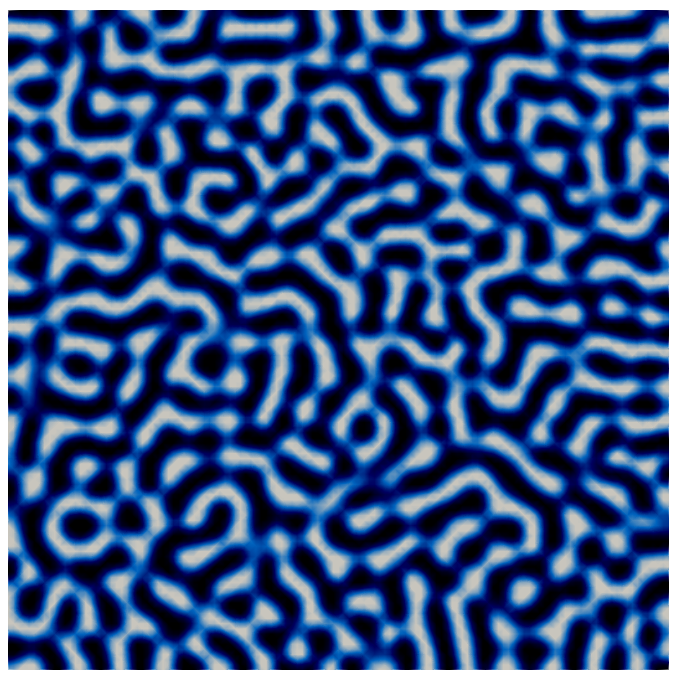}
\includegraphics[angle=-0,width=0.2\textwidth]{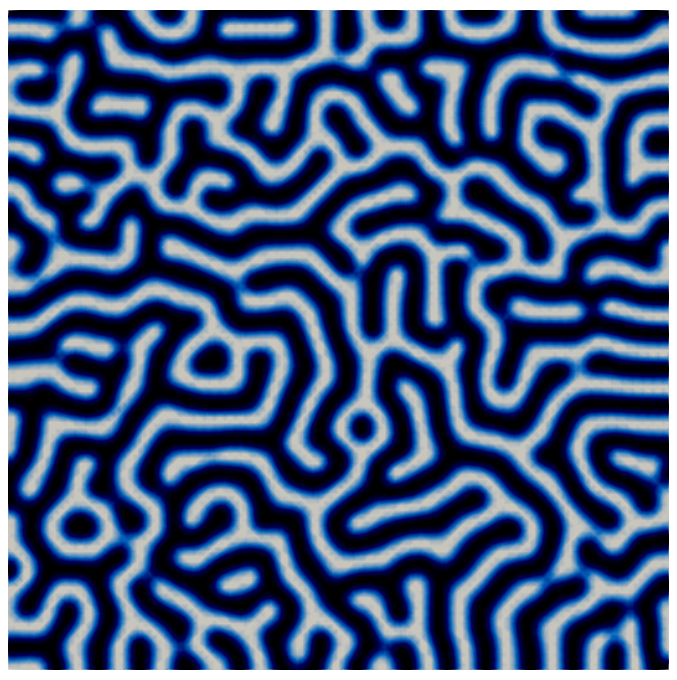}
\includegraphics[angle=-0,width=0.2\textwidth]{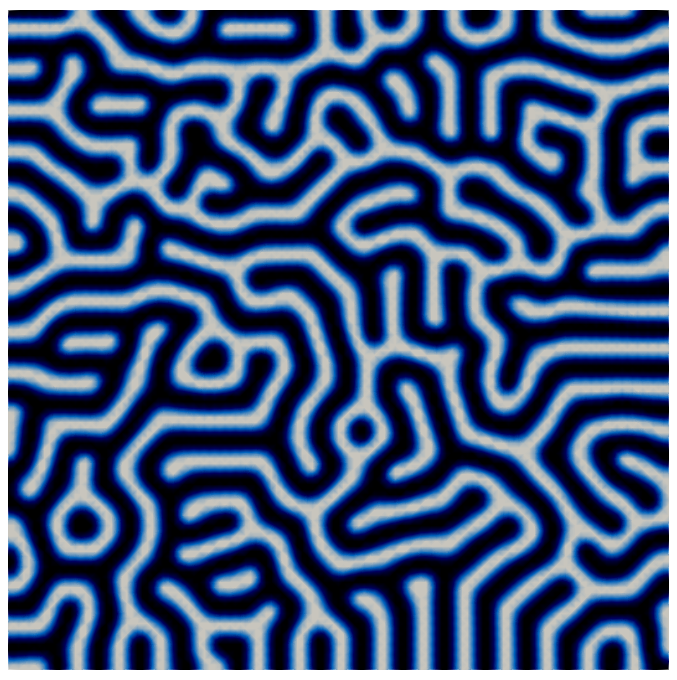}
} 
\caption{Computation for SH with $g=0$, $\gamma=1000$.
(Compare with Figure~\ref{fig:Fig32_delta0}.)
We display $\psi_h^n$ at times $t=0$, $0.001$, $0.005$, $0.01$, $0.1$.
}
\label{fig:Fig32_SH}
\end{figure}%

If we repeat the CHSH simulation from Figure~\ref{fig:Fig32_delta0}
for a smaller value of $\gamma$, we obtain the results in 
Figure~\ref{fig:Fig31_delta0}. We observe that they show some resemblance
to the patterns in \cite{Morales_JTB}.
\begin{figure}
\center
\mbox{
\includegraphics[angle=-0,width=0.2\textwidth]{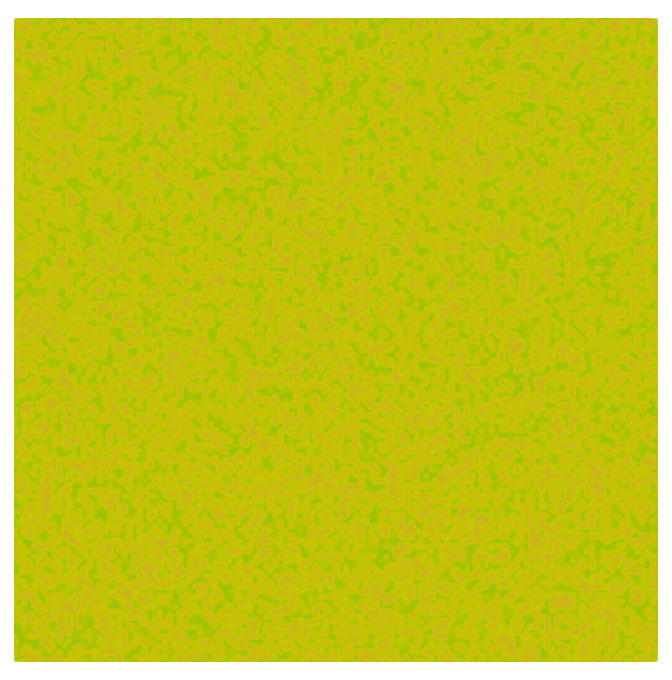}
\includegraphics[angle=-0,width=0.2\textwidth]{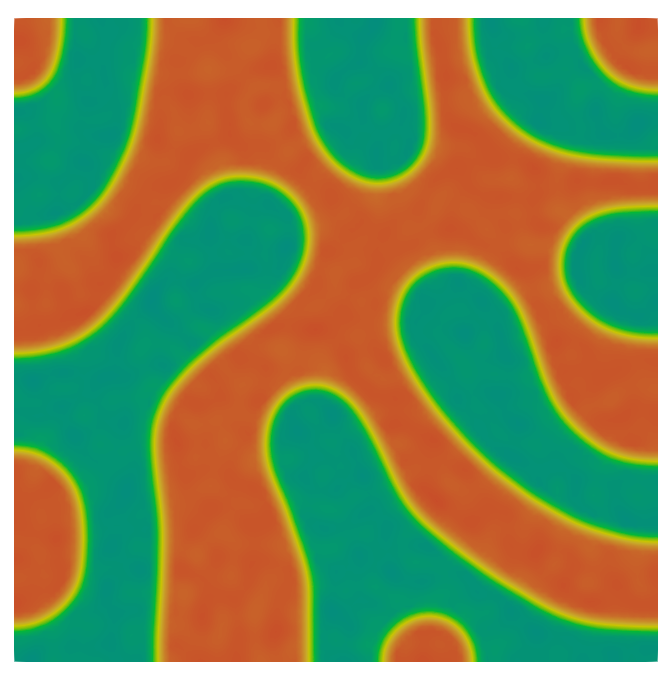}
\includegraphics[angle=-0,width=0.2\textwidth]{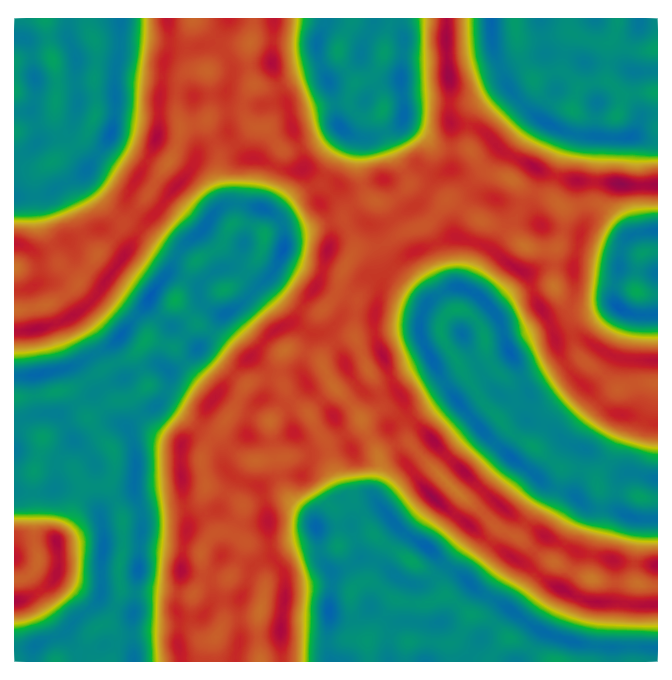}
\includegraphics[angle=-0,width=0.2\textwidth]{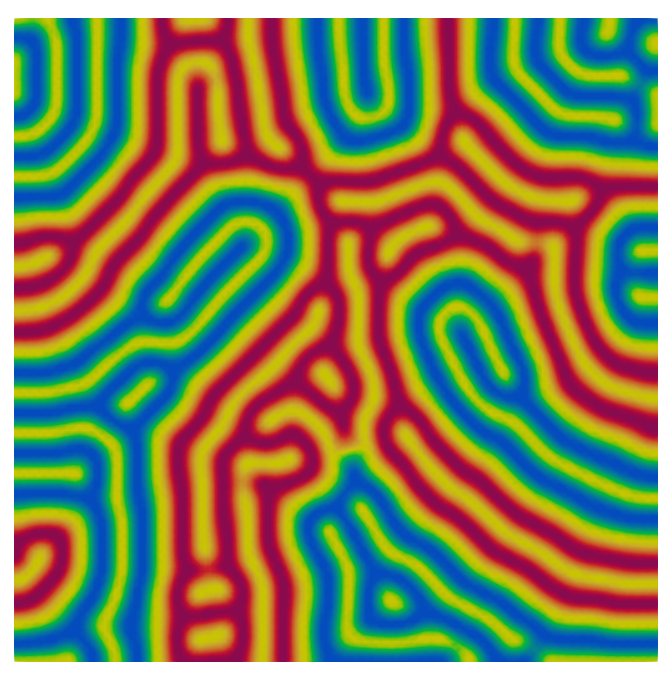}
\includegraphics[angle=-0,width=0.2\textwidth]{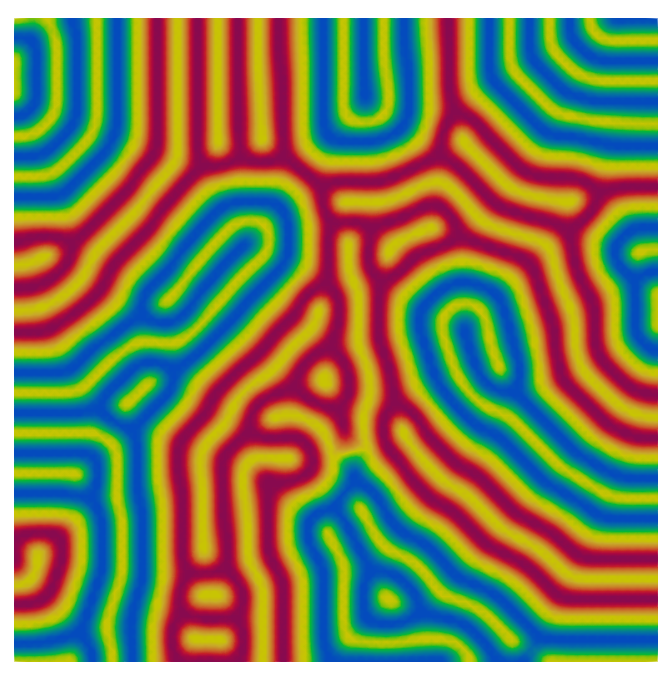}
} \\
\mbox{
\includegraphics[angle=-0,width=0.2\textwidth]{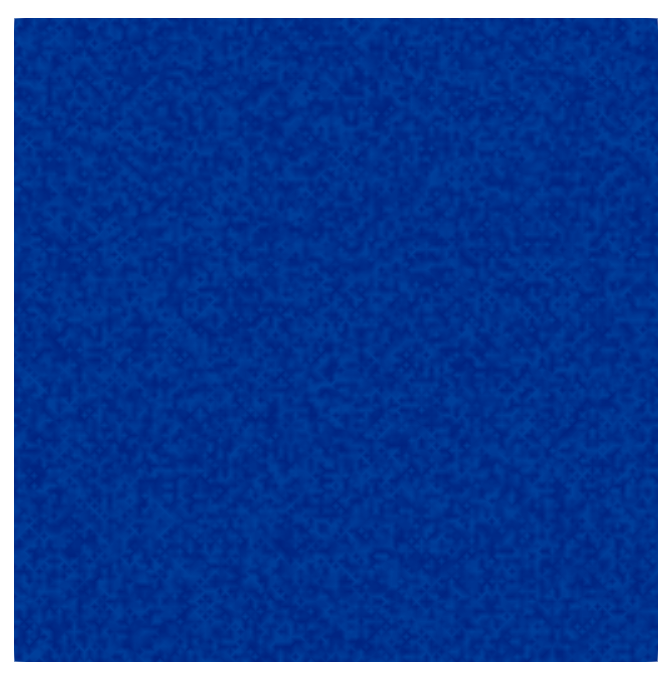}
\includegraphics[angle=-0,width=0.2\textwidth]{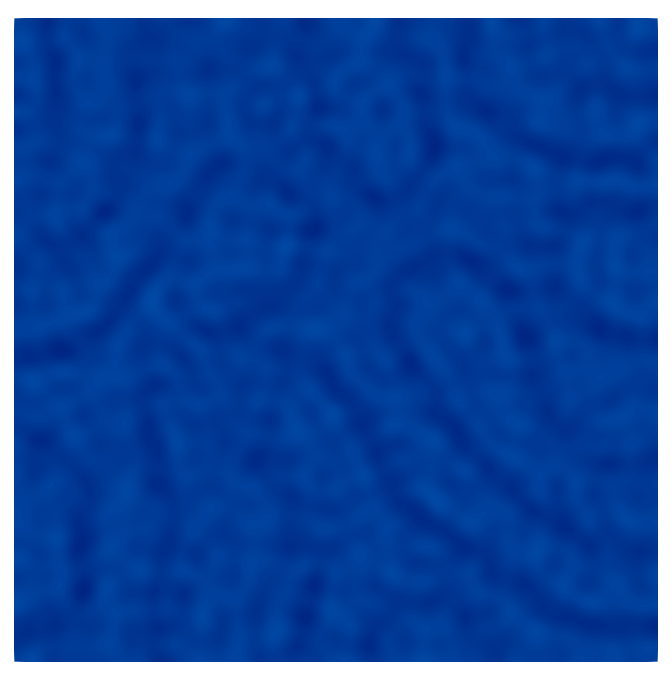}
\includegraphics[angle=-0,width=0.2\textwidth]{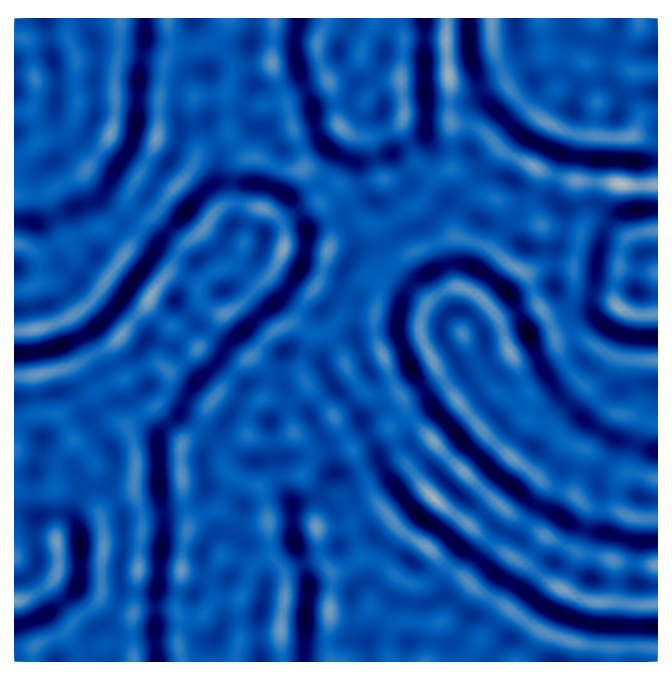}
\includegraphics[angle=-0,width=0.2\textwidth]{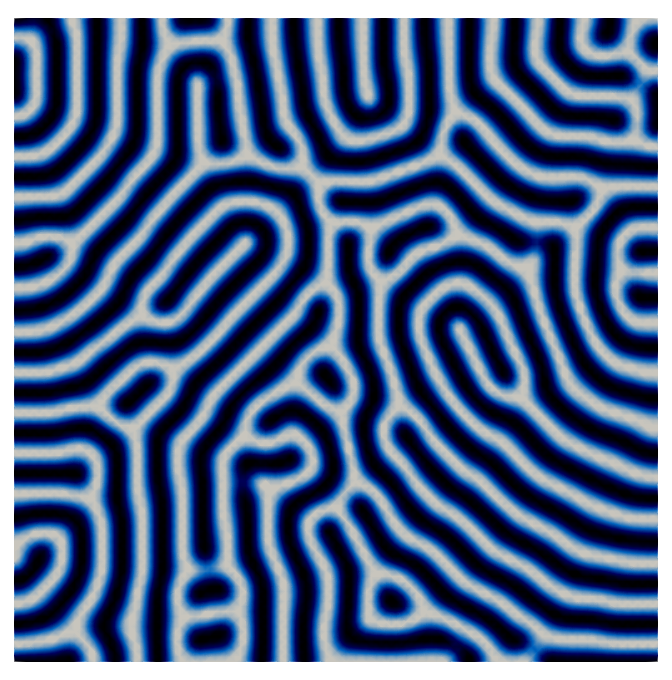}
\includegraphics[angle=-0,width=0.2\textwidth]{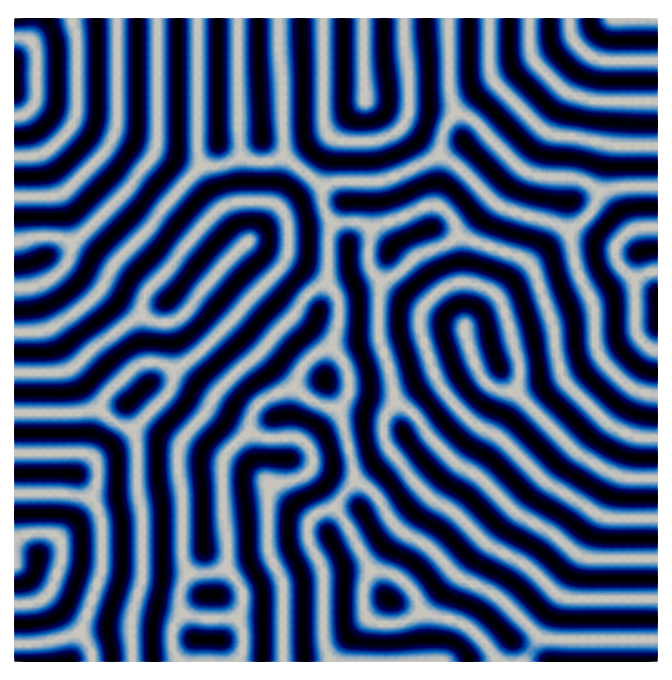}
} 
\caption{
Computation for CHSH with $g=0$, $\gamma=500$.
We display $\phi_h^n$ at times $t=0$, $0.001$, $0.005$, $0.05$, $0.5$.
Below we show $\psi_h^n$ at the same times.
}
\label{fig:Fig31_delta0}
\end{figure}%

Our next simulations investigate the effect of the parameter $g$ on the CHSH
evolutions. An experiment with $g=2000$ can be seen in Figure \ref{fig:Fig31d0_g2000}.
Here the phase $\psi=1$ is preferred by the evolution, which in turn has an
effect on the pattern that develops for $\phi$.
Numerical simulations with $g\in\{-300,-1000,-2000\}$ can be seen in
Figures~\ref{fig:Fig31d0_g-300}, \ref{fig:Fig31d0_g-1000} and
\ref{fig:Fig31d0_g-2000}, respectively. In Figure~\ref{fig:Fig31d0_g-2000} we observe the formation of islands in $\phi$ and $\psi$, cf.~Figure~7d in \cite{Morales_JTB}.
\begin{figure}
\center
\mbox{
\includegraphics[angle=-0,width=0.2\textwidth]{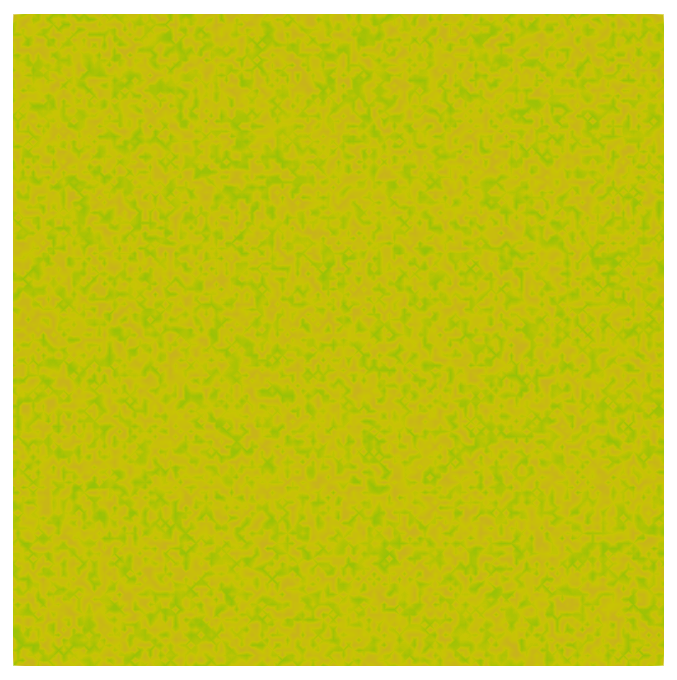}
\includegraphics[angle=-0,width=0.2\textwidth]{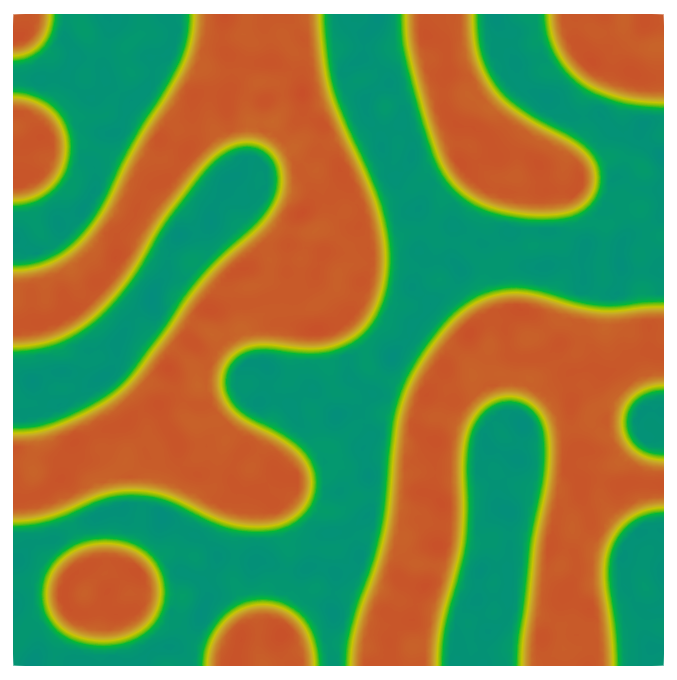}
\includegraphics[angle=-0,width=0.2\textwidth]{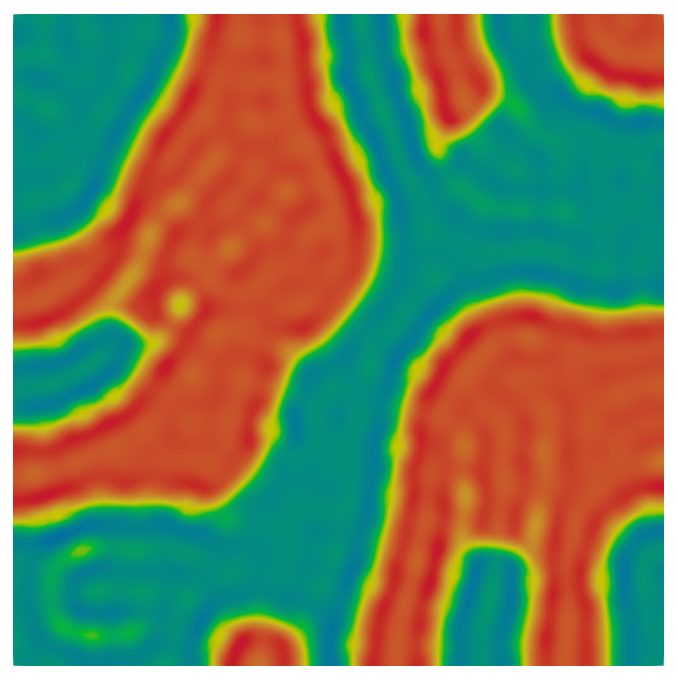}
\includegraphics[angle=-0,width=0.2\textwidth]{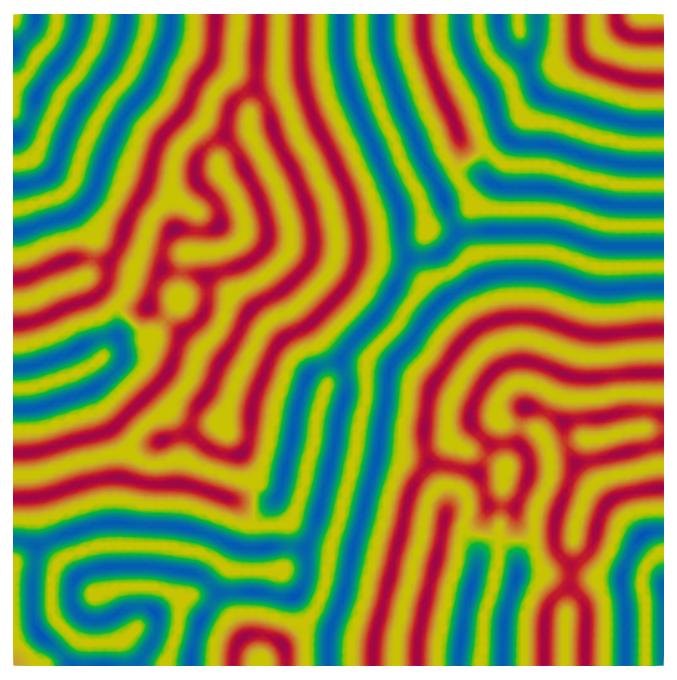}
\includegraphics[angle=-0,width=0.2\textwidth]{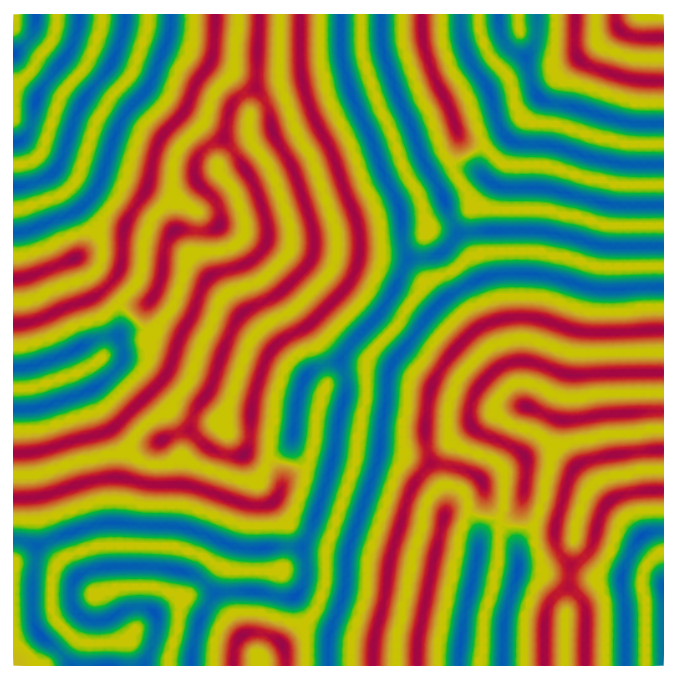}
} \\
\mbox{
\includegraphics[angle=-0,width=0.2\textwidth]{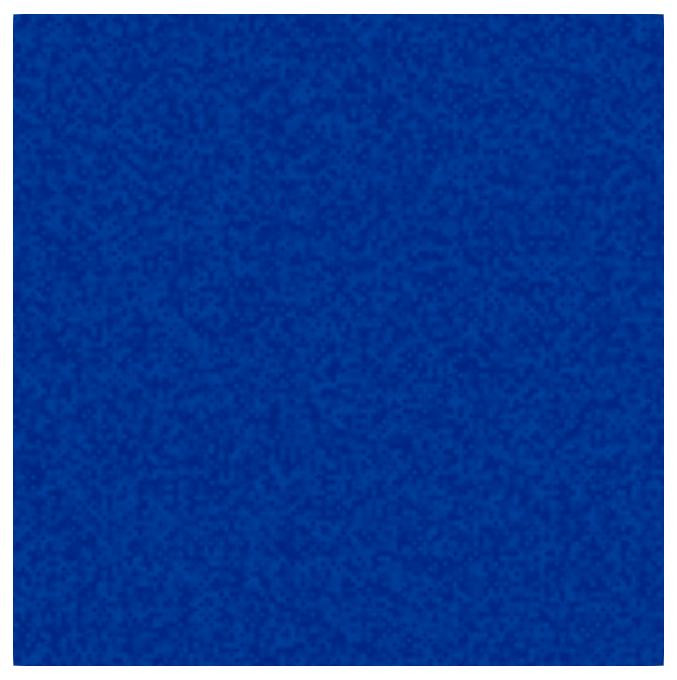}
\includegraphics[angle=-0,width=0.2\textwidth]{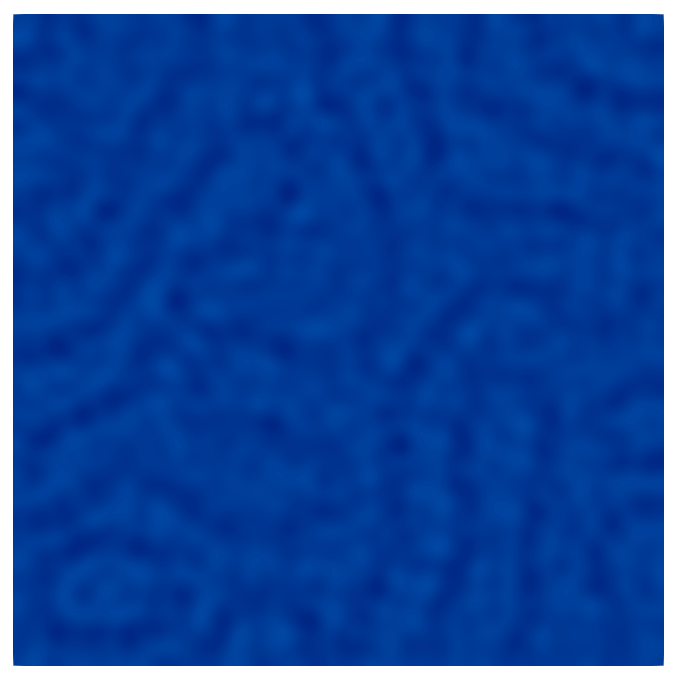}
\includegraphics[angle=-0,width=0.2\textwidth]{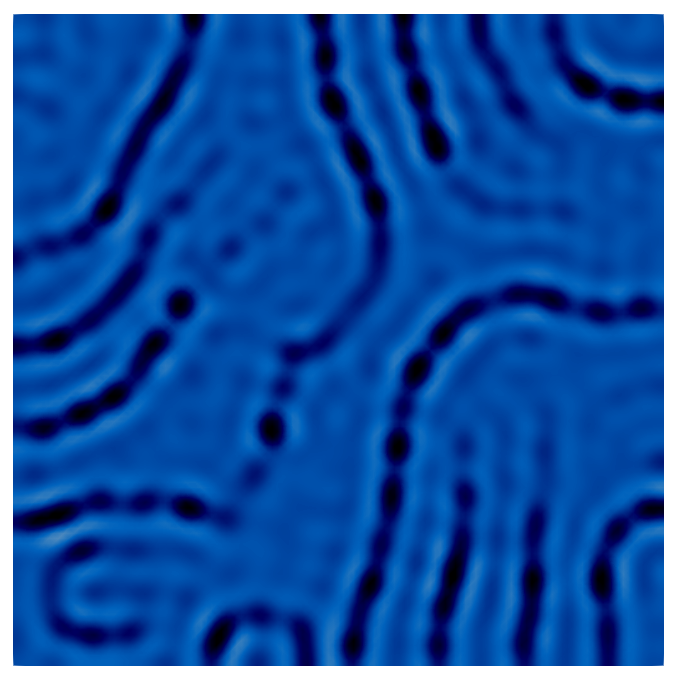}
\includegraphics[angle=-0,width=0.2\textwidth]{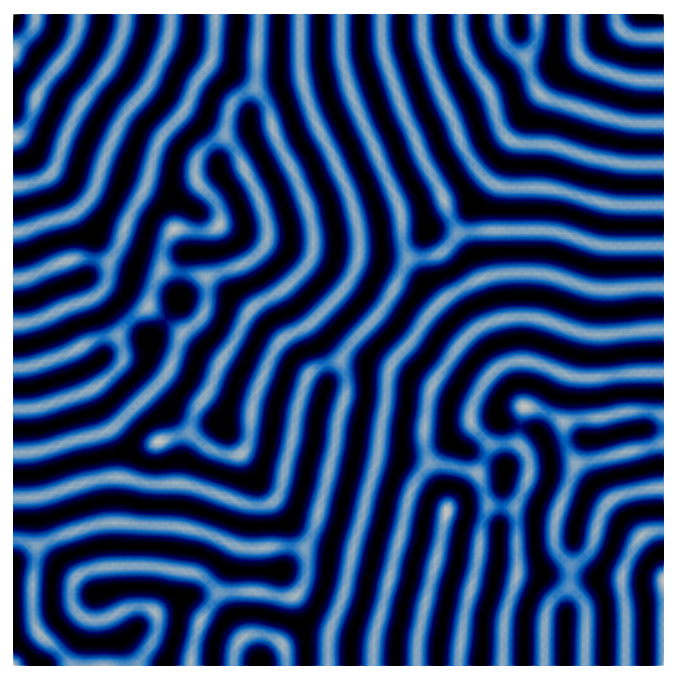}
\includegraphics[angle=-0,width=0.2\textwidth]{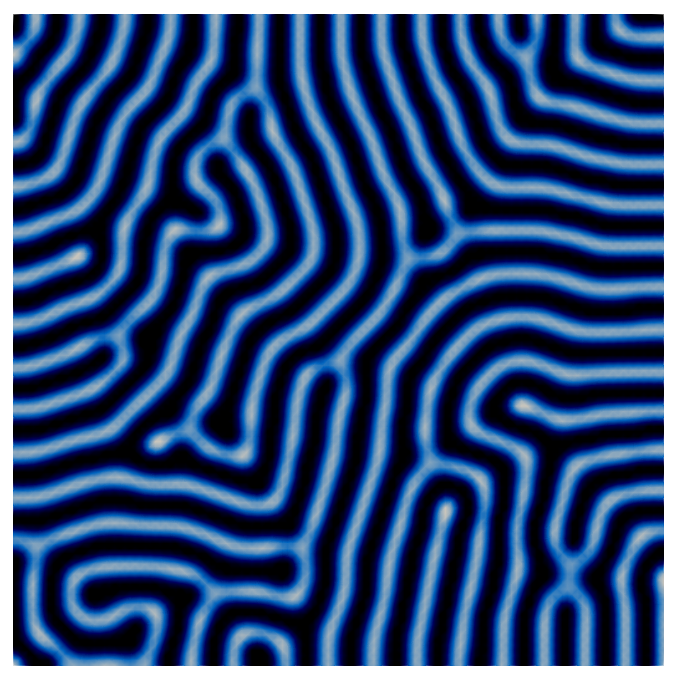}
} 
\caption{The same as Figure~\ref{fig:Fig31_delta0},
but with $g=2000$.
We display $\phi_h^n$ at times $t=0$, $0.001$, $0.005$, $0.05$, $0.5$.
Below we show $\psi_h^n$ at the same times.
}
\label{fig:Fig31d0_g2000}
\end{figure}%
\begin{figure}
\center
\mbox{
\includegraphics[angle=-0,width=0.2\textwidth]{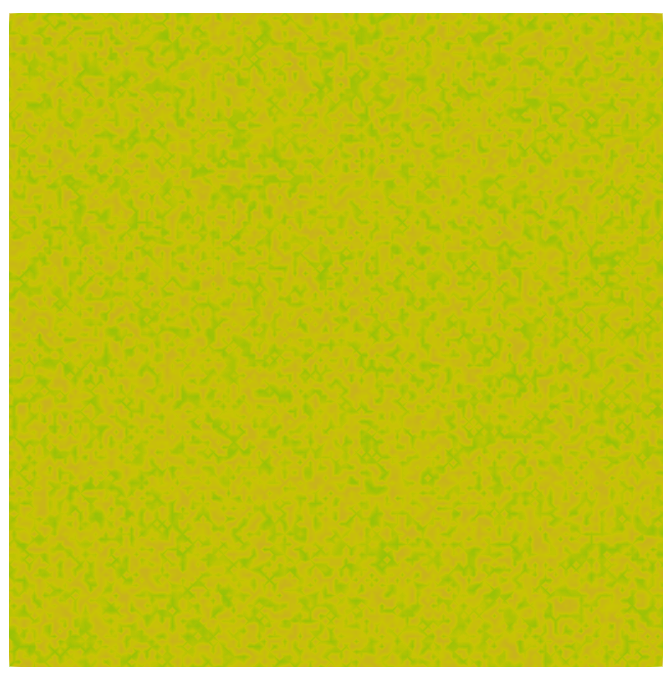}
\includegraphics[angle=-0,width=0.2\textwidth]{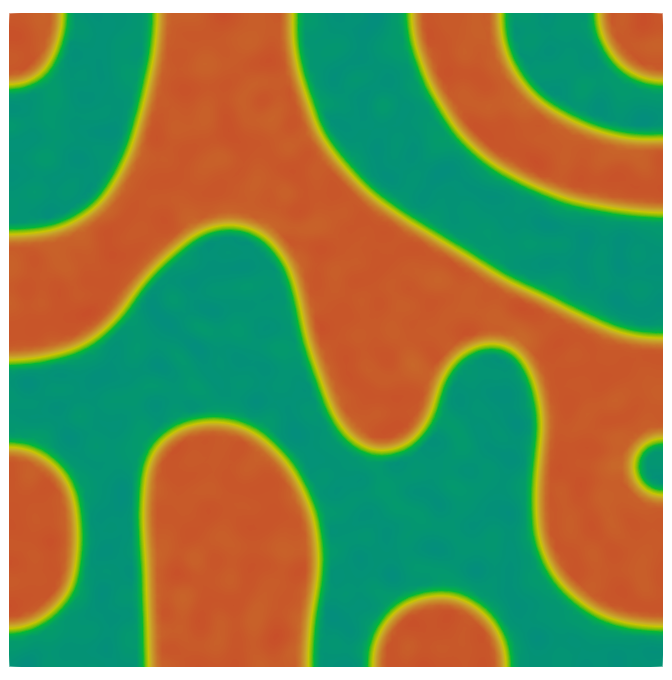}
\includegraphics[angle=-0,width=0.2\textwidth]{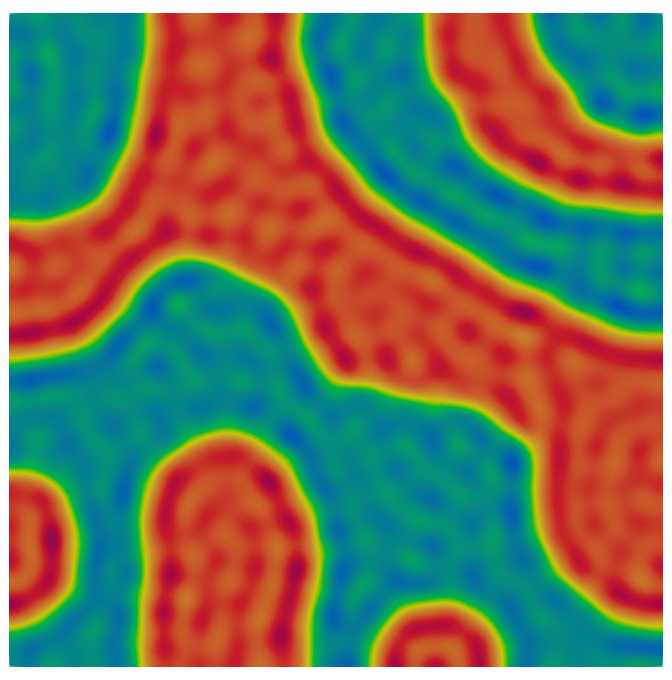}
\includegraphics[angle=-0,width=0.2\textwidth]{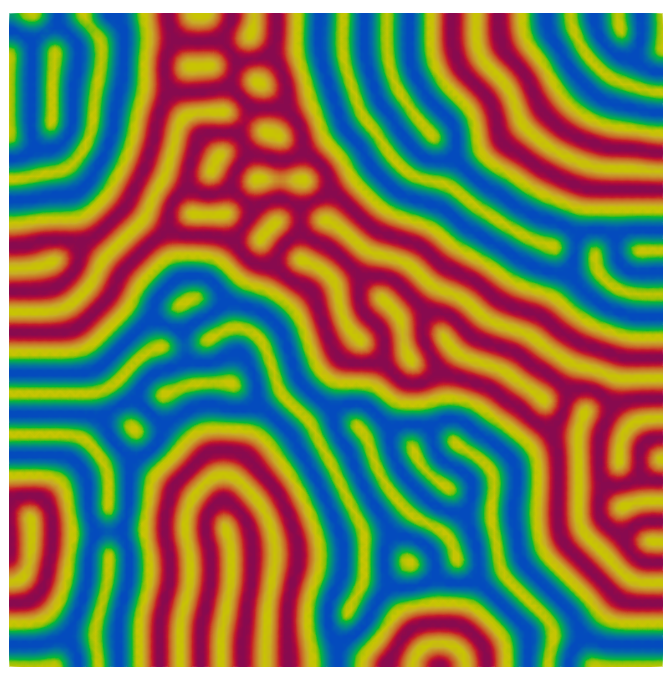}
\includegraphics[angle=-0,width=0.2\textwidth]{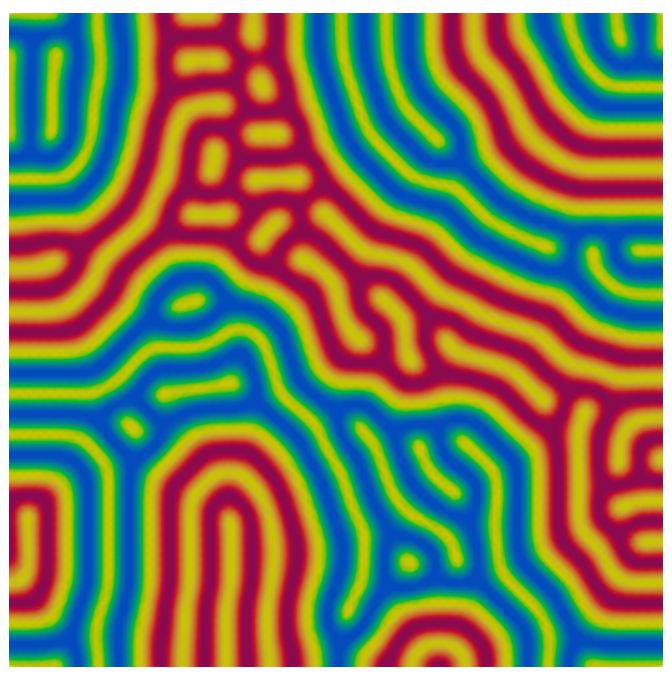}
} \\
\mbox{
\includegraphics[angle=-0,width=0.2\textwidth]{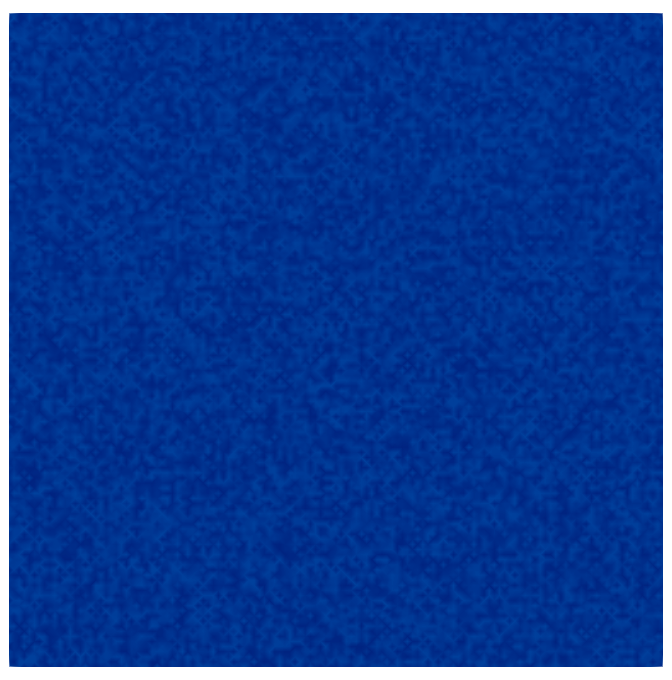}
\includegraphics[angle=-0,width=0.2\textwidth]{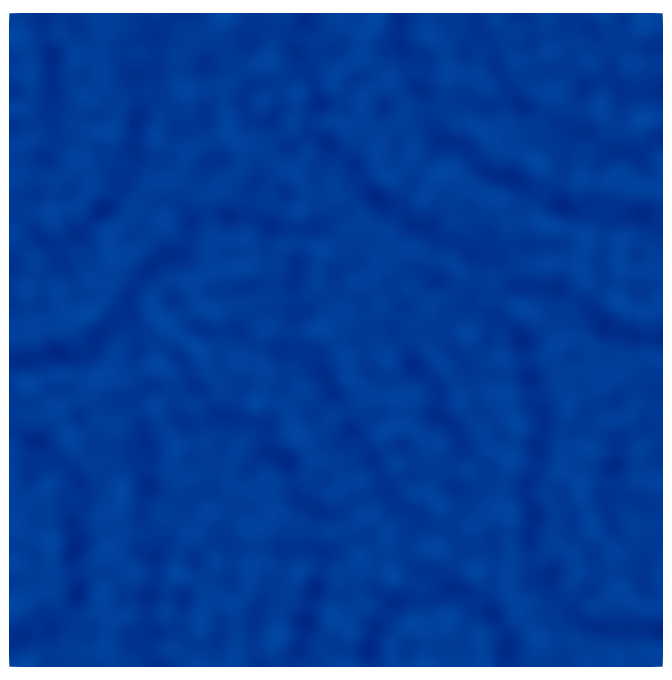}
\includegraphics[angle=-0,width=0.2\textwidth]{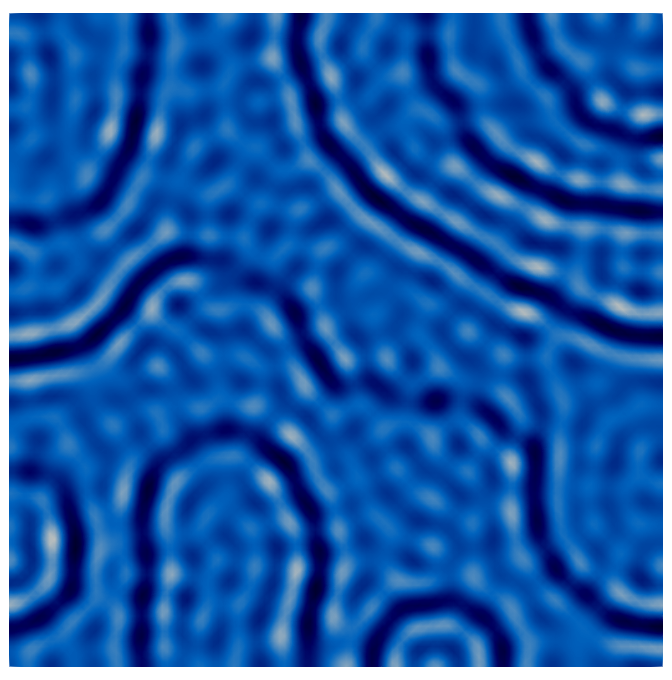}
\includegraphics[angle=-0,width=0.2\textwidth]{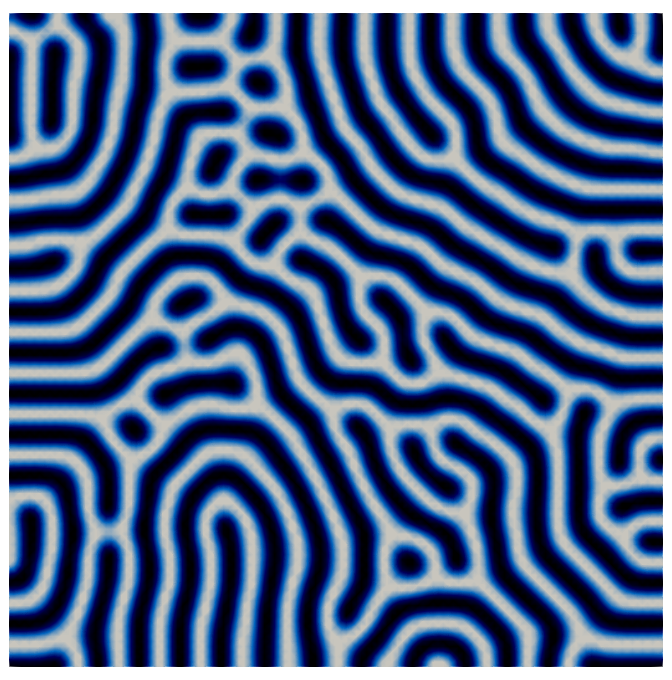}
\includegraphics[angle=-0,width=0.2\textwidth]{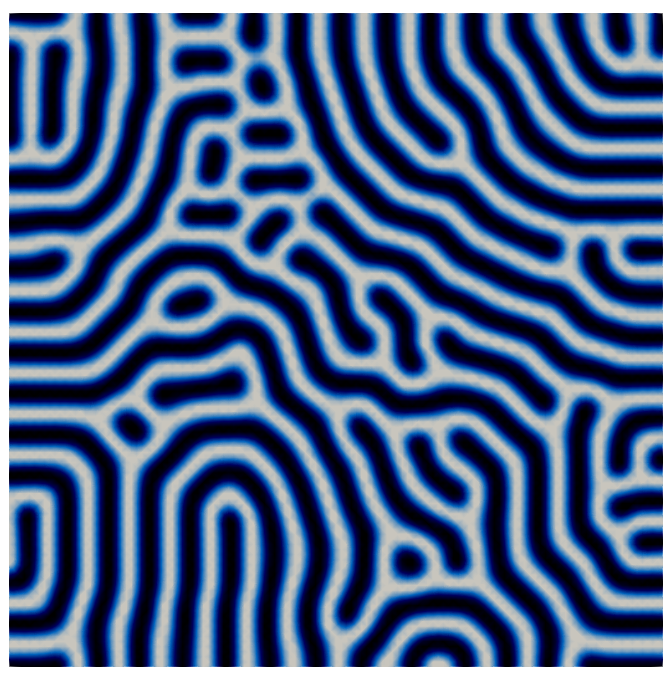}
} 
\caption{The same as Figure~\ref{fig:Fig31_delta0},
but with $g=-300$.
We display $\phi_h^n$ at times $t=0$, $0.001$, $0.005$, $0.05$, $0.5$.
Below we show $\psi_h^n$ at the same times.
}
\label{fig:Fig31d0_g-300}
\end{figure}%
\begin{figure}
\center
\mbox{
\includegraphics[angle=-0,width=0.2\textwidth]{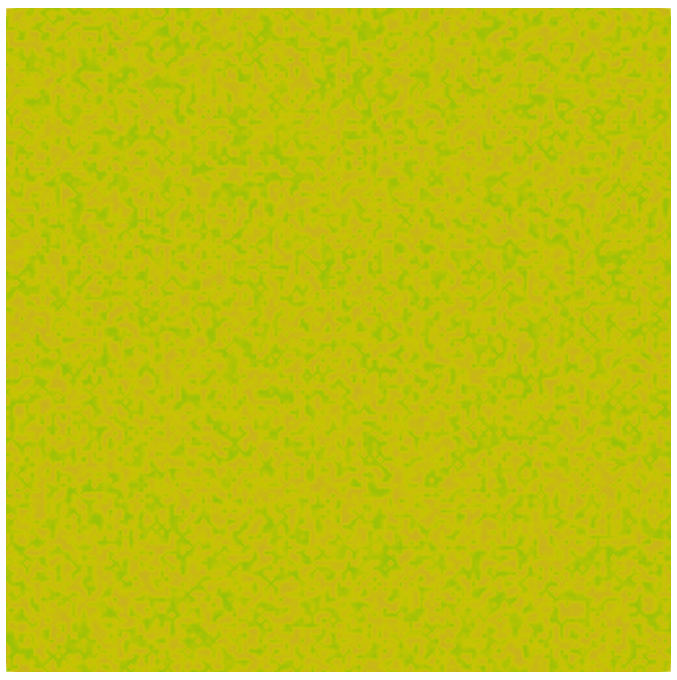}
\includegraphics[angle=-0,width=0.2\textwidth]{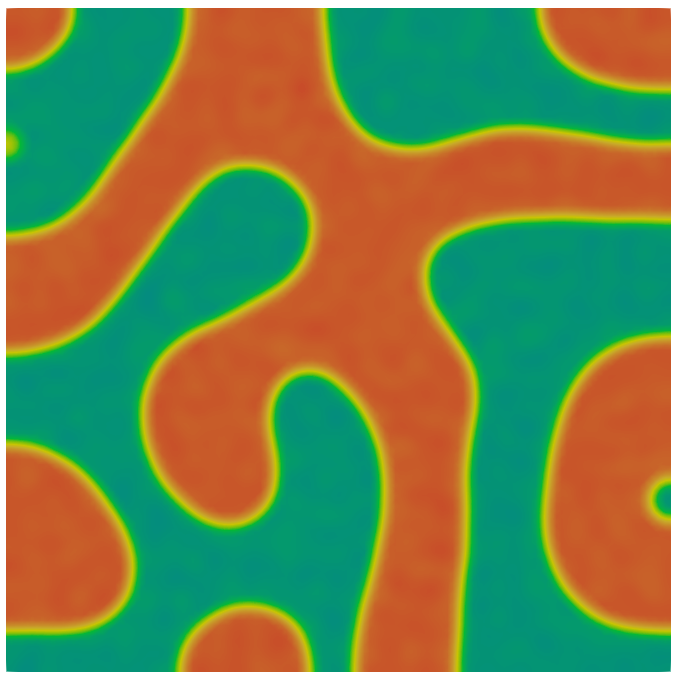}
\includegraphics[angle=-0,width=0.2\textwidth]{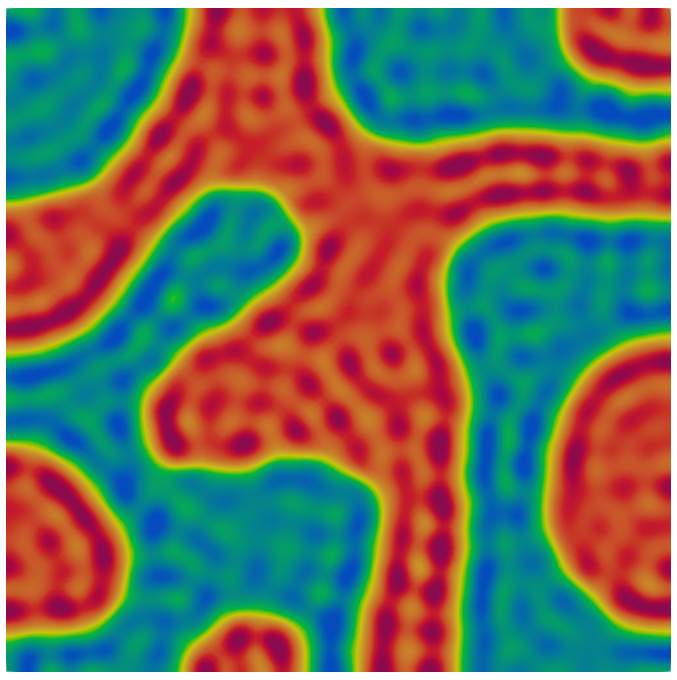}
\includegraphics[angle=-0,width=0.2\textwidth]{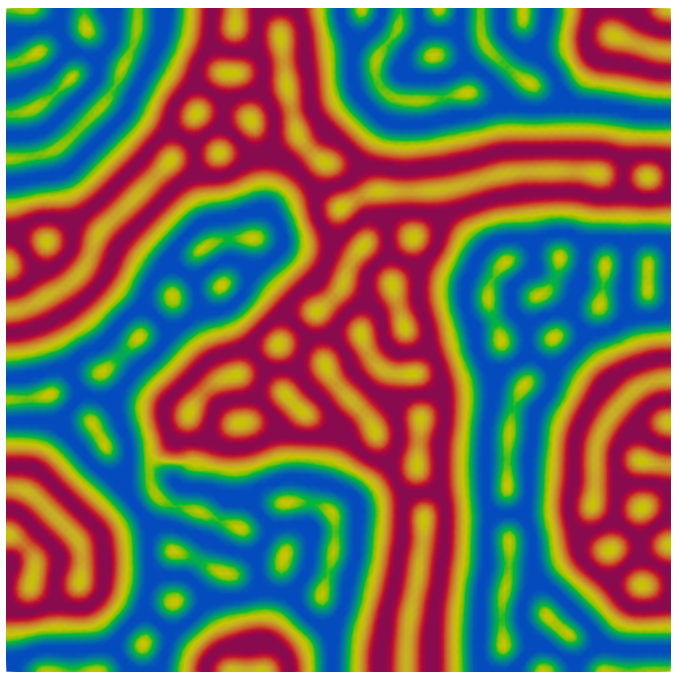}
\includegraphics[angle=-0,width=0.2\textwidth]{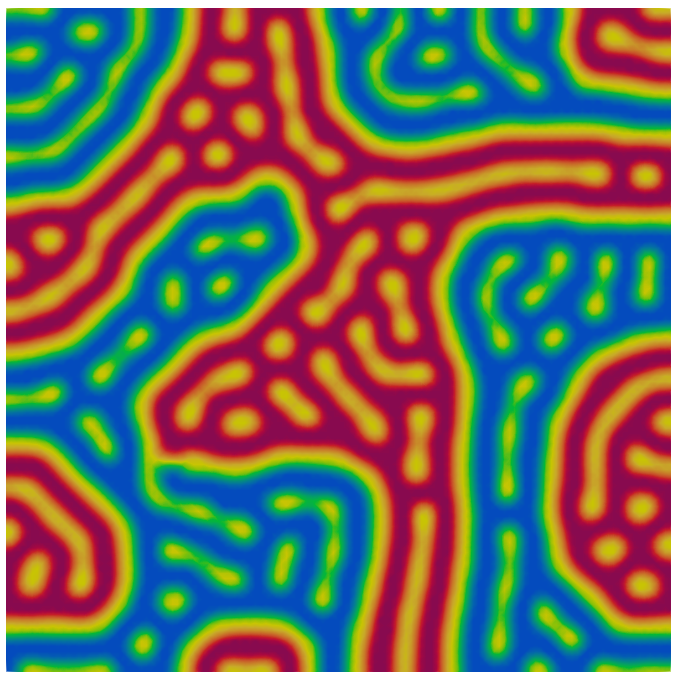}
} \\
\mbox{
\includegraphics[angle=-0,width=0.2\textwidth]{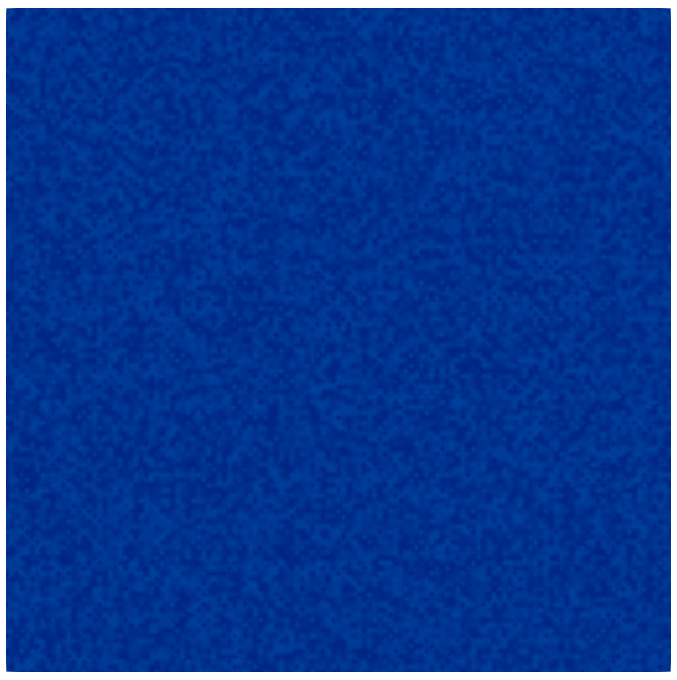}
\includegraphics[angle=-0,width=0.2\textwidth]{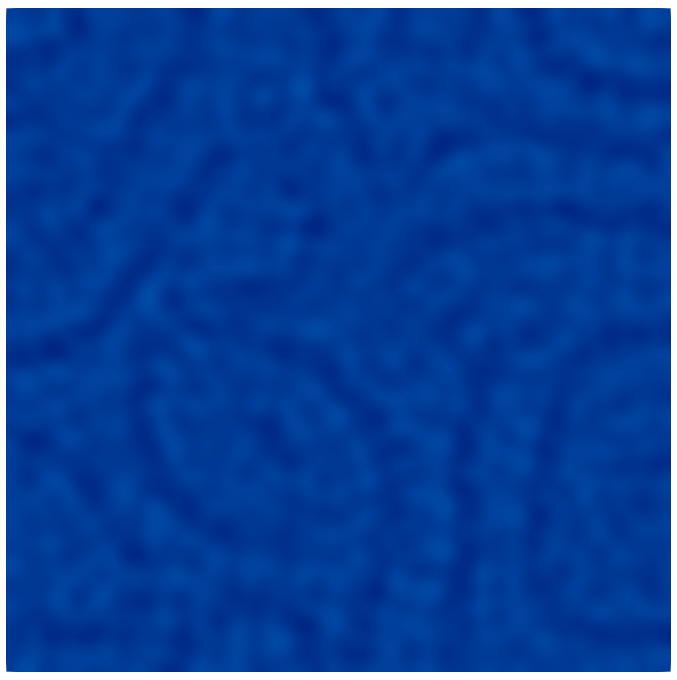}
\includegraphics[angle=-0,width=0.2\textwidth]{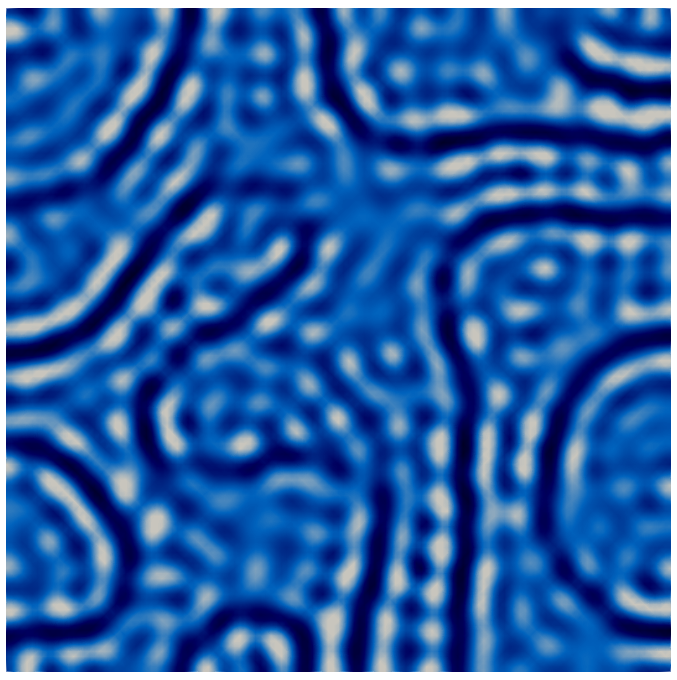}
\includegraphics[angle=-0,width=0.2\textwidth]{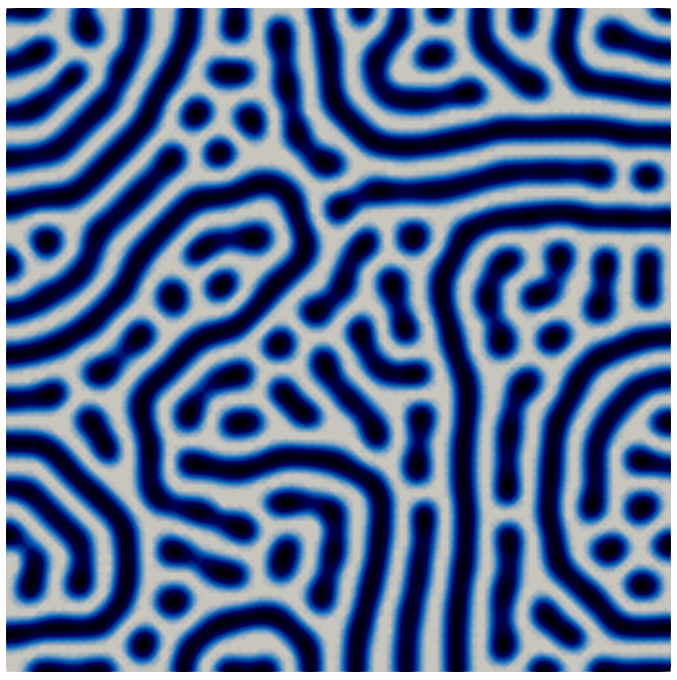}
\includegraphics[angle=-0,width=0.2\textwidth]{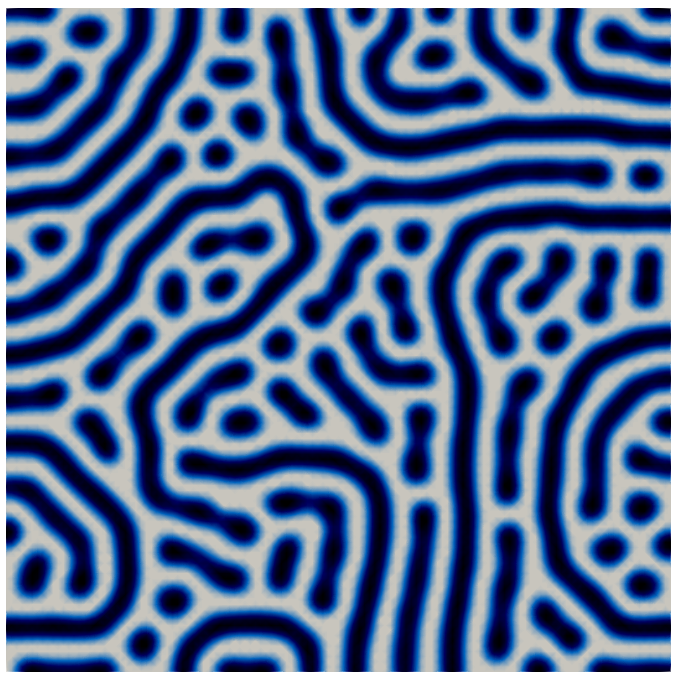}
} 
\caption{The same as Figure~\ref{fig:Fig31_delta0},
but with $g=-1000$.
We display $\phi_h^n$ at times $t=0$, $0.001$, $0.005$, $0.05$, $0.5$.
Below we show $\psi_h^n$ at the same times.
}
\label{fig:Fig31d0_g-1000}
\end{figure}%
\begin{figure}
\center
\mbox{
\includegraphics[angle=-0,width=0.2\textwidth]{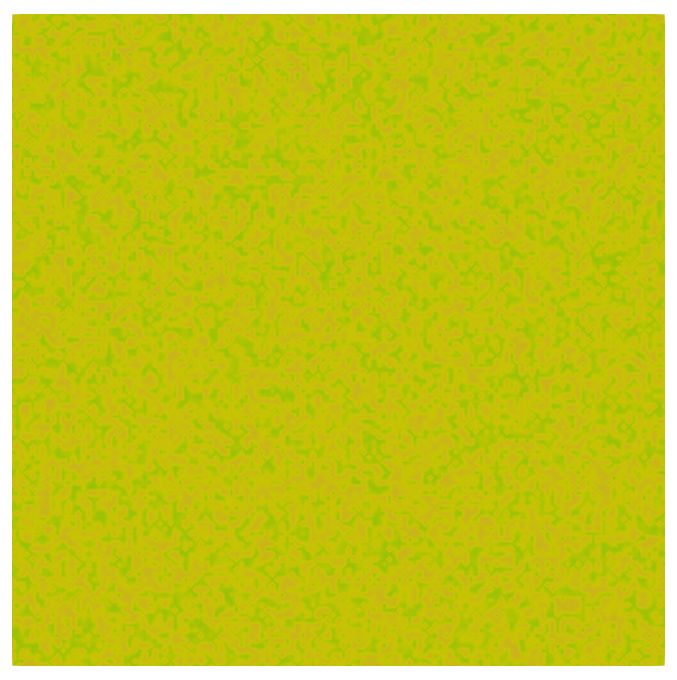}
\includegraphics[angle=-0,width=0.2\textwidth]{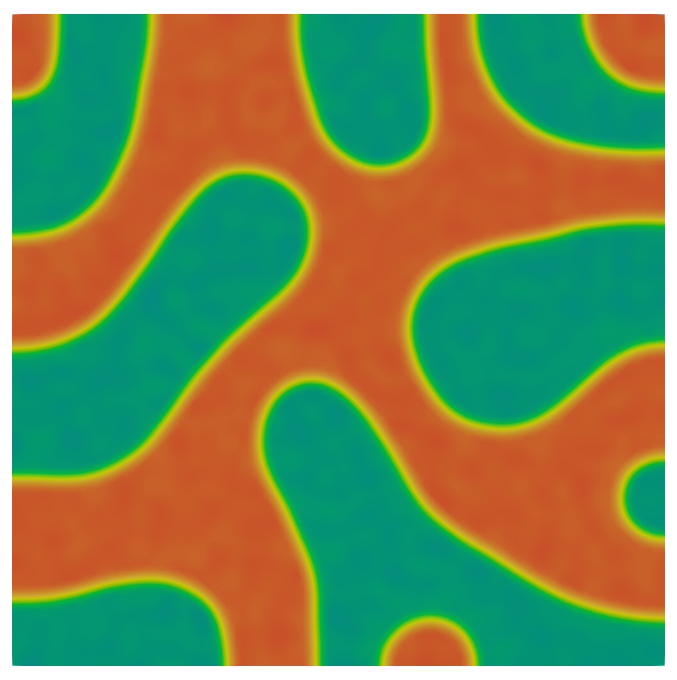}
\includegraphics[angle=-0,width=0.2\textwidth]{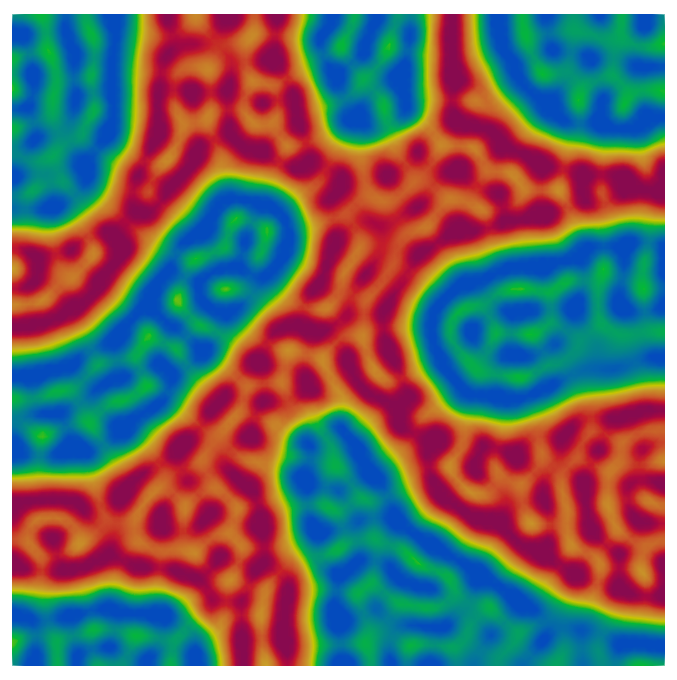}
\includegraphics[angle=-0,width=0.2\textwidth]{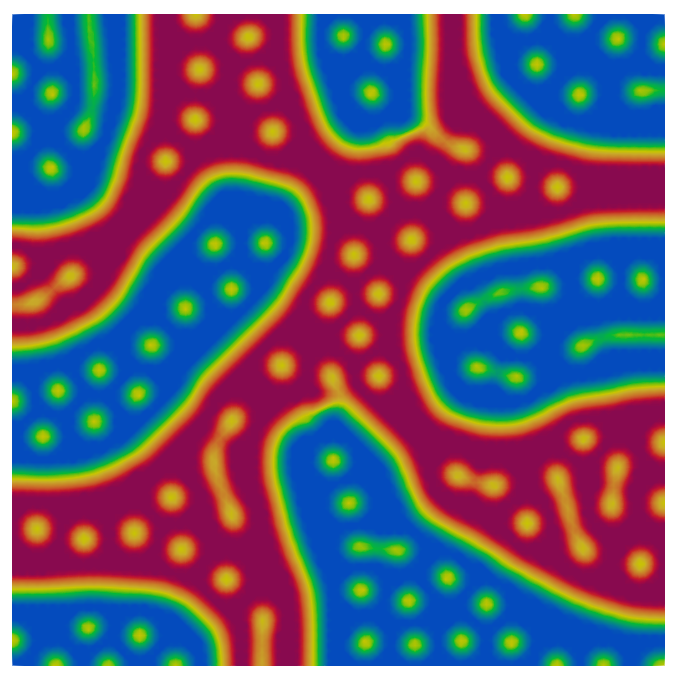}
\includegraphics[angle=-0,width=0.2\textwidth]{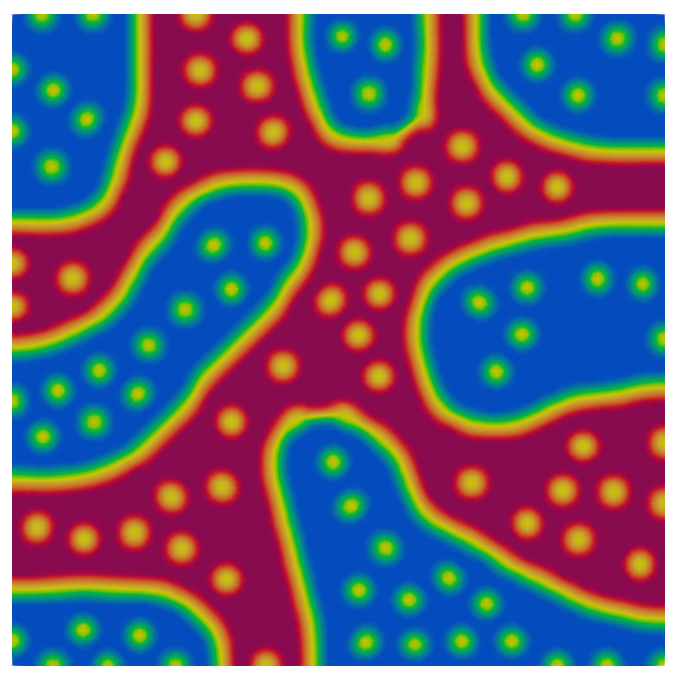}
} \\
\mbox{
\includegraphics[angle=-0,width=0.2\textwidth]{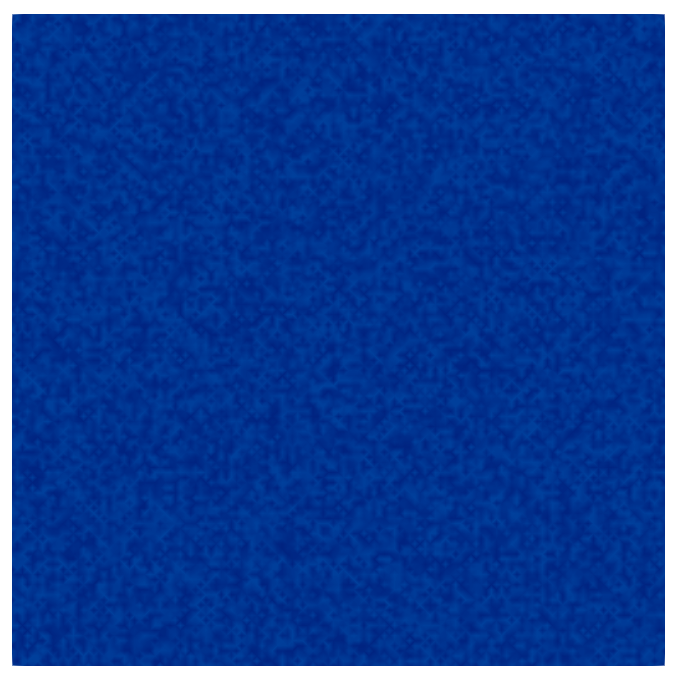}
\includegraphics[angle=-0,width=0.2\textwidth]{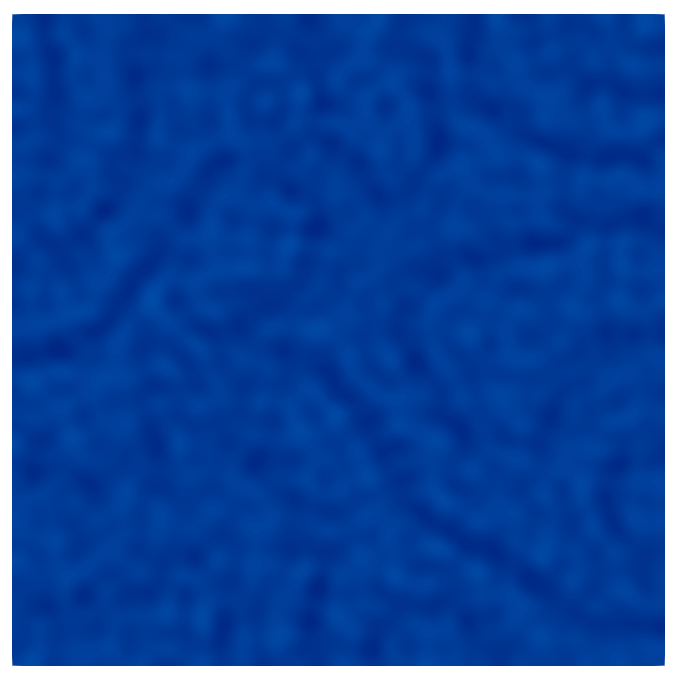}
\includegraphics[angle=-0,width=0.2\textwidth]{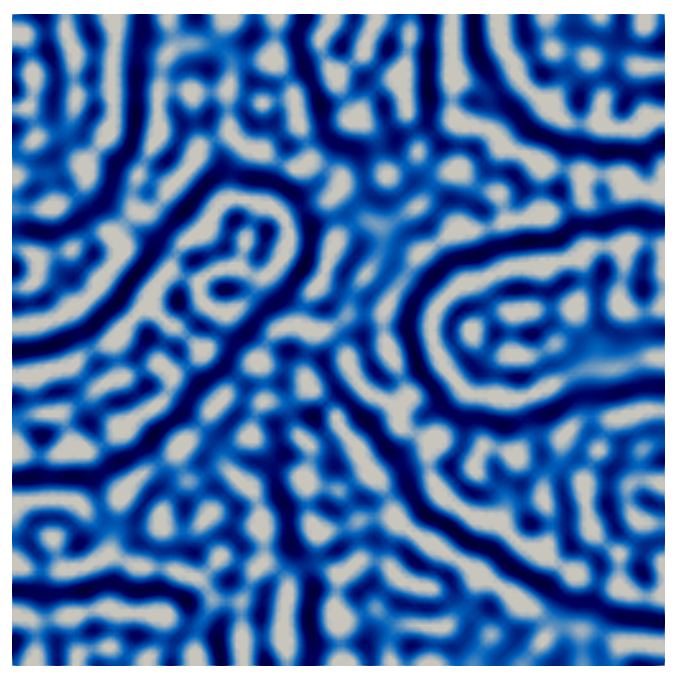}
\includegraphics[angle=-0,width=0.2\textwidth]{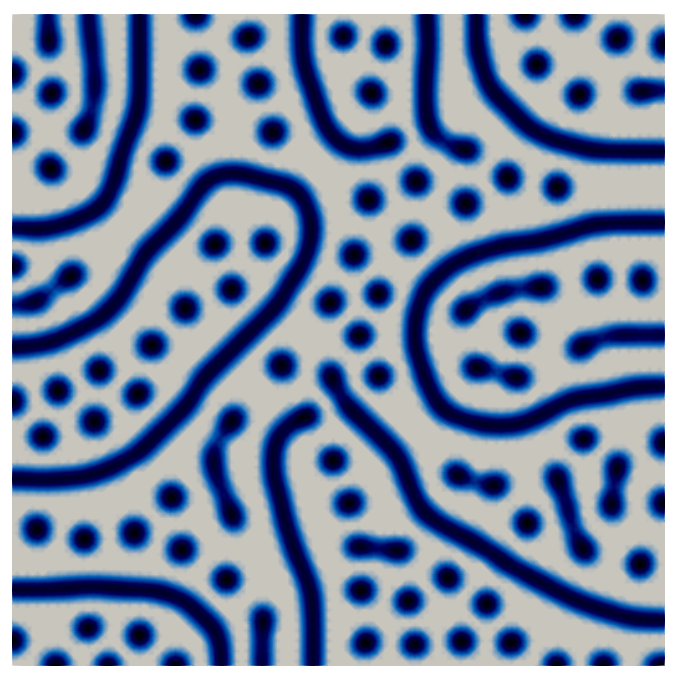}
\includegraphics[angle=-0,width=0.2\textwidth]{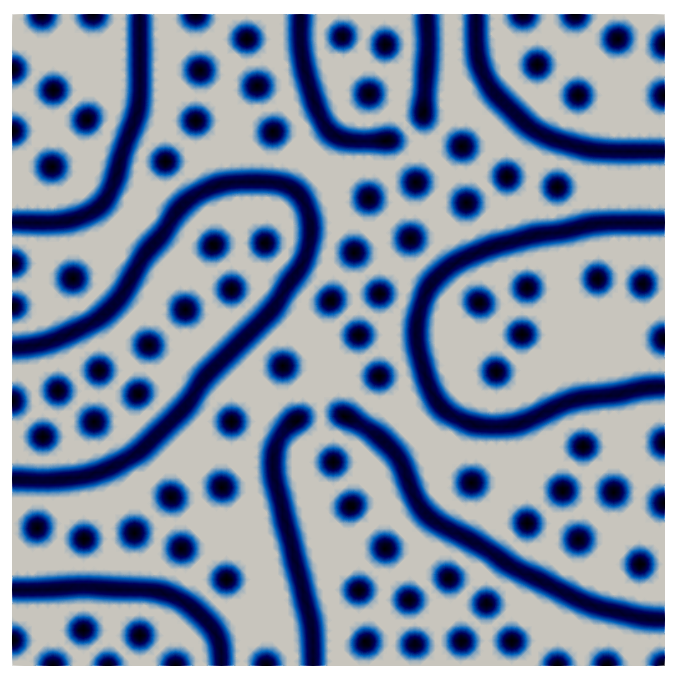}
} 
\caption{The same as Figure~\ref{fig:Fig31_delta0},
but with $g=-2000$.
We display $\phi_h^n$ at times $t=0$, $0.001$, $0.005$, $0.05,\,0.5$.
Below we show $\psi_h^n$ at the same times.
}
\label{fig:Fig31d0_g-2000}
\end{figure}%

Varying the value of $\delta$ leads to the evolution in
Figure~\ref{fig:Fig31d0_delta200} for $\delta=200$, and the evolution in
Figure~\ref{fig:Fig31d0_delta-100} for $\delta=-100$. It can be seen that in the two pure phases of $\phi$, the value of $\psi$ is very small when $\delta > 0$ while $\psi$ is close to $1$ when $\delta < 0$. This is the expected behavior attributed to the term $\frac{\delta}{2} \phi^2 (\psi - \tfrac{1}{2})$ in the total free energy.
\begin{figure}
\center
\mbox{
\includegraphics[angle=-0,width=0.2\textwidth]{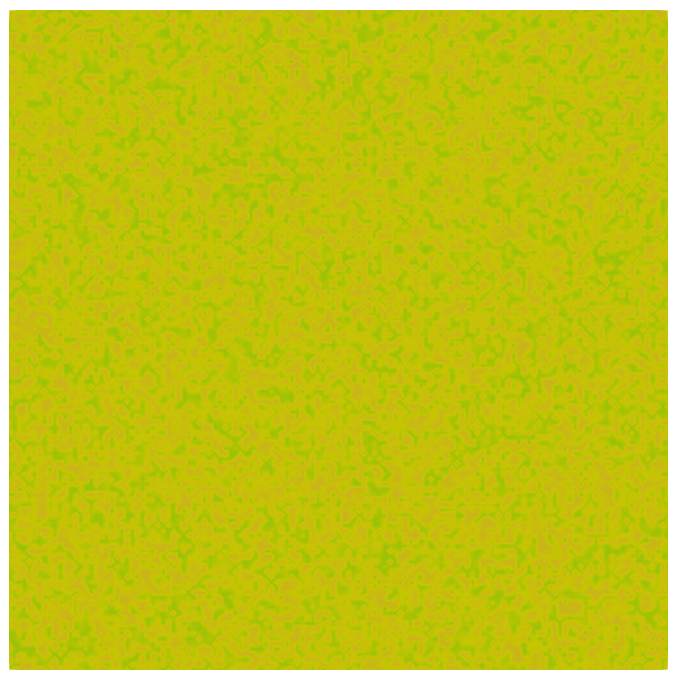}
\includegraphics[angle=-0,width=0.2\textwidth]{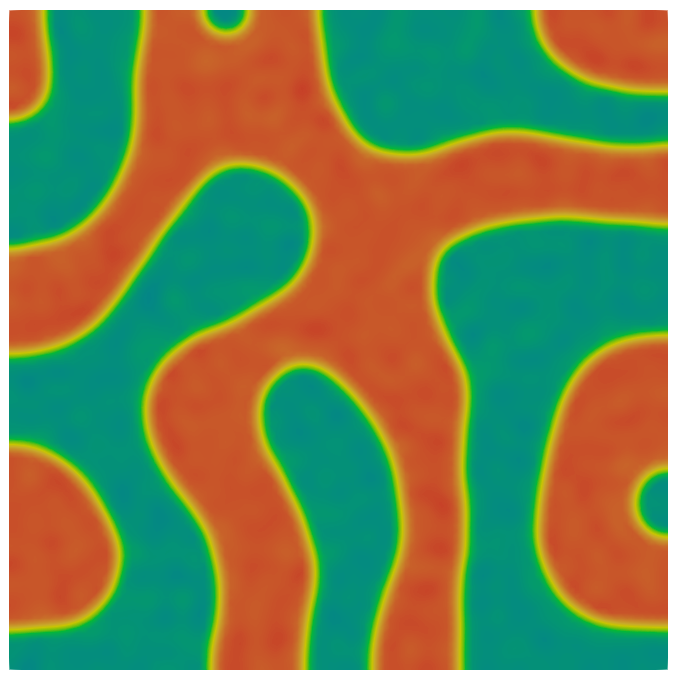}
\includegraphics[angle=-0,width=0.2\textwidth]{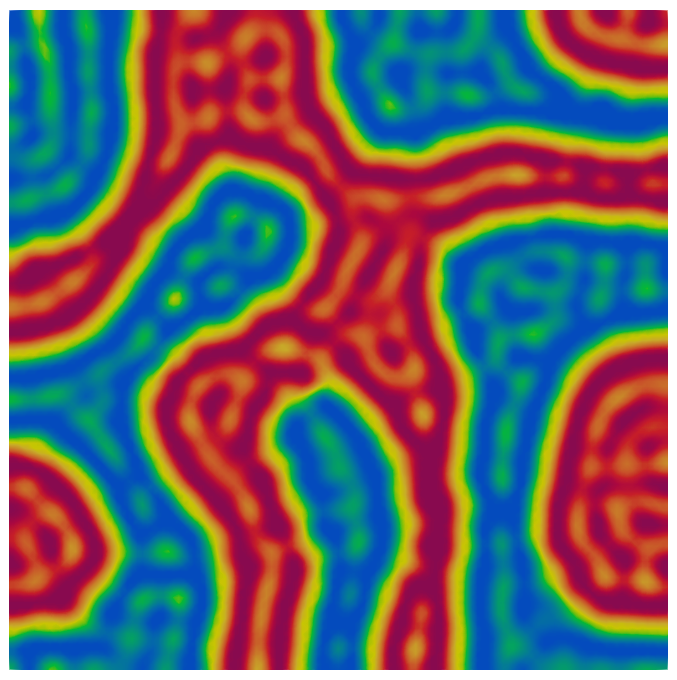}
\includegraphics[angle=-0,width=0.2\textwidth]{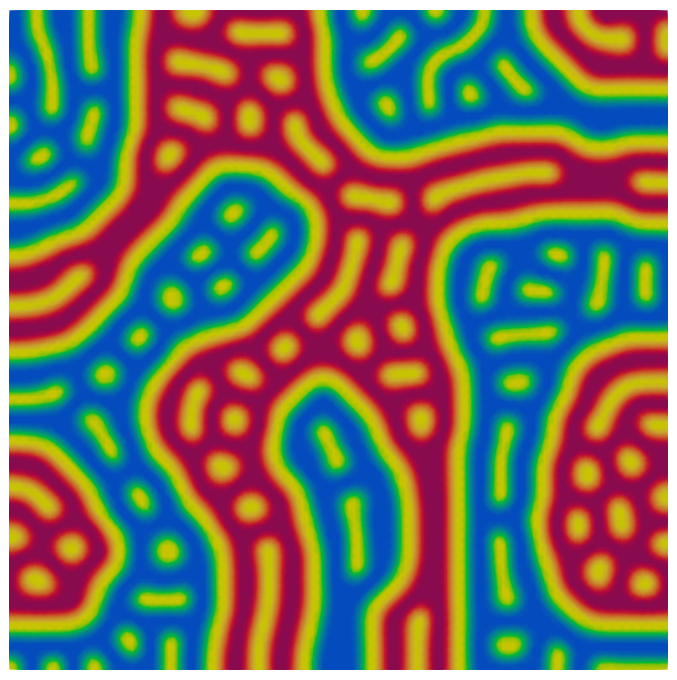}
\includegraphics[angle=-0,width=0.2\textwidth]{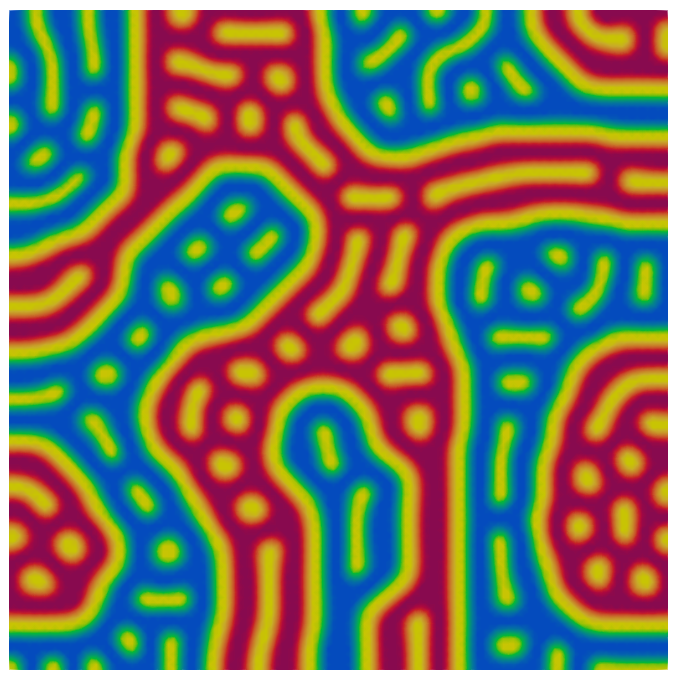}
} \\
\mbox{
\includegraphics[angle=-0,width=0.2\textwidth]{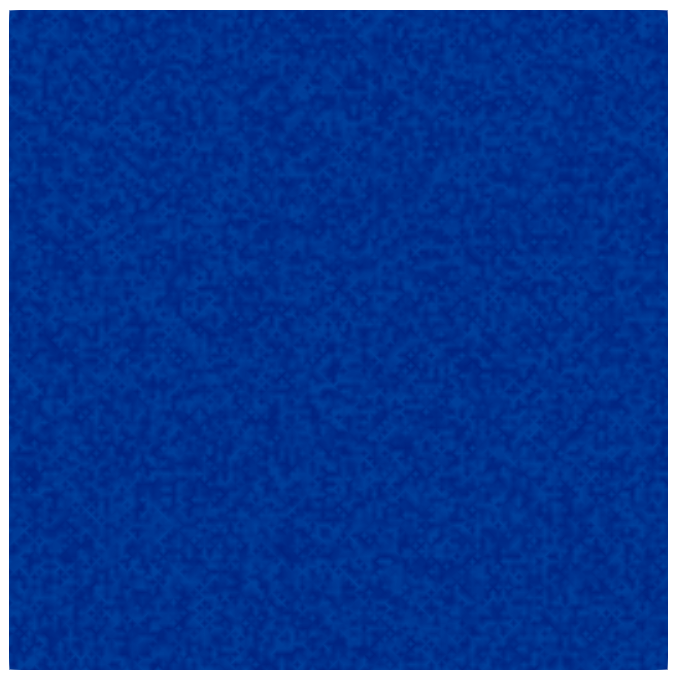}
\includegraphics[angle=-0,width=0.2\textwidth]{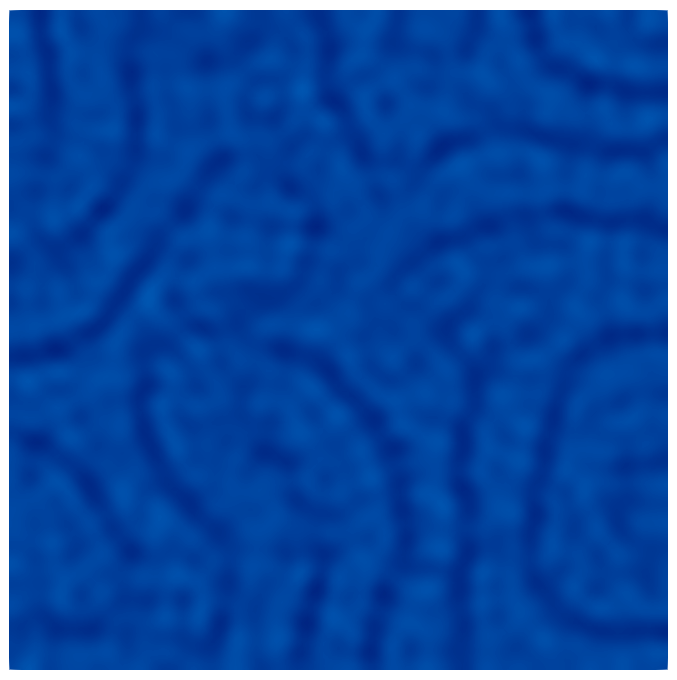}
\includegraphics[angle=-0,width=0.2\textwidth]{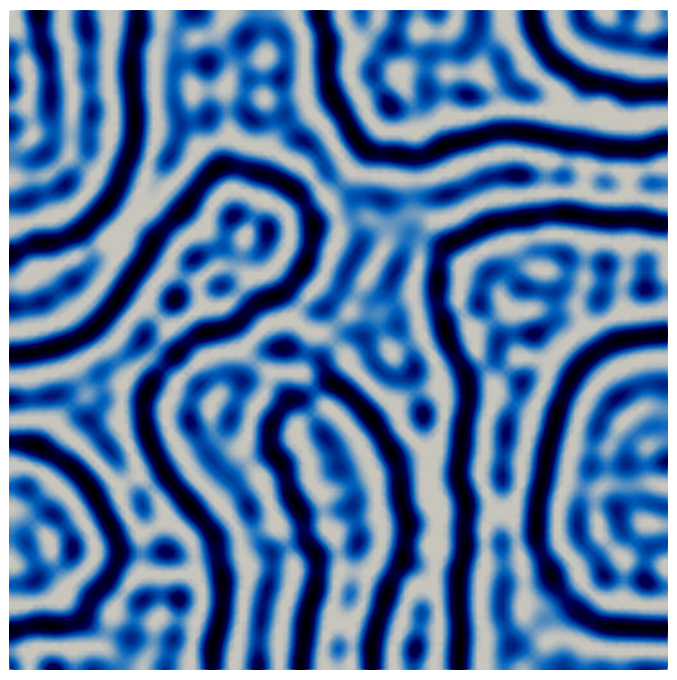}
\includegraphics[angle=-0,width=0.2\textwidth]{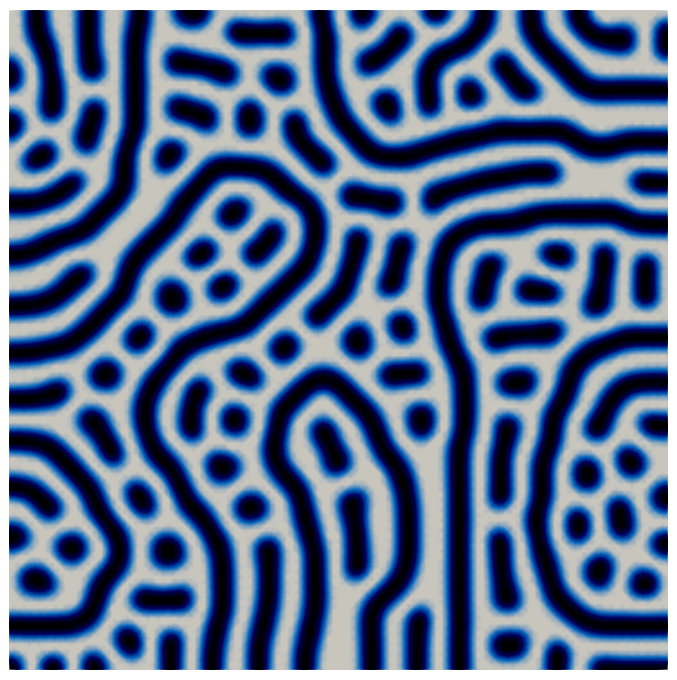}
\includegraphics[angle=-0,width=0.2\textwidth]{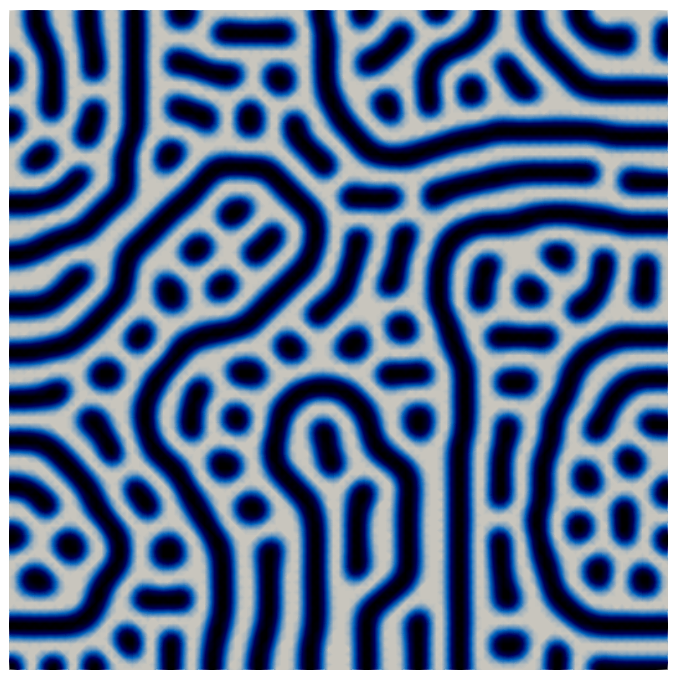}
} 
\caption{The same as Figure~\ref{fig:Fig31_delta0},
but with $\delta=200$.
We display $\phi_h^n$ at times $t=0$, $0.001$, $0.005$, $0.05$, $0.5$.
Below we show $\psi_h^n$ at the same times.
}
\label{fig:Fig31d0_delta200}
\end{figure}%
\begin{figure}
\center
\mbox{
\includegraphics[angle=-0,width=0.2\textwidth]{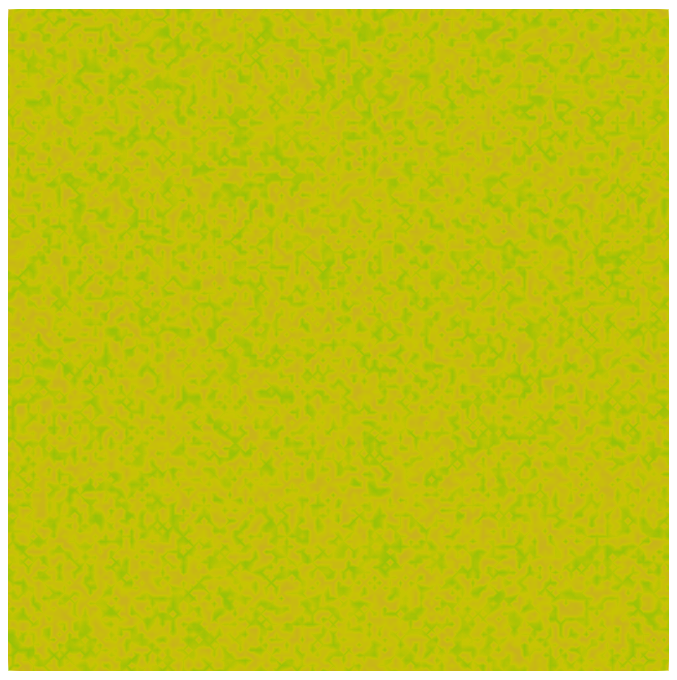}
\includegraphics[angle=-0,width=0.2\textwidth]{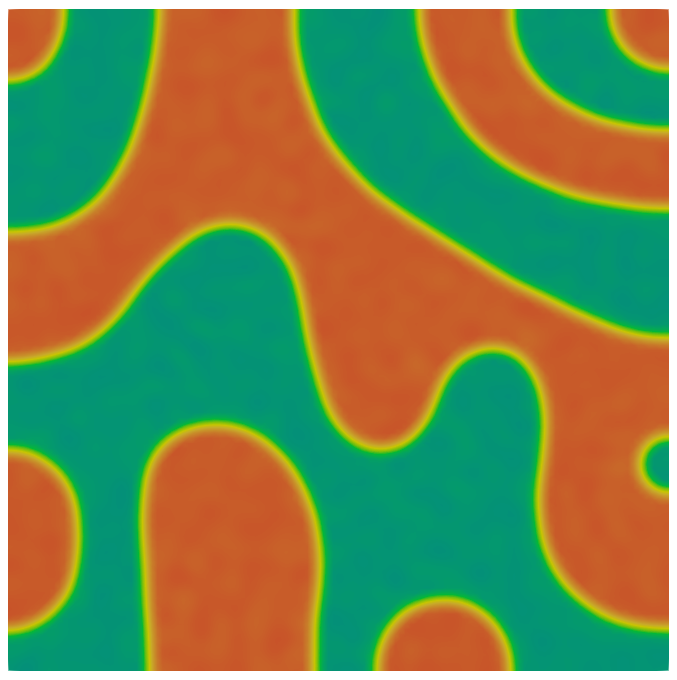}
\includegraphics[angle=-0,width=0.2\textwidth]{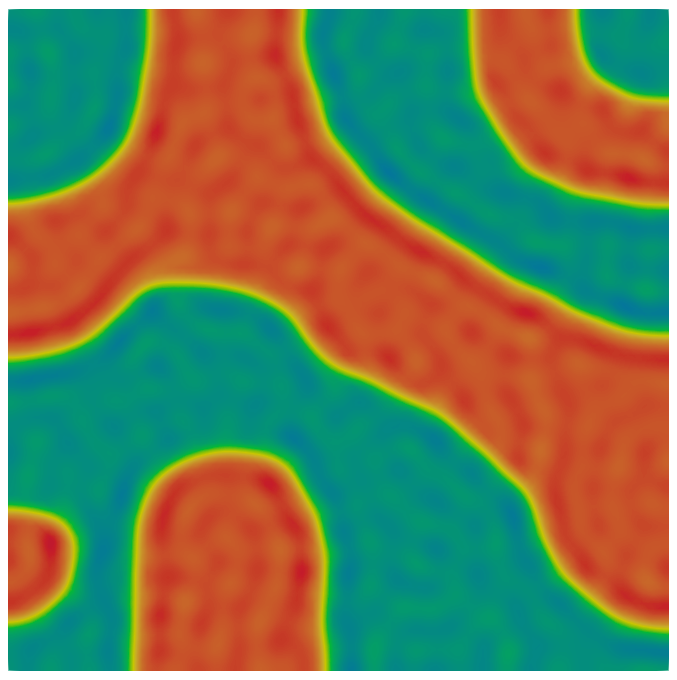}
\includegraphics[angle=-0,width=0.2\textwidth]{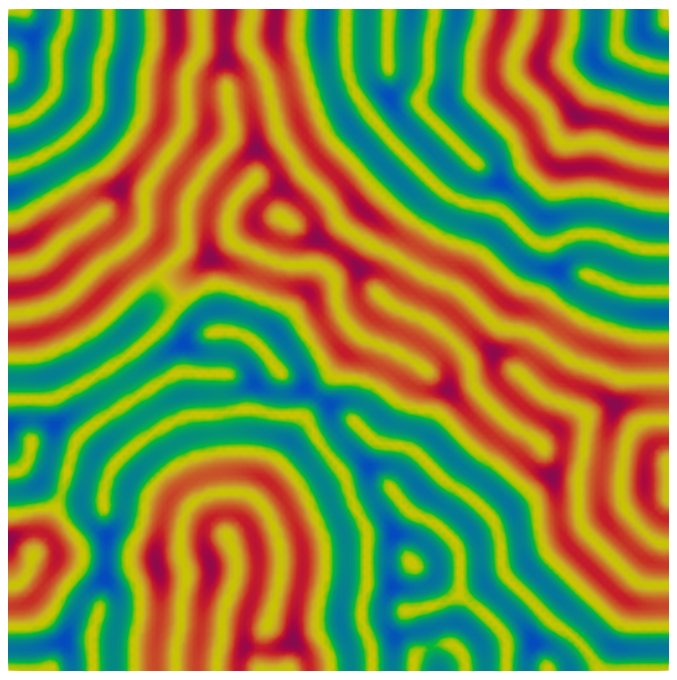}
\includegraphics[angle=-0,width=0.2\textwidth]{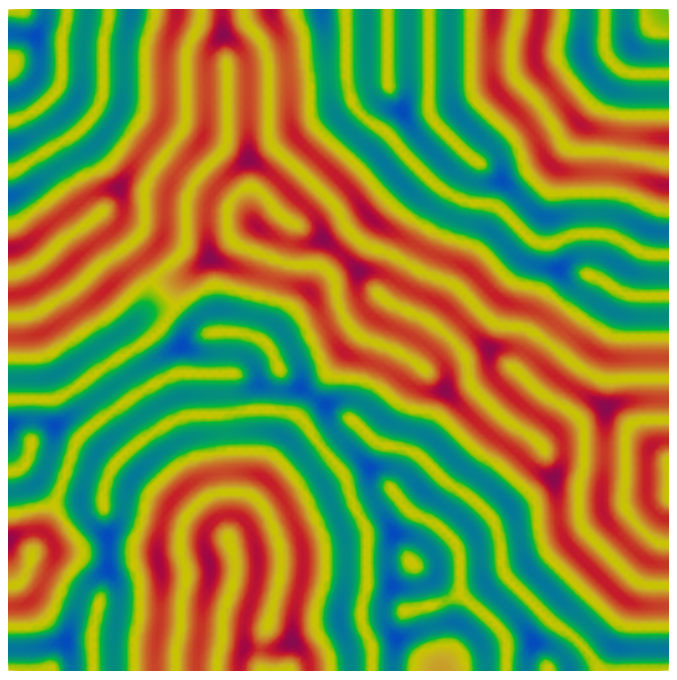}
} \\
\mbox{
\includegraphics[angle=-0,width=0.2\textwidth]{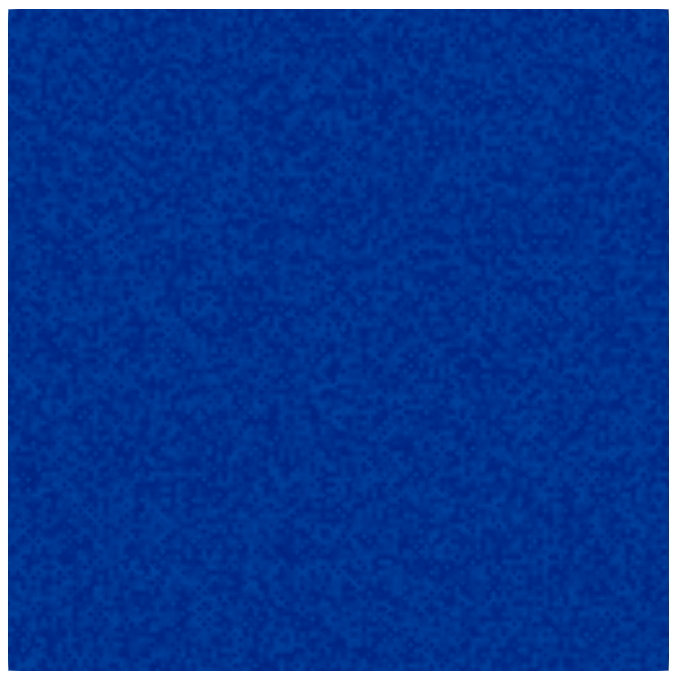}
\includegraphics[angle=-0,width=0.2\textwidth]{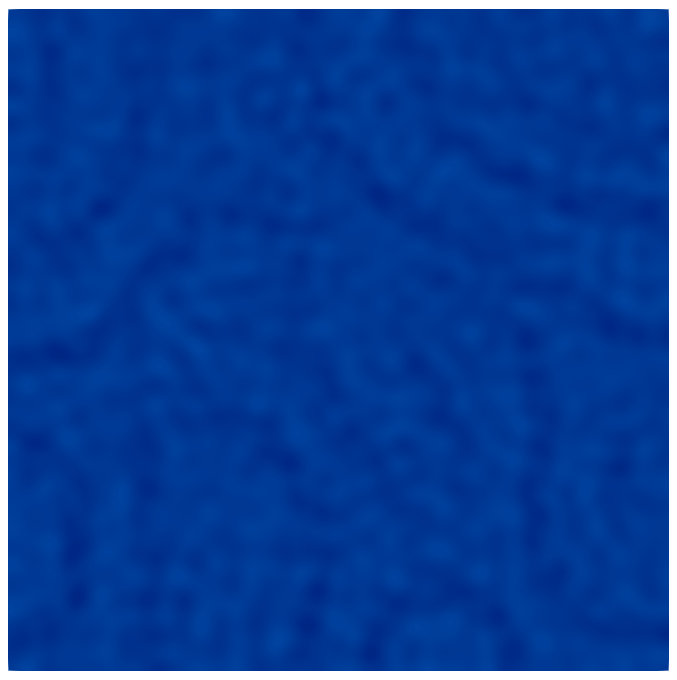}
\includegraphics[angle=-0,width=0.2\textwidth]{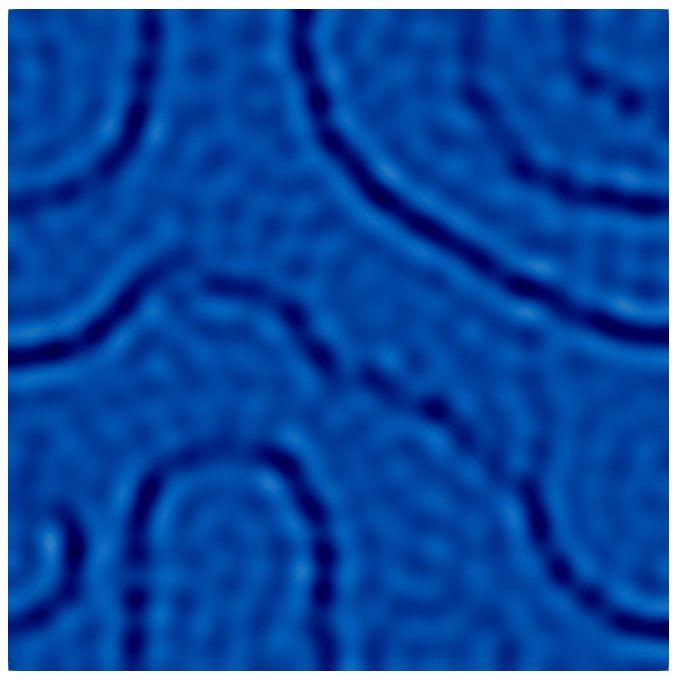}
\includegraphics[angle=-0,width=0.2\textwidth]{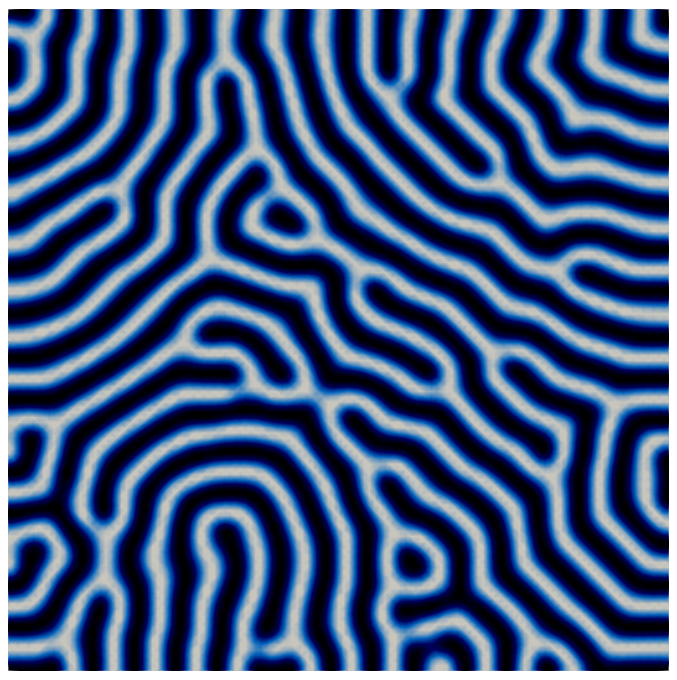}
\includegraphics[angle=-0,width=0.2\textwidth]{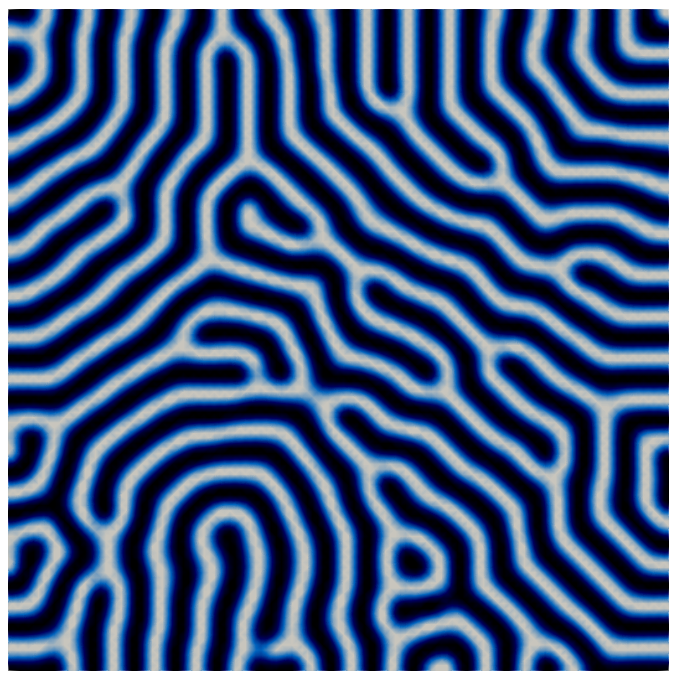}
} 
\caption{The same as Figure~\ref{fig:Fig31_delta0},
but with $\delta=-100$.
We display $\phi_h^n$ at times $t=0$, $0.001$, $0.005$, $0.05$, $0.5$.
Below we show $\psi_h^n$ at the same times.
}
\label{fig:Fig31d0_delta-100}
\end{figure}%

Finally, a computation for $\sigma=0.05$ can be seen in
Figure~\ref{fig:Fig31d0_sigma5e-2}. 
The presence of the $\sigma$ term in the energy leads to the absence of
oscillations in the phase characterized by $\phi = 1$ and $\psi = 0$.
If we use the larger value $\alpha=200$, then this effects becomes even
more pronounced, see Figure~\ref{fig:Fig31d0_alpha200_sigma5e-2} and compare to Figure~7b in \cite{Morales_JTB}.
\begin{figure}
\center
\mbox{
\includegraphics[angle=-0,width=0.2\textwidth]{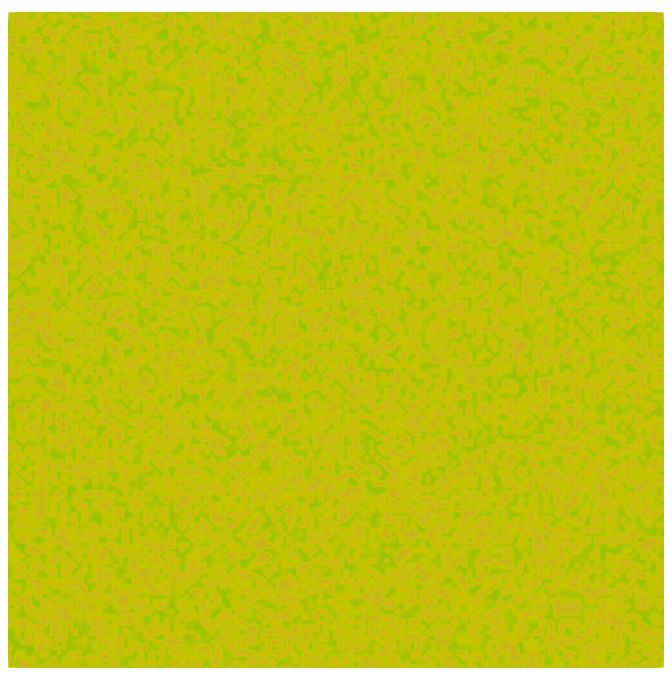}
\includegraphics[angle=-0,width=0.2\textwidth]{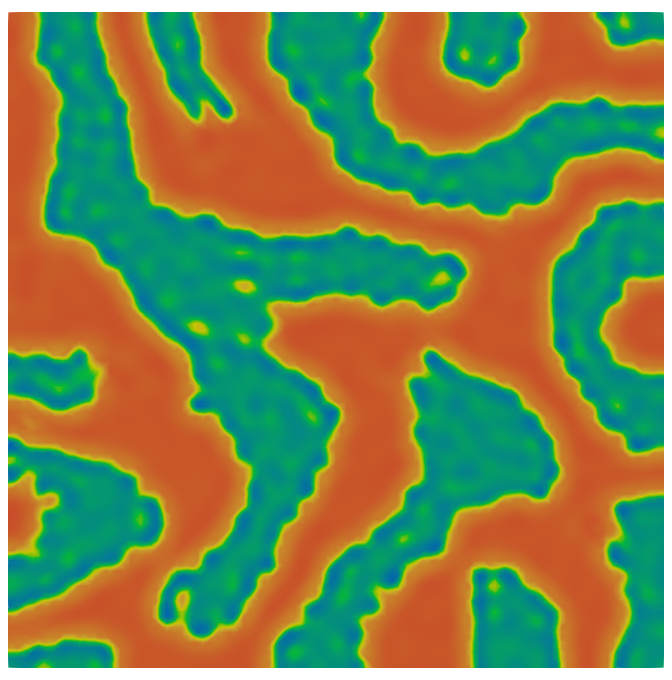}
\includegraphics[angle=-0,width=0.2\textwidth]{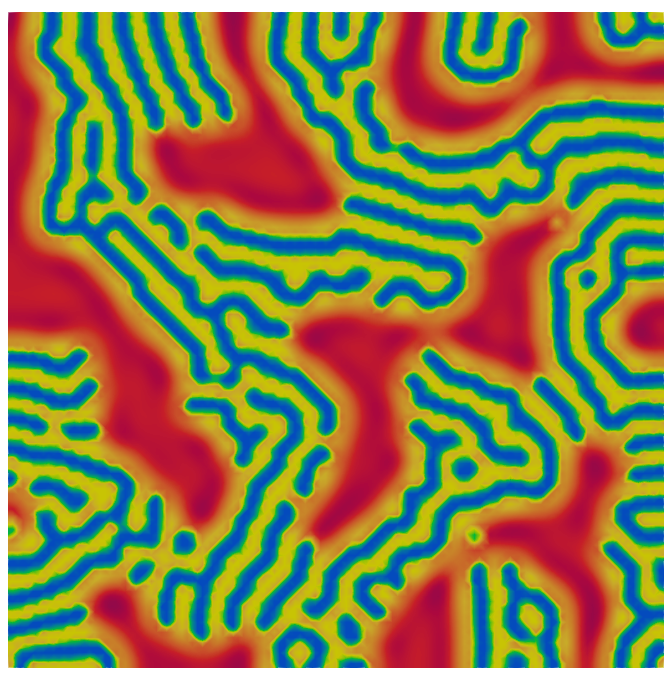}
\includegraphics[angle=-0,width=0.2\textwidth]{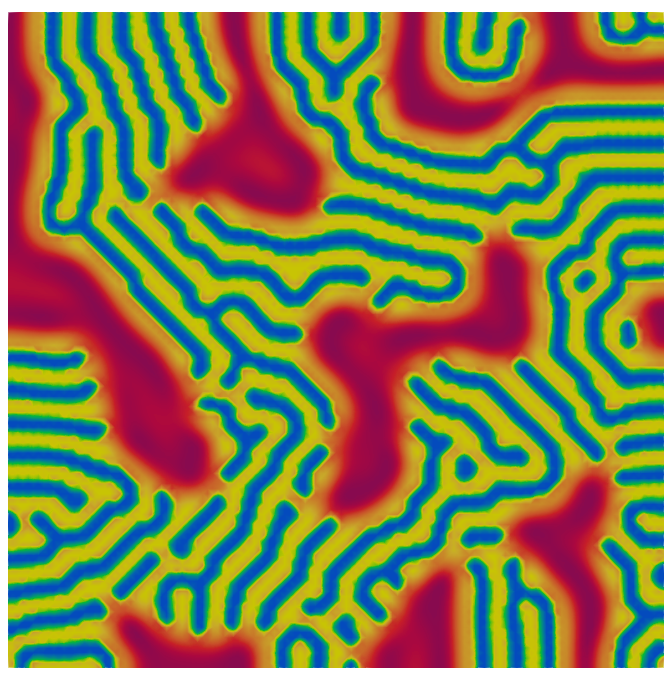}
\includegraphics[angle=-0,width=0.2\textwidth]{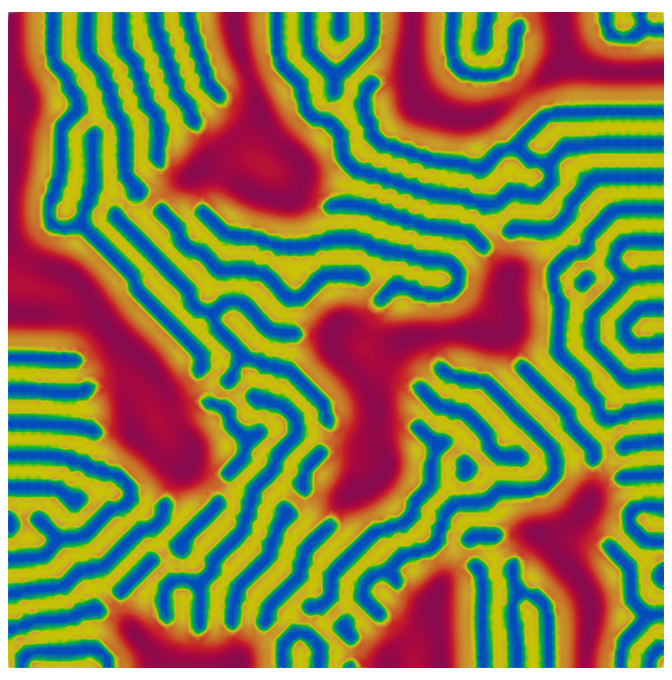}
} \\
\mbox{
\includegraphics[angle=-0,width=0.2\textwidth]{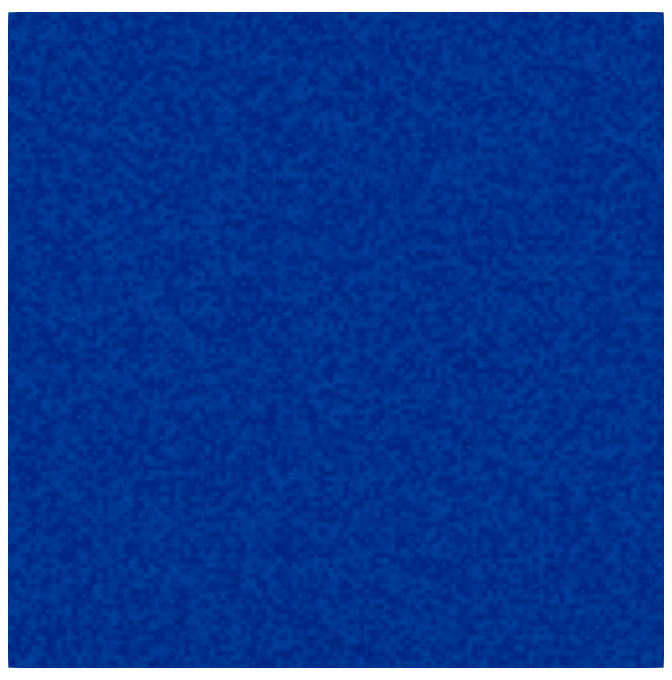}
\includegraphics[angle=-0,width=0.2\textwidth]{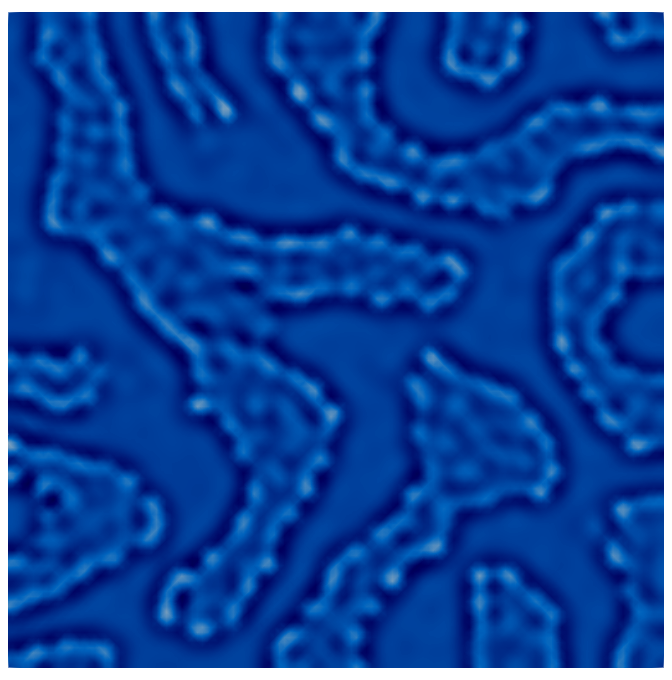}
\includegraphics[angle=-0,width=0.2\textwidth]{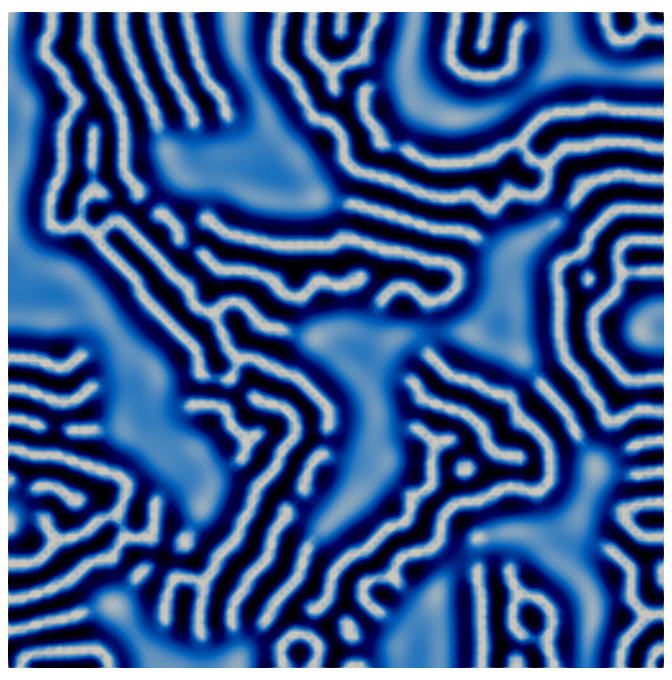}
\includegraphics[angle=-0,width=0.2\textwidth]{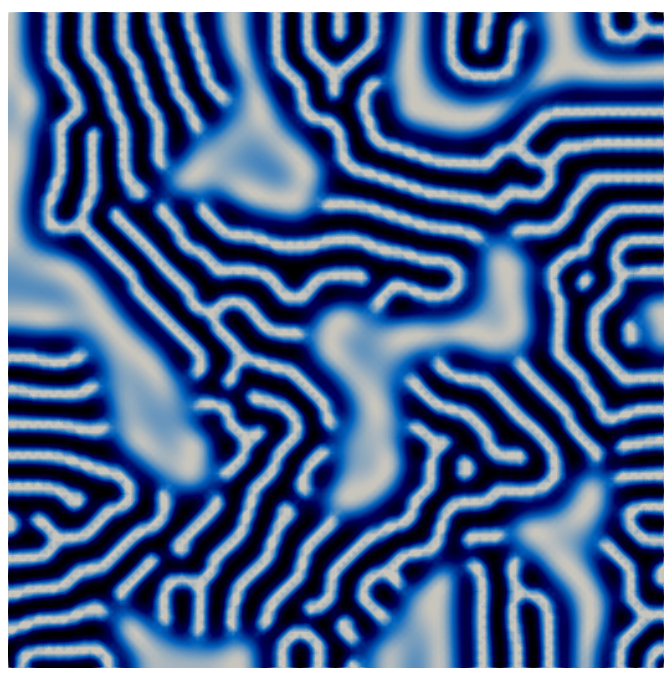}
\includegraphics[angle=-0,width=0.2\textwidth]{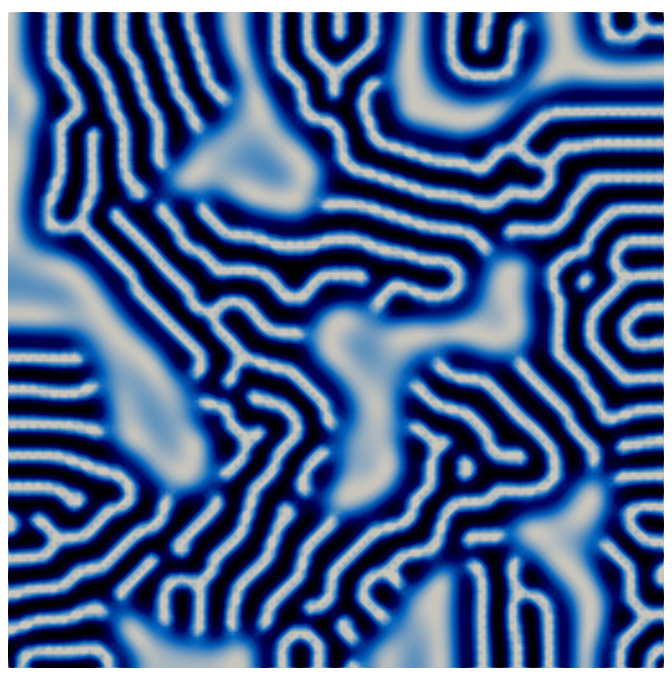}
} 
\caption{The same as Figure~\ref{fig:Fig31_delta0},
but with $\sigma=0.05$.
We display $\phi_h^n$ at times $t=0$, $0.001$, $0.005$, $0.05$, $0.5$.
Below we show $\psi_h^n$ at the same times.
}
\label{fig:Fig31d0_sigma5e-2}
\end{figure}%
\begin{figure}
\center
\mbox{
\includegraphics[angle=-0,width=0.2\textwidth]{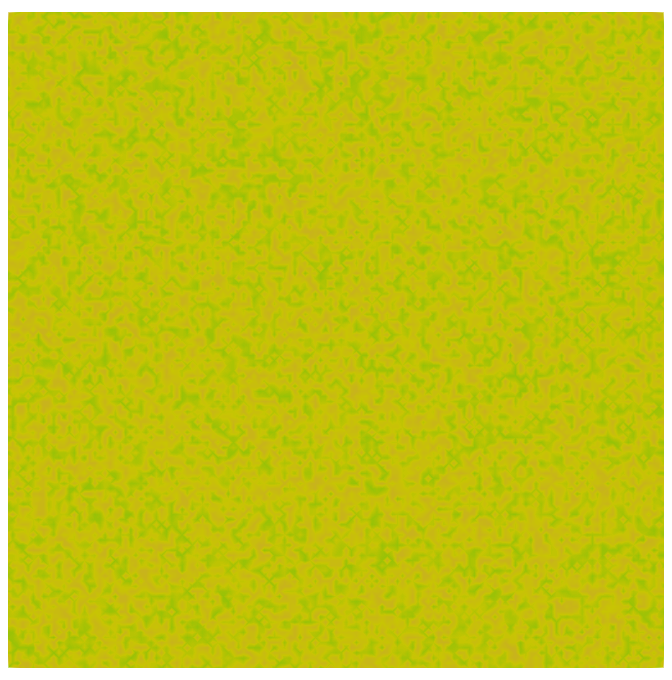}
\includegraphics[angle=-0,width=0.2\textwidth]{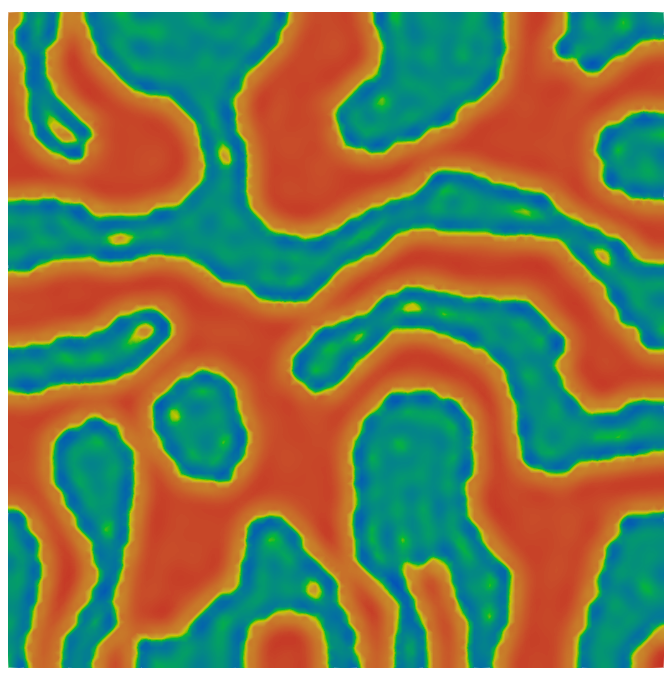}
\includegraphics[angle=-0,width=0.2\textwidth]{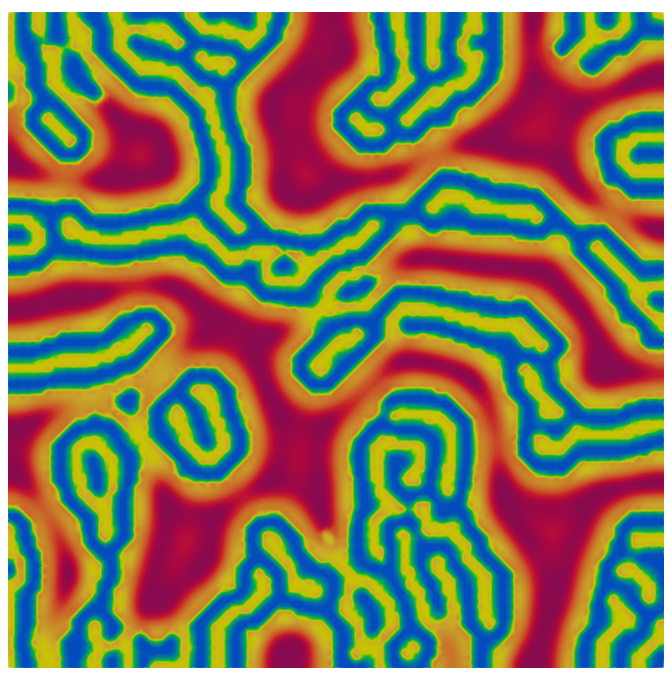}
\includegraphics[angle=-0,width=0.2\textwidth]{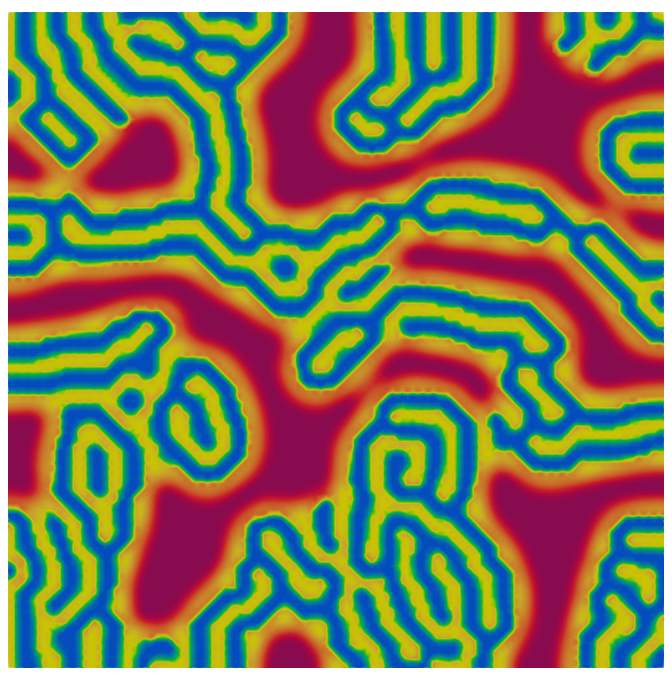}
\includegraphics[angle=-0,width=0.2\textwidth]{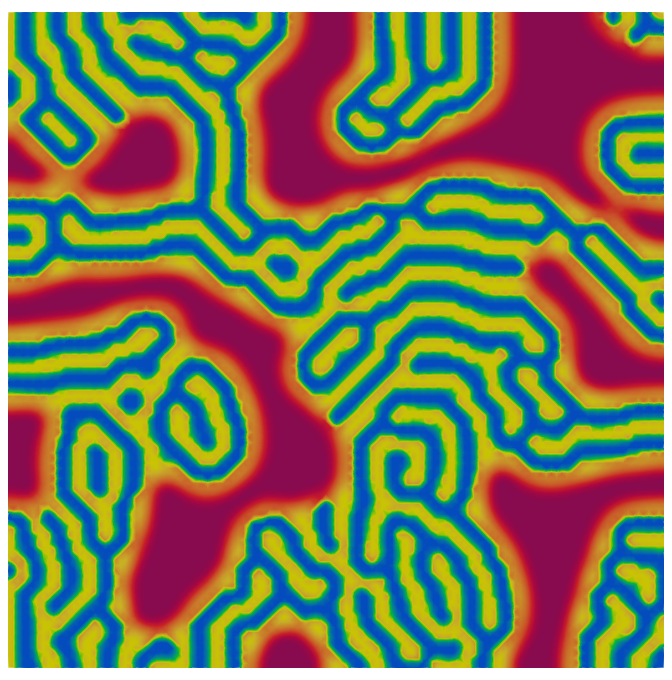}
} \\
\mbox{
\includegraphics[angle=-0,width=0.2\textwidth]{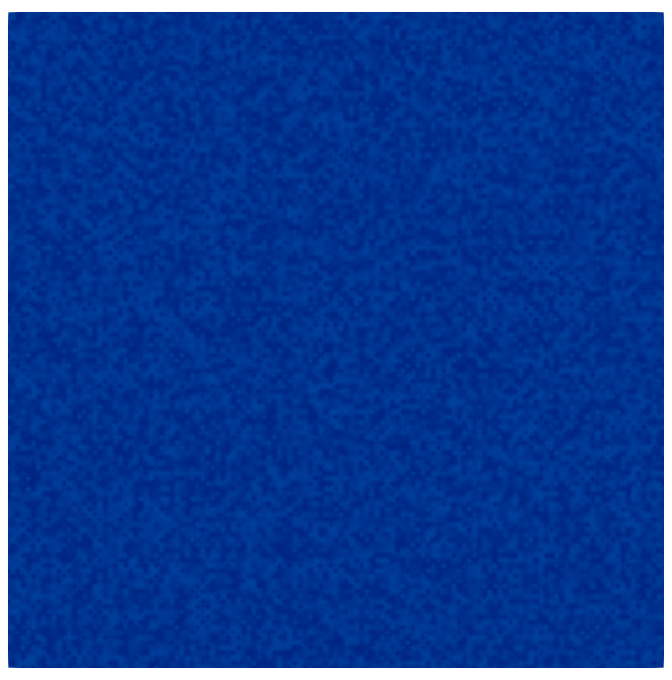}
\includegraphics[angle=-0,width=0.2\textwidth]{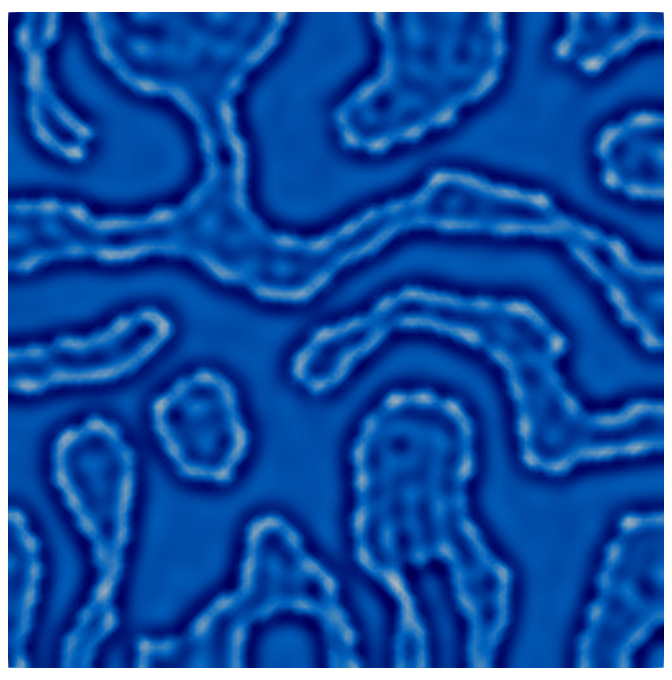}
\includegraphics[angle=-0,width=0.2\textwidth]{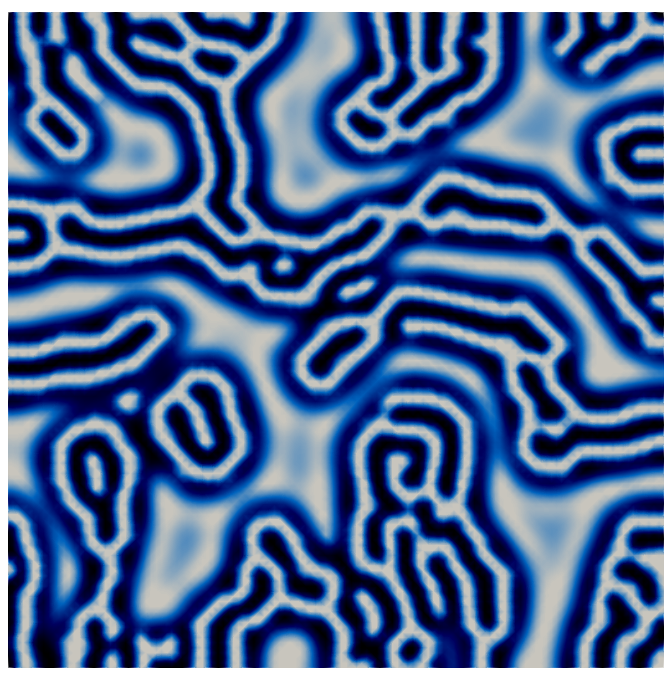}
\includegraphics[angle=-0,width=0.2\textwidth]{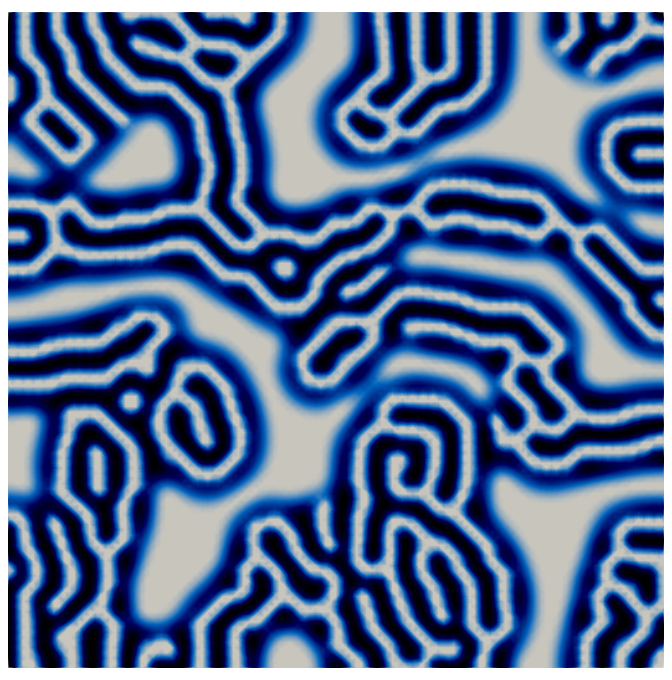}
\includegraphics[angle=-0,width=0.2\textwidth]{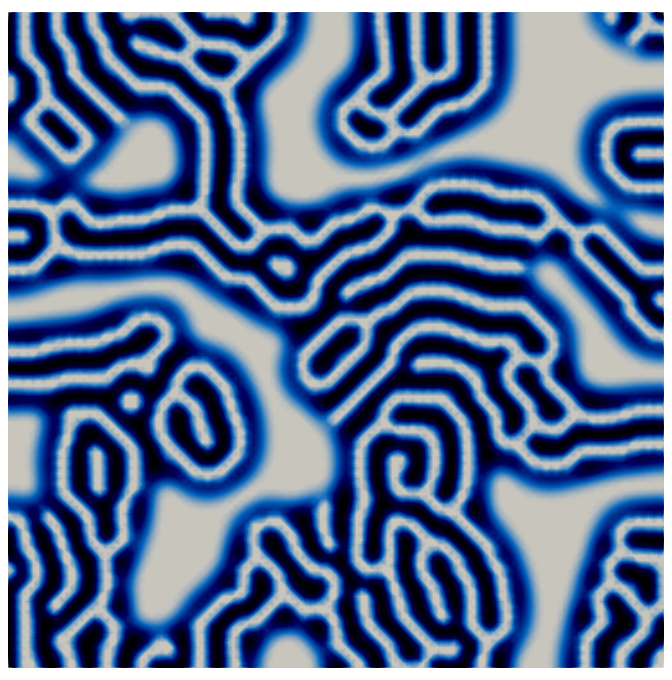}
} 
\caption{The same as Figure~\ref{fig:Fig31_delta0},
but with $\sigma=0.05$ and $\alpha=200$.
We display $\phi_h^n$ at times $t=0$, $0.001$, $0.005$, $0.05$, $0.5$.
Below we show $\psi_h^n$ at the same times.
}
\label{fig:Fig31d0_alpha200_sigma5e-2}
\end{figure}%


We conclude this section with some numerical simulations in three dimensions shown in Figures~\ref{fig:Fig3d_CH}, \ref{fig:Fig3d_SH} and \ref{fig:Fig3d_delta0}. All the
parameters are chosen as in the corresponding two dimensional experiments in
Figures~\ref{fig:Fig32_CH}, \ref{fig:Fig32_delta0} and \ref{fig:Fig32_SH}.
\begin{figure}
\center
\mbox{
\includegraphics[angle=-0,width=0.2\textwidth]{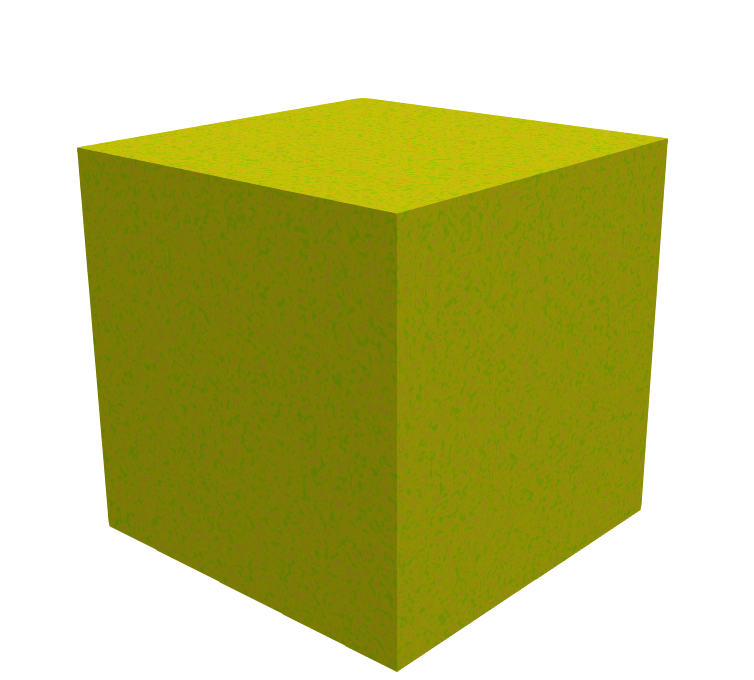}
\includegraphics[angle=-0,width=0.2\textwidth]{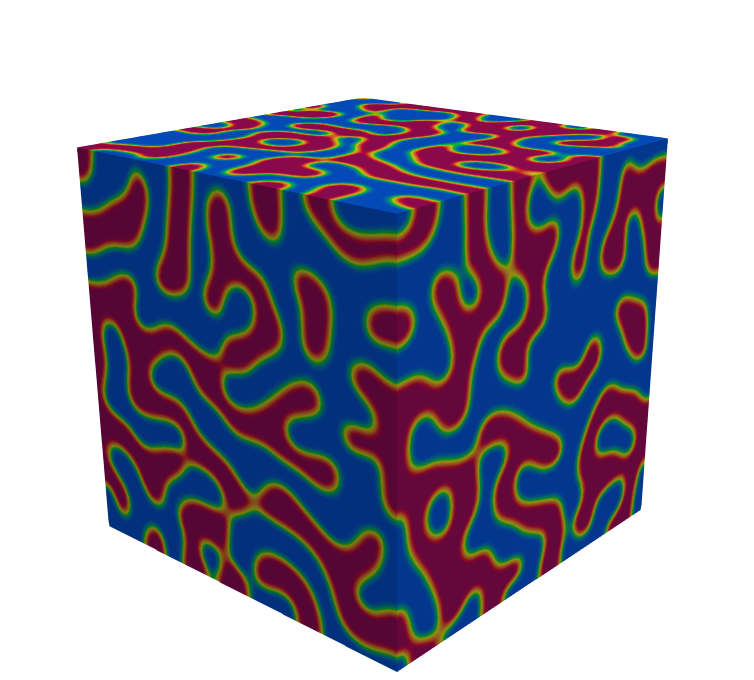}
\includegraphics[angle=-0,width=0.2\textwidth]{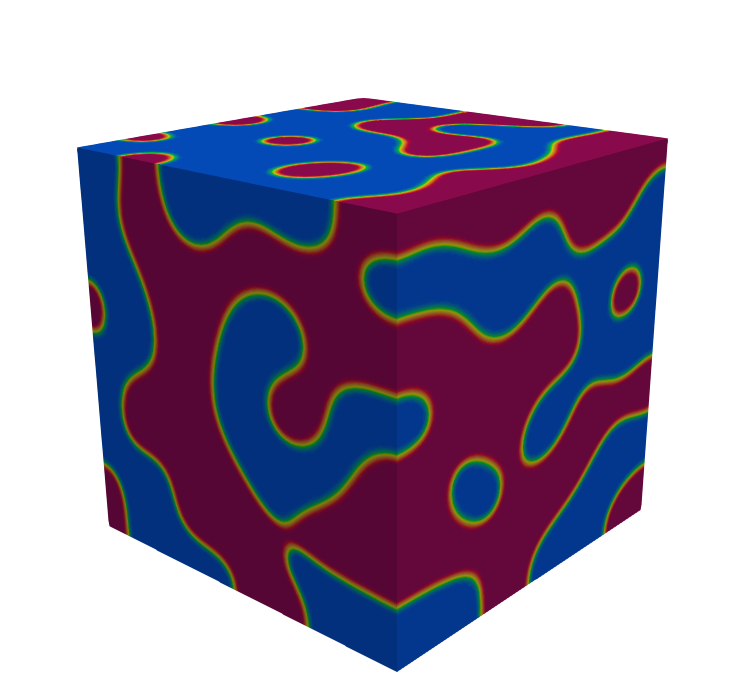}
\includegraphics[angle=-0,width=0.2\textwidth]{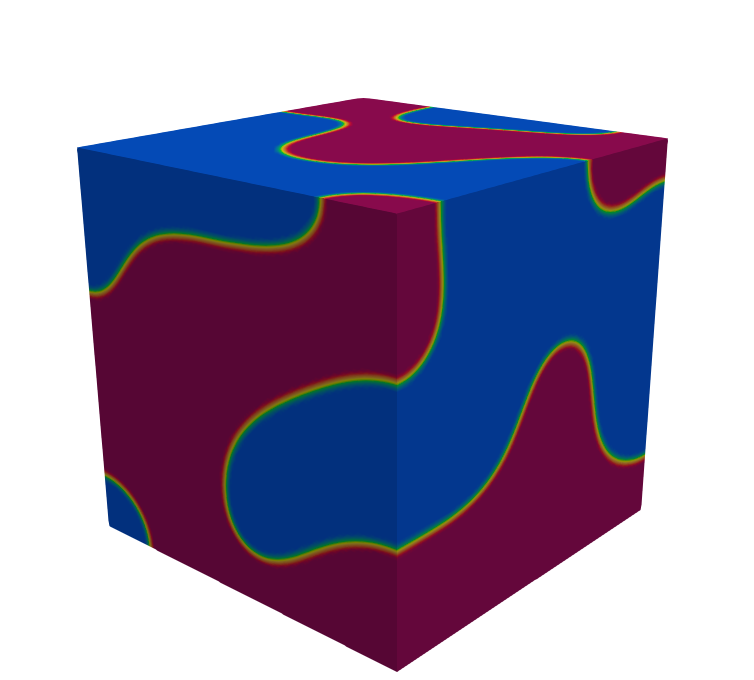}
\includegraphics[angle=-0,width=0.2\textwidth]{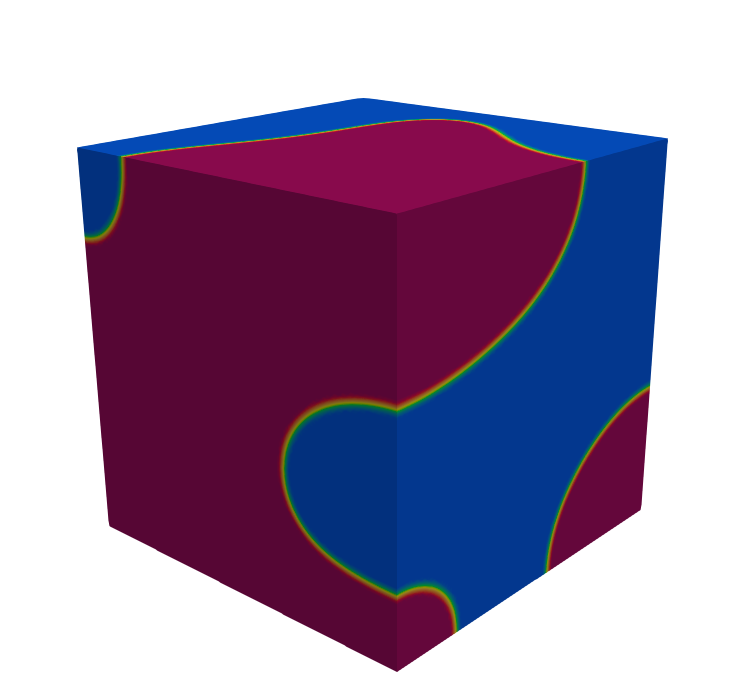}
} 
\caption{Spinodal decomposition for CH.
We display $\phi^n_h$ at times $t=0$, $10^{-4}$, $0.001$, $0.01$, $0.05$.
}
\label{fig:Fig3d_CH}
\end{figure}%
\begin{figure}
\center
\mbox{
\includegraphics[angle=-0,width=0.2\textwidth]{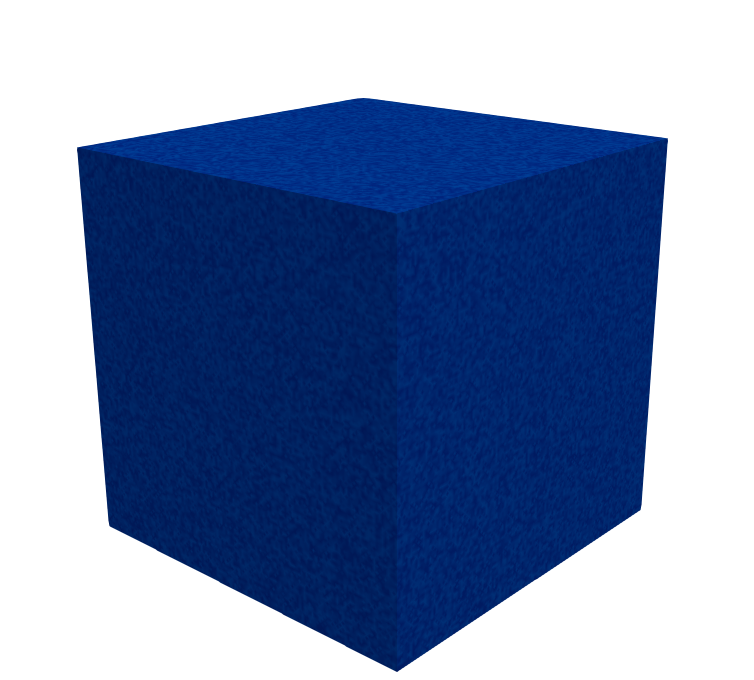}
\includegraphics[angle=-0,width=0.2\textwidth]{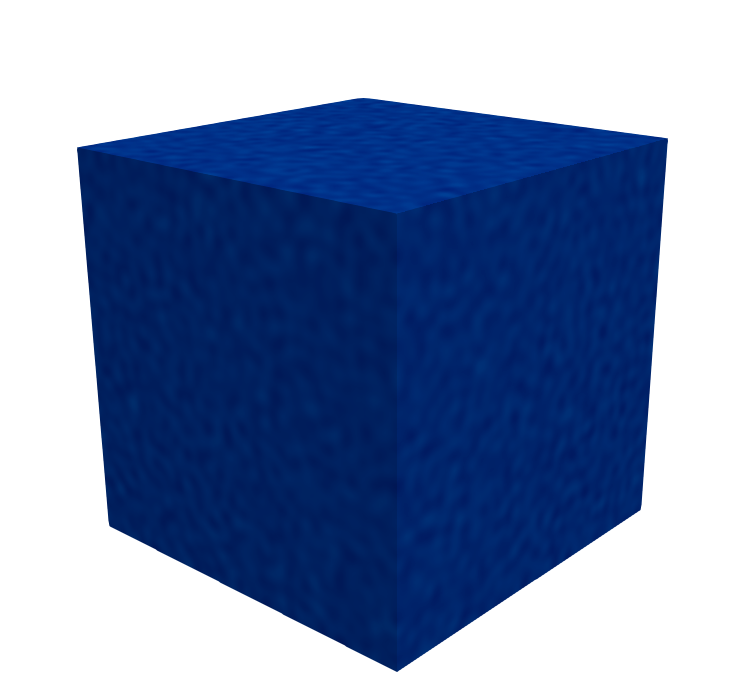}
\includegraphics[angle=-0,width=0.2\textwidth]{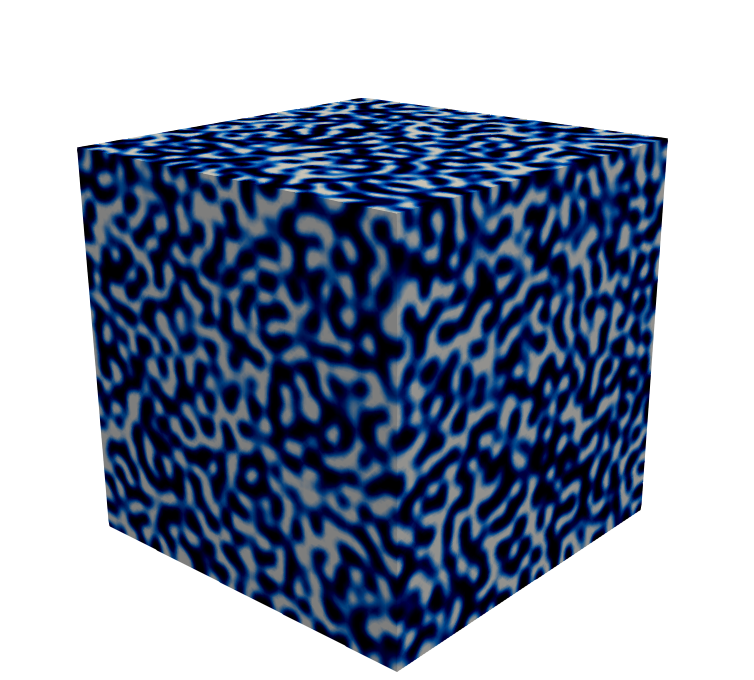}
\includegraphics[angle=-0,width=0.2\textwidth]{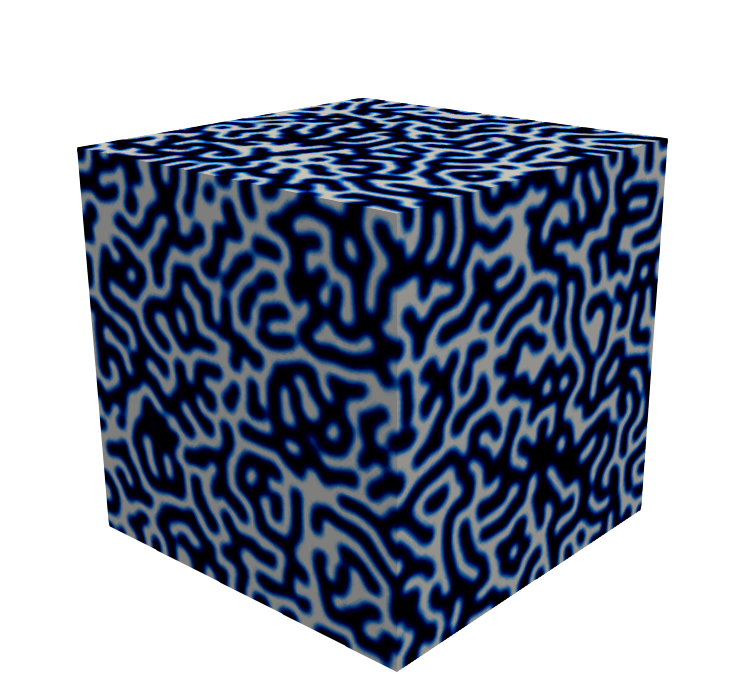}
\includegraphics[angle=-0,width=0.2\textwidth]{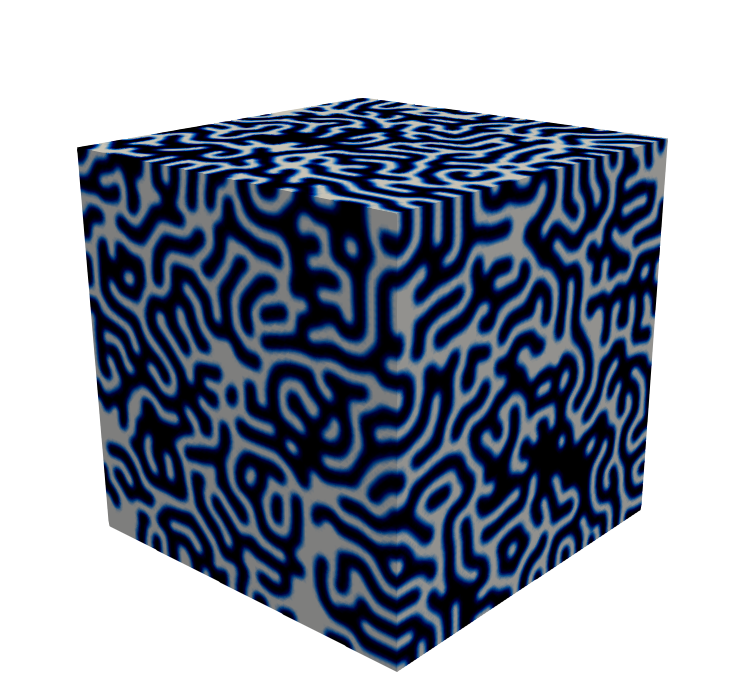}
}
\caption{Computation for SH with $g=0$, $\gamma=1000$.
We display $\psi_h^n$ at times $t=0$, $0.001$, $0.005$, $0.01$, $0.1$.
}
\label{fig:Fig3d_SH}
\end{figure}%
\begin{figure}
\center
\mbox{
\includegraphics[angle=-0,width=0.2\textwidth]{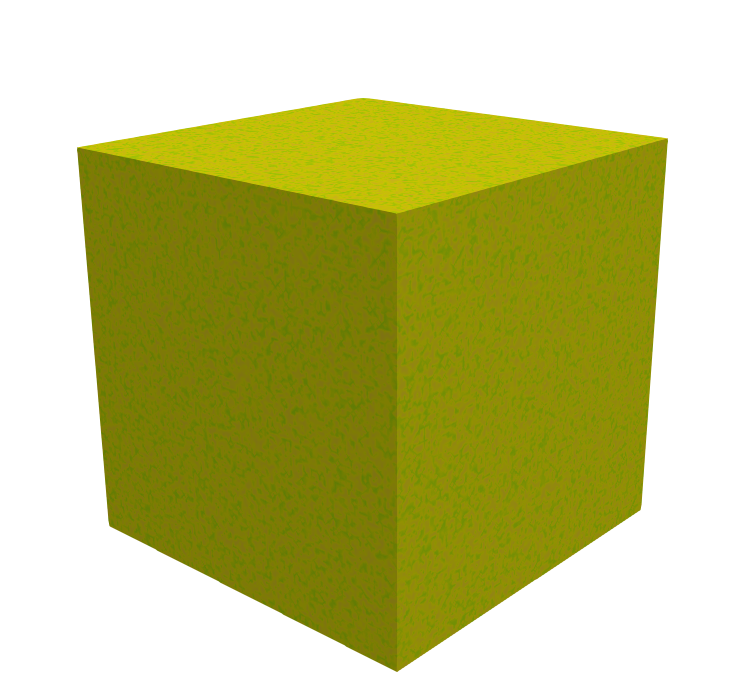}
\includegraphics[angle=-0,width=0.2\textwidth]{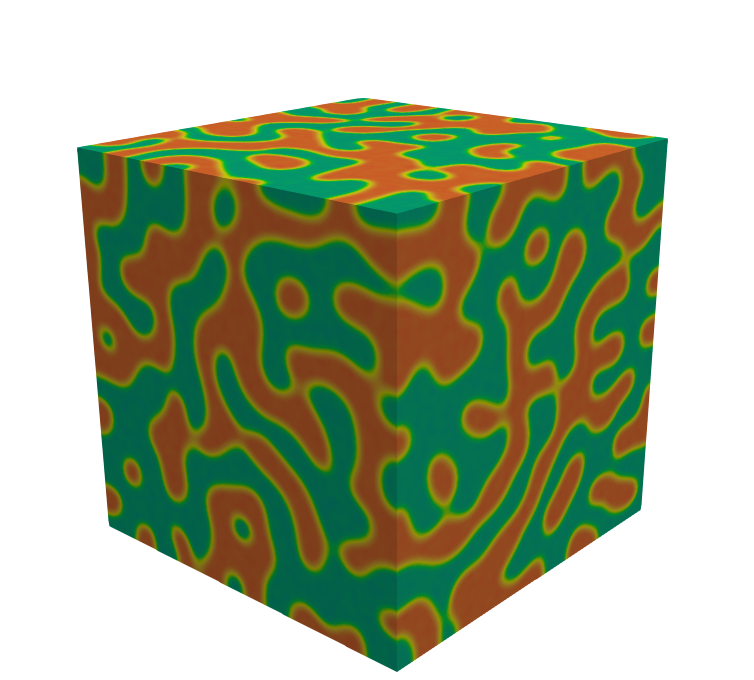}
\includegraphics[angle=-0,width=0.2\textwidth]{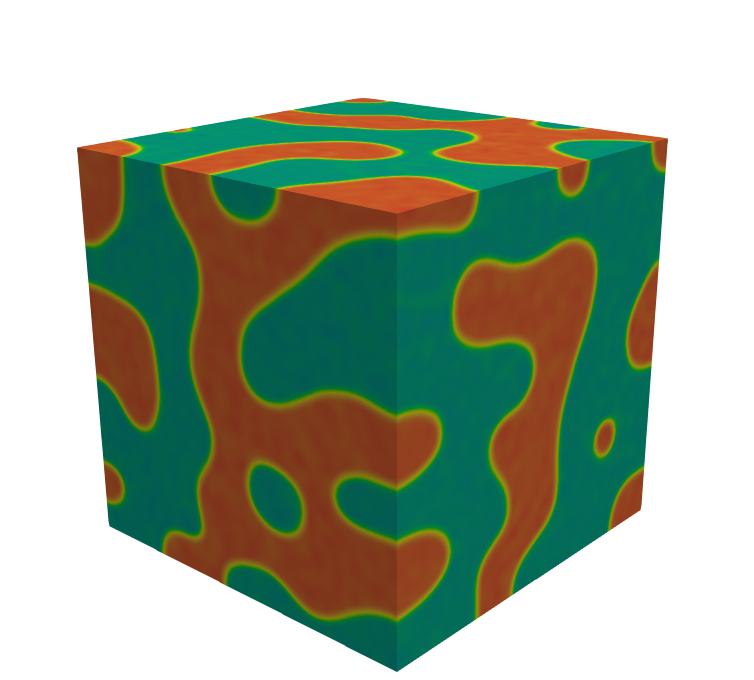}
\includegraphics[angle=-0,width=0.2\textwidth]{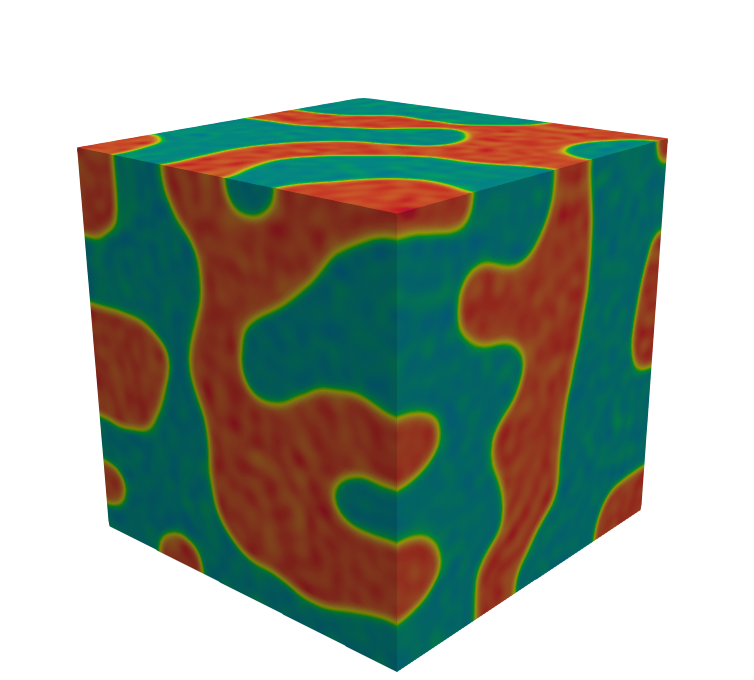}
\includegraphics[angle=-0,width=0.2\textwidth]{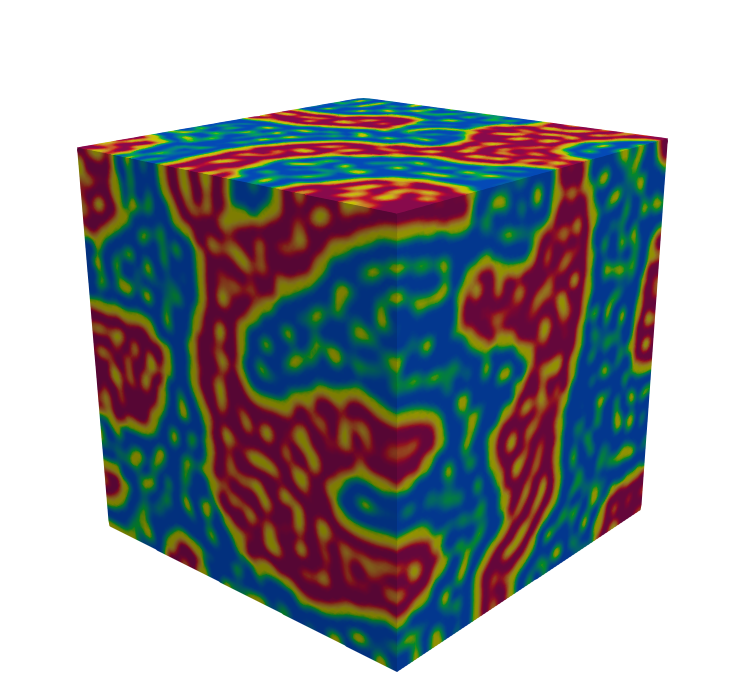}
} \\
\mbox{
\includegraphics[angle=-0,width=0.2\textwidth]{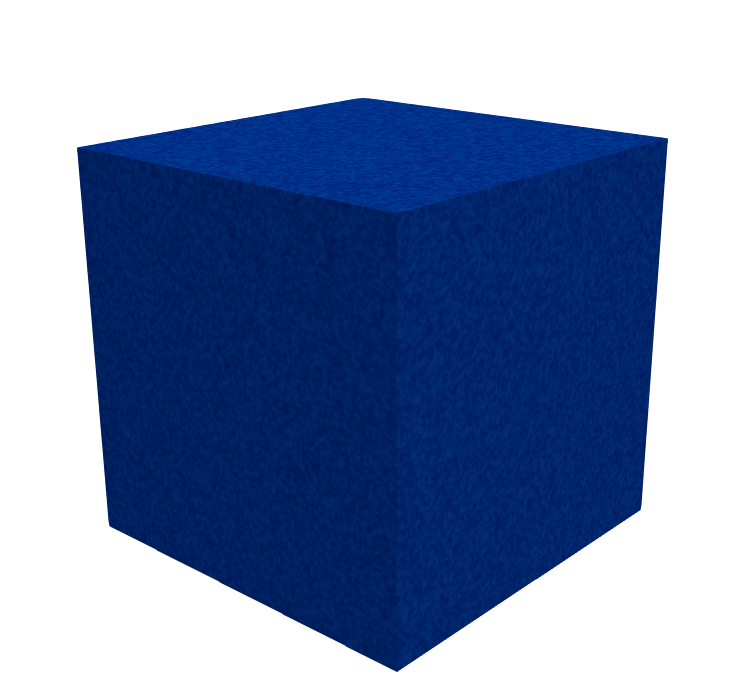}
\includegraphics[angle=-0,width=0.2\textwidth]{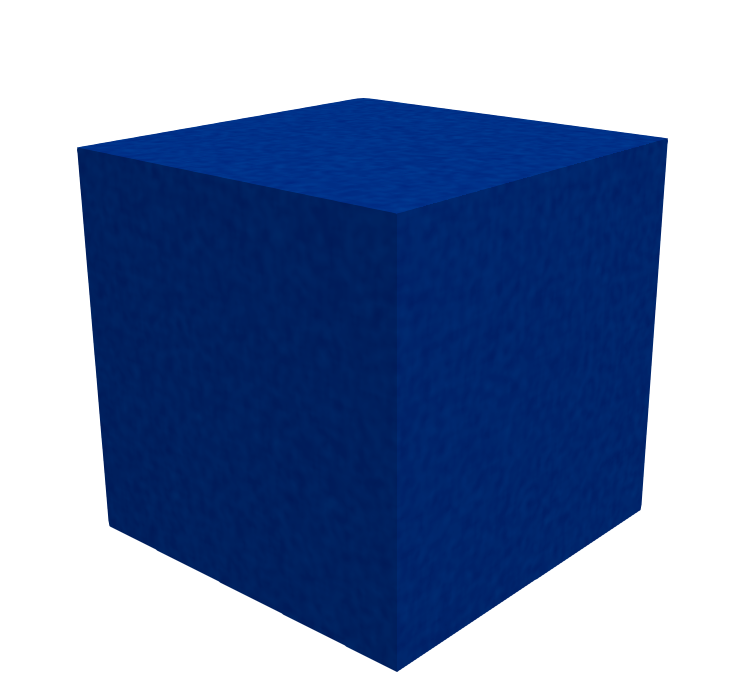}
\includegraphics[angle=-0,width=0.2\textwidth]{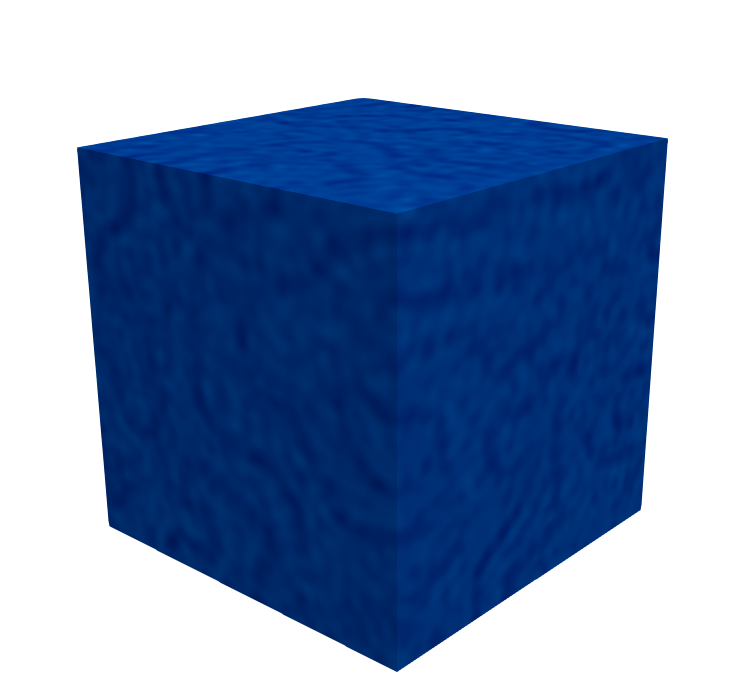}
\includegraphics[angle=-0,width=0.2\textwidth]{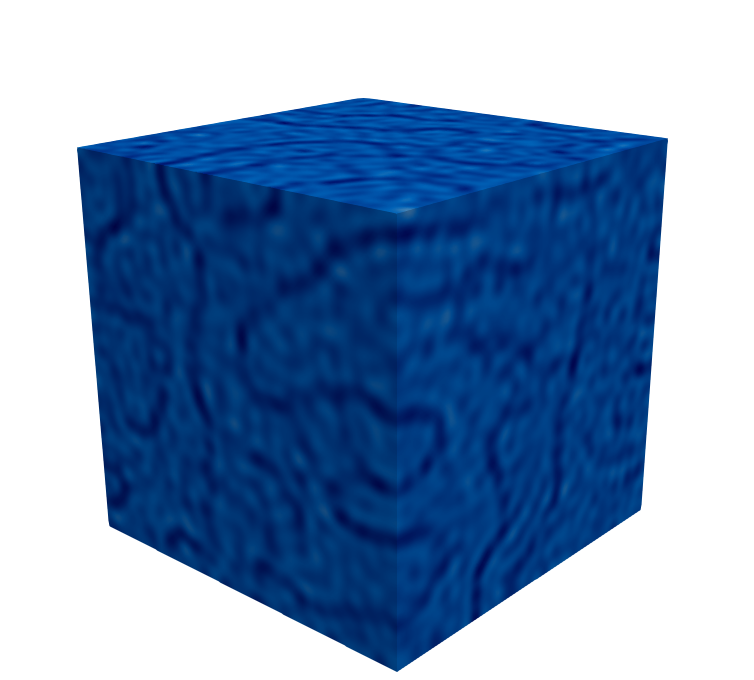}
\includegraphics[angle=-0,width=0.2\textwidth]{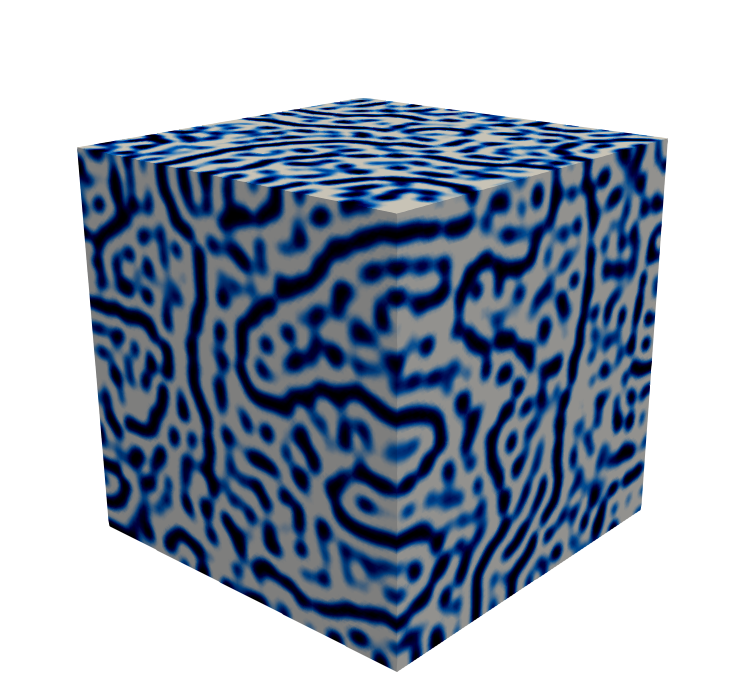}
}
\caption{
Computation for CHSH with $g=0$, $\gamma=1000$.
We display $\phi_h^n$ at times $t=0$, $10^{-4}$, $0.001$, $0.002$, $0.005$.
Below we show $\psi_h^n$ at the same times.
}
\label{fig:Fig3d_delta0}
\end{figure}%



\section*{\bf Acknowledgements}
\noindent AS gratefully acknowledge some support 
from the MIUR-PRIN Grant 2020F3NCPX ``Mathematics for industry 4.0 (Math4I4)'', from ``MUR GRANT Dipartimento di Eccellenza'' 2023-2027 and from the Alexander von Humboldt Foundation.
Additionally, AS acknowledges affiliation with GNAMPA (Gruppo Nazionale per l'Analisi Matematica, la Probabilit\`a  e le loro Applicazioni) of INdAM (Istituto Nazionale di Alta Matematica). KFL gratefully acknowledges the support by the Research Grants Council of the Hong Kong Special Administrative Region, China [Project No.: HKBU 12300321, HKBU 22300522 and HKBU 12302023].

\footnotesize
\bibliographystyle{plain}

\end{document}